\documentclass[11pt,a4paper]{article}
\usepackage[english]{babel}
\usepackage{url}
\usepackage{enumerate}
\usepackage[a4paper,width=160mm,top=25mm,bottom=25mm]{geometry}
\usepackage{amssymb,amsmath,verbatim,graphicx}
\usepackage{graphicx}
\usepackage{xcolor}
\usepackage{paralist}
\usepackage{tikz,graphics}%,eucal}
\usetikzlibrary{fit,positioning,calc}
\usepackage{tcolorbox}
\usetikzlibrary{positioning}
\usepackage{tkz-euclide}
\parskip\medskipamount
\newtheorem{theorem}{Theorem}[section]

\newtheorem{prop}[theorem]{Proposition}
\newtheorem{lemma}[theorem]{Lemma}
\newtheorem{cor}[theorem]{Corollary}

\newtheorem{fact}[theorem]{Fact}
\newtheorem{defi}[theorem]{Definition}
\newtheorem{main}[theorem]{Main Result}
\newtheorem{rem}[theorem]{Remark}

\newenvironment{proof}{\par\noindent\textbf{Proof}\hspace{1em}}{\qed}
\newenvironment{proof*}{\par\noindent\textbf{Proof}\hspace{1em}}{}
\def\<{\langle}
\def\>{\rangle}
\newcommand{\PAG}{\mathbb{P}}

\newcommand{\A}{\mathbb{A}}
\newcommand{\K}{\mathbb{K}}

\renewcommand{\L}{\mathbb{L}}

\newcommand{\cP}{\mathcal{P}}

\newcommand{\cL}{\mathcal{L}}

\newcommand{\CD}{\mathsf{CD}}

\renewcommand{\O}{\mathbb{O}}
\newcommand{\bH}{\mathbb{H}}

\newcommand{\Res}{\mathrm{Res}}
\newcommand{\kar}{\mathrm{char}\,}

\newcommand{\B}{\mathbb{B}}

\newcommand{\bO}{\mathbb{O}}

\newcommand{\bS}{\mathbb{S}}

\def\qed{{\hfill\hphantom{.}\nobreak\hfill$\Box$}}

\usepackage{mathtools}
  \newlength\longest

\begin{document}

  \author{Anneleen De Schepper\thanks{Supported by the Fund for Scientific Research - Flanders (FWO - Vlaanderen)}}
\title{A geometric characterisation of subvarieties of the standard $\mathsf{E}_6$-variety related to the ternions, degenerate split quaternions and sextonions over arbitrary fields}
\date{\footnotesize  Department of Mathematics: Algebra and Geometry\\
Ghent University, Krijgslaan 281-S25, B-9000 Ghent\\
 \texttt{Anneleen.DeSchepper@UGent.be}\\ } 
 \maketitle

\begin{abstract}
The main achievement of this paper is a geometric characterisation of certain subvarieties of the  Cartan variety (the standard projective variety associated to the split exceptional group of Lie type $\mathsf{E}_6$) over an arbitrary field $\K$. The characterised varieties arise as Veronese representations of certain ring projective planes over quadratic subalgebras of the split octonions $\mathbb{O}'$ over $\K$ (among which the sextonions, a 6-dimensional non-associative algebra). We describe how these varieties are linked to the Freudenthal-Tits magic square, and discuss how they would even fit in, when also allowing the sextonions and other ``degenerate composition algebras'' as the algebras used to construct the square.
\end{abstract}

{\footnotesize
}
%\setcounter{tocdepth}{2}

%\tableofcontents

\section{Introduction}
The characterisation which forms the core of this paper could be carried out without knowing the existence of the Freudenthal-Tits magic square (FTMS). However, the latter carries both the idea and motivation for it, in the sense that the characteristic behaviour of the varieties of the FTMS (in particular, its second row) hints at the existence of similarly behaving varieties across the borders of the square (leaving the non-degenerate world). This gives rise to an extended version of square; in particular of the split version of its second row, the varieties of which are exactly the ones we wish to study and characterise. Below, we explain this in more detail.

\subsection{Context: Characterisations related to the FTMS}
 The FTMS is a $4\times4$ array of, depending on the viewpoint, Lie algebras, Dynkin diagrams, buildings, projective varieties.  Our viewpoint will be geometric in the sense of Tits (\cite{Tit:55}), and over an arbitrary field $\K$. The square can be constructed (in various ways) using a pair of composition algebras $(\A_1,\A_2)$ over $\K$.  The algebra $\A_1$ indexes the rows and indicates the  the rank of the varieties in that row; the algebra $\A_2$ indexes the columns and encodes the algebraic structure over which the varieties in that column are defined. In the geometric version of the square that we consider, the algebra $\A_1$ is always split, whereas $\A_2$ can be either division or split, giving rise to two versions of the square (referred to as `non-split' and `split', respectively). Let us illustrate this by zooming in on the second row, which will be most relevant for this paper. 

\begin{compactenum}
\item[$-$] In the non-split version, this row contains projective planes over division composition algebras over $\K$; %, which can be obtained from the split version, by means of Galois descent,  when defining the varieties of the split version over a quadratic Galois extension of $\K$.
\item[$-$] In the split version, the three last entries of this row are the Segre variety $\mathcal{S}_{2,2}(\K)$ (Dynkin type $\mathsf{A}_2 \times \mathsf{A}_2$), the line Grassmannian variety $\mathcal{G}_{6,2}(\K)$ (Dynkin type $\mathsf{A}_5$) and the Cartan variety $\mathcal{E}_{6,1}(\K)$ (Dynkin type $\mathsf{E}_6$) (i.e., the projective version of the well known 27-dimensional module of the (split) exceptional  group of Lie type $\mathsf{E}_6$), respectively. Abstractly, these are ring projective planes over the split composition algebras over $\K$, and the mentioned varieties can be obtained by taking the Veronese representation of these planes (\cite{thirdrow}). The first entry, which coincides with the non-split case,  could thus be seen as the quadric Veronese variety $\mathcal{V}_2(\K)$ (Dynkin type $\mathsf{A}_2$), i.e., the Veronese representation of the projective plane over $\K$. 
\end{compactenum}

The work of J.\ Schillewaert H.\ Van Maldeghem (e.g., \cite{SVM, Kra-Sch-Mal:15})  recently culminated in a common characterisation of the Veronese representations of the varieties of the second row of the FTMS~\cite{conjecture}. These Veronese varieties are point sets in a projective space equipped with a family of quadrics of a certain kind, dependent on the composition algebra. Their characterisation was achieved by means of three simple axioms, and was accomplished among an infinite\footnote{This family is infinite since the only restriction on these non-degenerate quadrics is that subspaces they generate inside  $\mathbb{P}$ should all have the same, yet arbitrary, dimension $d+1$ with $d\in \mathbb{N}$.} family of objects consisting of points and arbitrary (non-degenerate) quadrics in projective space $\mathbb{P}$ over $\K$. The fact that such a general characterisation singles out exactly the varieties of the FTMS demonstrates the latter's special behaviour once more. It is especially remarkable that this can be done for the split and non-split version simultaneously.  The dichotomy of the composition algebras (division/split) translates geometrically in the fact that in the above-mentioned Veronese varieties of the FTMS, the quadrics are  either all  line-free (i.e., of minimal Witt index) or  all  hyperbolic (i.e., of maximal Witt index), respectively.

\subsection{Motivation: Characterisations across the borders of the FTMS}
Inspired by a low dimensional test case elaborated in \cite{Sch-Mal:14},  the author and H.\ Van Maldeghem extended  the above setting  to certain ``degenerate'' composition algebras $\mathbb{B}$ (\cite{ADSHVM}). These algebras $\mathbb{B}$ are setwise given by $\A \oplus t\A$, where $\A$ is an associative division composition algebra over $\K$ and $t$ an indeterminate with $t^2=0$, and satisfy the Cayley-Dickson multiplication formulas (for example, when $\A=\K$, this yields the dual numbers over $\K$). Equivalently, $\B$ is the result of applying the Cayley-Dickson process to $\A$ with $0$ as a primitive element; we will hence refer to $\B$ by $\mathsf{CD}(\A,0)$. Just like the composition algebras, $\mathsf{CD}(\A,0)$ is  quadratic and alternative (since $\A$ is associative), and its norm form $a+tb \mapsto \mathsf{N}(a)$ (where $\mathsf{N}$ is the norm form of $\A$) is multiplicative, though degenerate. When taken to the above setting, where the quadrics are determined by the norm form, this translates geometrically to projective varieties equipped with \emph{degenerate} quadrics  whose base is a line-free quadric. Using similar axioms as in  \cite{conjecture}, it was shown  in \cite{ADSHVM}  that point-sets equipped with such quadrics (a priori inside arbitrary dimensions)  arise from the Veronese representation of a  projective Hjelmslev plane defined over an algebra $\mathsf{CD}(\A,0)$ with $\A$ an associative division composition algebra. Moreover, the resulting varieties are related to the affine buildings of absolute type $\tilde{\mathsf{A}}_2$, $\tilde{\mathsf{A}}_5$ and $\tilde{\mathsf{E}}_6$ in the sense that they can be seen as the radius 2 sphere of a special vertex of such a building. The current paper investigates the following question:

\textit{What happens for the algebras setwise given by $\A \oplus t\A$, where $\A$ is an associative composition algebra over $\K$ which is \emph{not} division?}

A first essential difference with the former case is that we should not only consider the algebras $\mathsf{CD}(\A,0)$. Indeed, also the \emph{ternions} $\mathbb{T}$, a non-commutative 3-dimensional subalgebra of the split quaternions $\bH'$,  and \emph{sextonions} $\mathbb{S}$, a strictly alternative 6-dimensional subalgebra of the split octonions $\bO'$, can be written as $\L' \oplus t\L'$ (where $\L' = \K \times \K$) and $\bH'  \oplus t\bH'$, respectively (cf.\ Proposition~\ref{qadd}). The way we convey it, this gives rise to an additional layer for the FTMS. Denoting the division and split composition algebras of dimensions 2,4,8 over $\K$ by $\L$ and $\L'$, $\bH$ and $\bH'$ and $\bO$ and $\bO'$, respectively,  we see this second layer of the second row as ring projective planes over the following algebras (giving both the non-split and the split case):
\[\begin{array}{|c|c|c|c|c|c|c|} \hline
 & 1 	& 2 	& 3 	& 4 	& 6 	&8 \\ \hline
\text{non-split} &/	& \mathsf{CD}(\K,0)	&	/&\mathsf{CD}(\L,0)	 &/	& \mathsf{CD}(\bH,0)	 \\\hline
\text{split} &/	& \mathsf{CD}(\K,0)		& \text{ternions } \mathbb{T}	& \mathsf{CD}(\L',0)	 & \text{sextonions } \bS & \mathsf{CD}(\bH',0)	 \\\hline
\end{array}\]

 A second difference is that it turns out that the Veronese variety associated to the octonion algebra $\mathsf{CD}(\bH',0)$, where $\bH'$ are the split quaternions, behaves differently compared to the ones associated to the other algebras in that series. It does not fit in, not in any natural way. This manifests itself in some sense in the fact that, opposed to the Veronese varieties associated to the other algebras in that row, it cannot be seen as a subvariety of the Cartan variety $\mathcal{E}_{6,1}(\K)$. 

The link between the FTMS and the sextonions  $\mathbb{S}$ was already explored in \cite{Wes:06} by Westbury, who suggested to extend the FTMS (which he considers as a square  of complex semisimple Lie algebras) by adding a row/column between the third and the fourth one. Around the same time, also Landsberg and Manivel considered this intermediate Lie algebra between $\mathfrak{e}_7$ and $\mathfrak{e}_8$ in \cite{Lan-Man:06}. In Section 8 of that paper, they in particular study some  (algebraic) geometric properties of the sextonionic plane, i.e., a Veronese variety associated to $\mathbb{S}$.
Our approach on the other hand starts from the (incidence) geometric properties of this Veronese variety and its smaller siblings related to $\mathbb{T}$ and $\mathsf{CD}(\L',0)$ (the one related to $\mathsf{CD}(\K,0)$ is viewed as a part of the non-split case). By means of three simple axioms, we then characterise these varieties. We also provide two additional ways of viewing them: on the one hand, constructed from two (dual) representations of the non-degenerate varieties they are composed of (involving $\mathcal{S}_{2,2}(\K)$ and $\mathcal{G}_{6,2}(\K)$) (cf. Section~\ref{DLG}), and on the other hand, as subvarieties of the $26$-dimensional projective $\mathcal{E}_{6,1}(\K)$-variety, obtained by slicing it with certain subspaces of dimension $11, 14$ or $20$ (cf. Section~\ref{splitgeom}). This hence also gives us additional insight in the geometric structure of  $\mathcal{E}_{6,1}(\K)$.

\subsection{Main result: Characterisation of the Veronese varieties related to the ``new'' split second row of the FTMS}
 Stating Main Result~\ref{main} requires more notation and a slightly technical set-up,  so we refer to Section~\ref{mainsub} for that. For the purpose of this introduction, we prefer a simplified set-up, by which means we can explain a related characterisation, proved in \cite{SVM}. With this, we cannot only informally situate the current main result, but also  point out   similarities and differences compared to other results.

Consider a set of points $X$ in a projective space $\mathbb{P}^N(\K)$, with $N \in \mathbb{N}\cup \{\infty\}$, equipped with a family $\Xi$ of subspaces of $\mathbb{P}^N(\K)$ (of arbitrary yet fixed dimension $d+1< \infty$), $|\Xi|\geq 2$, such that, for each  $\xi\in\Xi$, the intersection $X(\xi):=X\cap \xi$ is a parabolic or hyperbolic quadric generating $\xi$ (i.e., the maximal isotropic subspaces on $X(\xi):=X \cap \xi$ have projective dimension $\lfloor \frac{d}{2} \rfloor$). Then the pair $(X,\Xi)$ is called a \emph{split Veronese set (of type $d$)} if the following axioms are satisfied:

\begin{compactenum}
\item[(SV1)] Each pair of distinct points $p_1,p_2 \in X$ is contained in a member of $\Xi$;
\item[(SV2)] If $\xi_1,\xi_2$ are distinct members of $\Xi$, then $\xi_1 \cap \xi_2 \subseteq X$
\item[(SV3)] for each point $x \in X$,  then there are $\xi_1,\xi_2\in \Xi$ containing $x$ such that the subspace $T_x$ generated by all $d$-spaces $T_x(X(\xi))$ is generated by $T_x(X(\xi_1))$ and $T_x(X(\xi_2))$. \end{compactenum}

Suppose that  $(X,\Xi)$ is a Veronese set of split type $d$ in $\mathbb{P}^N(\K)$. The result obtained in~\cite{SVM} says that $d \in \{1,2,4,8\}$ and $N \leq 3d+2$, and moreover, if $N =3d+2$, then  $(X,\Xi)$ is projectively unique and the resulting varieties are exactly the (Veronese) varieties of the split version of the second row of the FTMS, mentioned earlier. If $N < 3d+2$ then $(X,\Xi)$ is either $\mathcal{S}_{1,2}(\K)$ (a subvariety of $\mathcal{S}_{2,2}(\K)$) or $\mathcal{G}_{5,2}(\K)$ (a subvariety of $\mathcal{G}_{6,2}(\K)$).

A natural way to define  Veronese varieties related to $\mathbb{T}$, $\mathsf{CD}(\L',0)$, $\mathbb{S}$ and $\mathsf{CD}(\bH',0)$ is by giving an affine description (see Section~\ref{splitgeom}). A study of these reveals that, except for the variety associated to $\mathsf{CD}(\bH',0)$, they also come with a  set of points and quadrics  ``more or less'' satisfying axioms (SV1), (SV2) and (SV3) above. It was not obvious to see why this was ``more or less'', but it turns out that, this time, the structure is not homogenous: there are two types of points and two types of quadrics. Taking this into account when rephrasing axioms (SV1) up to (SV3), the resulting axioms are nothing but a natural extension of them (cf.\ Section~\ref{Defvv}).

It remains slightly mysterious why the variety related to $\mathsf{CD}(\bH',0)$ does not satisfy these axioms, not even ``more or less''. As hinted at above, it stands out from the other varieties in that it is not a subvariety of $\mathcal{E}_{6,1}(\K)$, which is explained by the fact that all other algebras under consideration are subalgebras of the split octonions $\bO'$, whereas $\mathsf{CD}(\bH',0)$  is not.   Axioms (SV1) and (SV2) (and also their natural extensions) imply that the convex closure of two points of $X$ should form a quadric (corresponding to a member of $\Xi$). However, for the variety related to  $\mathsf{CD}(\bH',0)$, the convex closure of two points is not a quadric anymore, it rather is a bunch of quadrics. One could argue that it is not a surprise that octonions come with different behaviour, though we did not anticipate this. Indeed, in the non-split case (treated in~\cite{ADSHVM}), the Veronese variety related to the octonion algebra $\mathsf{CD}(\bH,0)$ behaves as do the Veronese varieties related to $\mathsf{CD}(\L,0)$ and $\mathsf{CD}(\K,0)$ (with notation as above). 

If we denote by $X$ and $Z$ the two types of point sets, and by $\Xi$ and $\Theta$ the two types of subspaces intersecting $X \cup Z$ in certain quadrics, and call $(X,Z,\Xi,\Theta)$ a \emph{dual split Veronese set} if it satisfies the axioms extending (SV1), (SV2) and (SV3) as explained above, then informally the main result reads as follows (where we only exclude the field with two elements, see~Remark~\ref{f2}).

\textbf{Main Result--informal statement}~\textit{If $(X,Z,\Xi,\Theta)$ is a dual split Veronese set in $\mathbb{P}^N(\K)$, where $\K$ is an arbitrary field with $|\K|>2$, then, up to projectivity and up to projection from a subspace contained in each member of $\Xi\cup \Theta$,  either $\Theta$ is empty and then $(X,\Xi)$ is a split Veronese set, or $\Theta$ is non-empty and there are four possibilities,  all of which are subvarieties of $\mathcal{E}_{6,1}(\K)$;  three of them can be obtained as a Veronese variety associated to one of  $\mathbb{T}, \mathsf{CD}(\L',0), \bS$,  the fourth and smallest case is a subvariety of the Veronese variety associated to $\mathbb{T}$.}

\subsection{Structure of the paper}

In \textbf{Section~\ref{VV}} the ``degenerate composition algebras'' are formally introduced  and discussed to the extent that we will need them. Afterwards, the Veronese varieties associated to  $\mathbb{T}$, $\mathsf{CD}(\L',0)$, $\bS$ and $\mathsf{CD}(\bH',0)$ are defined and studied briefly. In Proposition~\ref{standardex}, we show that,  apart from the last one, they all satisfies properties that naturally generalise (SV1), (SV2) and (SV3) above.  The study of these varieties can be found in more detail in the author's Ph.D.\ thesis (\cite{thesis}), as now we focus on their description rather than on the calculations leading to them.

In \textbf{Section~\ref{sVV}}, the axiomatic set-up for dual split Veronese sets is given, as is a purely geometric description of certain families of varieties that we will encounter later on, each of them containing examples of dual split Veronese sets. This geometric description does not rely on an underlying algebraic structure, but is of course in accordance with the coordinate description given in Section~\ref{VV}. In \textbf{Section~\ref{mainsub}} we state the formal version of our main result. The remainder of the paper (\textbf{Sections~\ref{bp}~to~\ref{r>1}}) deals with its proof, the structure of which is explained in \textbf{Section~\ref{sotp}}.

\section{``Degenerate composition algebras" and associated Veronese varieties}\label{VV}

Henceforth, let $\K$ be an arbitrary field. Seeing that composition algebras and (non-degenerate) quadratic alternative algebras are equivalent notions, we will use the latter setting to incorporate the degenerate case. In the literature, one does not find too much on quadratic alternative algebras  with a possibly non-trivial radical \'and without restriction on the characteristic of $\K$. %What follows is a brief version of Chapter 4 of the author's Ph.D.\ thesis~\cite{thesis}. 
What follows is a combination of elements from~\cite{McC:04,McC'} (allowing characteristic 2, but restricting to the non-degenerate case) and~\cite{Kun-Sch:86} (where the possibilities for the radical are examined, excluding characteristic 2). 

\subsection{Quadratic alternative algebras and their radical}
Let $\A$ be a unital quadratic alternative  $\K$-algebra, i.e., the associator $[a,b,c]:=(ab)c-a(bc)$ yields a trilinear alternating map and each $a\in \A$ satisfies a quadratic equation $x^2-\mathsf{T}(a)x+\mathsf{N}(a)=0$, where the \emph{trace} $\mathsf{T}: \A \rightarrow \K: a \mapsto \mathsf{T}(a)$ is a linear map with $\mathsf{T}(1)=2$ and the \emph{norm} $\mathsf{N}: \A \rightarrow \K: a \mapsto \mathsf{N}(a)$ is a quadratic map with $\mathsf{N}(1)=1$.
The canonical involution associated to $\A$ is given by the map $\A \rightarrow \A: x \mapsto \overline{x}:=\mathsf{T}(x)-x$, which is indeed an involutive anti-automorphism ($\overline{xy}=\overline{y}\,\overline{x}$ for all $x,y\in \A$),  fixing $\K$. Note that $\mathsf{N}(a)=a\overline{a}$ for each $a\in \A$.

 The bilinear form $f$ associated to the quadratic form $\mathsf{N}$ is given by $f(x,y)=\mathsf{N}(x+y)-\mathsf{N}(x)-\mathsf{N}(y)=x\overline{y}+y\overline{x}$. Its \emph{radical} is the set  $\mathsf{rad}(f)=\{x \in \A \mid f(x,y)=0 \; \forall y \in \A\}$.  
 %We call $\A$ \emph{non-singular} if $f$ is non-singular, i.e., if $\mathsf{rad}(f)=\{0\}$. 
 We call $\A$  \emph{non-degenerate} if its norm form $\mathsf{N}$ is non-degenerate, i.e., if $\mathsf{N}$ is anisotropic on $\mathsf{rad}(f)$, so if  $\{r \in \mathsf{rad}(f) \mid \mathsf{N}(r)=0\}$ is trivial. We call the latter set the \emph{radical} $R$ of $\A$. One could also describe $R$ as  the \emph{nil radical of $\A$}, which is the maximal ideal of $\A$ with the property that each of its elements is nilpotent. 
 %Note that $\A$ can only be singular but non-degenerate if $\kar(\K)=2$, for otherwise $\mathsf{rad}(f)=R$.
 
Our interest goes out to the quadratic alternative algebras $\A$ for which $R$ is, as a ring, generated by a single element. Since $R$ is a 2-sided ideal of $\A$ and $\A(\A r)=\A r=r\A=(r\A)\A$ for each $r\in R$ (even for each $r\in \mathsf{rad}(f)$), this is equivalent to requiring that $R$ is a principal ideal of $\A$.

The non-degenerate quadratic alternative $\K$-algebras $\A$ with $\dim_\K(\A)<\infty$ can all be produced using the Cayley-Dickson doubling process. Below, we give an extended version of this process, extended in the sense that it also produces \emph{degenerate} algebras, i.e., with a non-trivial radical~$R$.

\subsection{The (extended) Cayley-Dickson doubling process}\label{ecd}

%Let $\A$ be a quadratic alternative $\K$-algebra with associated involution $x\mapsto \overline{x}$ as before, and 
Let $\zeta$ be any element in $\K$.
One application of the \textit{Cayley-Dickson doubling process} on the algebra $\A$ using $\zeta$ as a \textit{primitive element} results in a $\K$-algebra which setwise equals $\A \times \A$,  addition is defined componentwise  and  multiplication is given by \[(a,b)\times(c,d)=(ac+ \zeta d\overline{b}, \overline{a}d +cb).\] This resulting $\K$-algebra is denoted by $\mathsf{CD}(\A,\zeta)$ and  is quadratic too. Its associated involution and norm form are given by $\overline{(a,b)}=(\overline{a},-b)$ and   $\mathsf{N}(a,b)=(a,b)\cdot \overline{(a,b)}= (\mathsf{N}(a)-\zeta \mathsf{N}(b),0)$, respectively. The fact that we allow the primitive element $\zeta$ to be $0$ is the point at which the above process extends the standard one, and with this option, $N$ will be degenerate.

%To simplify calculating, we will make us of the isomorphism $(a,b) \mapsto a+tb$ between $\mathsf{CD}(\A,\zeta)$ and $\A \oplus t\A:=\{a+tb \mid a,b \in \A\}$,  where $t$ is an indeterminate with $t^2 =\zeta$, and in which the multiplication and involution are such that $(a,b)\mapsto a+tb$ preserves both the multiplication and the involution, which comes down to the following rules for all $a,b,c,d \in \A$:
%\begin{equation}\label{mult}at=t\overline{a}, \quad a(td)=t(\overline{a}d), \quad (tb)c=t(cb), \quad (tb)(td)=t^2(d\overline{b}).\end{equation}
 We list some well-known features of  $\mathsf{CD}(\A,\zeta)$ in terms of $\A$ and $\zeta$:
\begin{fact}\label{struct}
The algebra $\mathsf{CD}(\A,\zeta)$ is 
\begin{compactenum}
\item[$(i)$] a division algebra $\Leftrightarrow$ $\zeta \notin \mathsf{N}(\A)$ and $\A$ is division;
\item[$(ii)$] non-degenerate $\Leftrightarrow$  $\zeta \neq 0$ and $\A$ is non-degenerate;
%\item[$(iii)$] singular but non-degenerate $\Leftrightarrow$  $\kar{K}=2$, $\A$ is non-degenerate and $\zeta \notin \A^2$;
\item[$(iii)$] \emph{commutative} $\Leftrightarrow$ $\A$ is commutative and $\overline{a}=a$ for each $a \in \A$;
\item[$(iv)$] \emph{associative}  $\Leftrightarrow$  $\A$ is commutative and associative;
\item[$(v)$] \emph{alternative} $\Leftrightarrow$ $\A$ is associative.
\end{compactenum}
\end{fact}

When $\A=\K$, the induced involution in $\mathsf{CD}(\K,\zeta)$ is given by $(a,b) \mapsto (a,-b)$. This involution is non-trivial precisely if $\kar(\K)\neq 2$. Hence, Fact~\ref{struct}$(iii)$-$(v)$ implies that successive applications of the Cayley-Dickson doubling process on $\K$,  if $\kar(\K) \neq 2$, eventually lead to a strictly alternative (i.e., non-associative) algebra; and, if $\kar(\K) = 2$, only produce commutative associative algebras with a trivial involution. If $\kar(\K) = 2$, then the first step of the process is replaces by considering the quotient of $\K[x]$ by the ideal $(x^2+x+\zeta)$ for some $\zeta \in \K$, which has a non-trivial involution and after which we can again apply the usual process. 
%The multiplication and (non-trivial) involution go as follows:
%\[ (a,b)\cdot(c,d) =(ac+ \zeta d\overline{b}, \overline{a}d +cb+d\overline{b}), \quad \overline{(a,b)}=(a+b,b). \]
%Also here, $(a,b) \mapsto a+tb$  yields an isomorphism between $\mathsf{CD}(\A,\zeta)$ and $\A \oplus t\A$, where $t$ now satisfies $t^2=t+\zeta$ (note that $\overline{t}=t+1$) and with multiplication formulas as in~(\ref{mult}). Further steps of the process can then be carried out as usual.

\subsection{Non-degenerate split quadratic alternative algebras}\label{sqaa}

Let $\A$ be a non-degenerate quadratic alternative algebra. It is a well-known fact that its norm form $\mathsf{N}$ is either anisotropic on $\A$ or hyperbolic (i.e., has maximal Witt index). In the former case, $\A$ is a division algebra, since $x \in \A$ is invertible if and only if $\mathsf{N}(x) \neq 0$,

\begin{defi} If $\mathsf{N}$ is not anisotropic, then $\A$  is called \emph{split}. \end{defi}

Since the norm form of $\A$ completely determines $\A$ (two non-degenerate quadratic algebras are isomorphic if and only if their respective norm forms are equivalent quadratic forms), and since any two hyperbolic quadratic forms in the same (even) dimension are equivalent, we have that all non-degenerate split quadratic alternative algebras over $\K$ with the same dimension over $\K$ are isomorphic. This allows us to speak of  \emph{the}  non-degenerate split quadratic alternative algebras over $\K$, which we will refer to as $\K$, $\L'$, $\bH'$ and $\bO'$. They can be described as follows (independently of the characteristic),  (see for instance \cite{Kap:53}):

 \begin{fact}\label{qad} Let $\A$ be a non-degenerate split quadratic alternative algebra over a field $\K$. Then $\A$ is isomorphic to either $\K$, $\K \times \K$, the $2 \times 2$-matrices $\mathcal{M}_{2}(\K)$ over $\K$ or the split octonions $\mathsf{CD}(\mathcal{M}_{2}(\K),1)$. 
\end{fact}

The split octonions, being non-associative, cannot be given by ordinary matrices and their ordinary multiplication. Zorn's vector-matrices however are a special way of writing the split octonions as $2\times 2$-matrices, the off-diagonal elements of which are vectors:

\[\bO' \cong \left\{\begin{pmatrix} a & \left[\begin{array}{ccc} b & x & y\end{array}\right] \\
\left[\begin{array}{c} c\\ z\\ u\end{array}\right] & d \end{pmatrix}:a,b,c,d,x,y,u,z\in\K\right\},\]

and the multiplication is given as follows (with the usual dot product and vector product):

\[\begin{pmatrix} a & \boldsymbol{v} \\ \boldsymbol{w} & d \end{pmatrix} \begin{pmatrix} a' & \boldsymbol{v'} \\ \boldsymbol{w'} & d' \end{pmatrix} = 
\begin{pmatrix} aa' +\boldsymbol{v}\cdot\boldsymbol{w'} & a\boldsymbol{v'}+d'\boldsymbol{v} + \boldsymbol{w} \times \boldsymbol{w'} \\ a'\boldsymbol{w}+d\boldsymbol{w'} - \boldsymbol{v}\times\boldsymbol{v'} & dd' + \boldsymbol{v'}\cdot\boldsymbol{w} \end{pmatrix},
 \]
and 
\[\begin{pmatrix} a & \boldsymbol{v} \\ \boldsymbol{w} & d \end{pmatrix} 
\overline{\begin{pmatrix} a' & \boldsymbol{v'} \\ \boldsymbol{w'} & d' \end{pmatrix}} = \begin{pmatrix} a & \boldsymbol{v} \\ \boldsymbol{w} & d \end{pmatrix} 
\begin{pmatrix} d & -\boldsymbol{v} \\ -\boldsymbol{w} & a\end{pmatrix} =  (ad-\boldsymbol{v}\cdot\boldsymbol{v}) \, \mathsf{I}_2.
 \]

For short, we denote the matrices in the above set by $M(a,b,c,d,x,y,z,u)$.

\subsection{Split quadratic alternative algebras with a level 1 degeneracy}

Let $\A$ be a degenerate quadratic alternative unital $\K$-algebra with a non-trivial radical~$R$. Then one can show that $\A$ contains a non-degenerate quadratic associative unital algebra $\B$ such that $\A=\B \oplus R$. For the proof of the following proposition we refer to Section 4.4.1 of~\cite{thesis}, adding that it is based methods occurring in \cite{McC'} to classify the composition algebras (non-degenerate by definition).

\begin{prop}\label{qadd}
Let $\A$ be a degenerate quadratic alternative $\K$-algebra whose radical $R$ is generated (as a ring) by a single element $t\in \A\setminus\{0\}$. Then $\A$ has a non-degenerate quadratic associative unital algebra $\B$ such that $\A=\B \oplus t\B$. Moreover, if $\B$ is split, then either $\A$ is isomorphic to $\CD(\B,0)$ where $\B \in \{\K, \L', \bH' \}$, or $\dim_\K(\A)\in\{3,6\}$. In the latter case, $\A$ is isomorphic to the following respective quotients of $\CD(\B,0)$:
\begin{compactenum}
\item[\emph{(a)}] the upper triangular $2 \times 2$-matrices over $\K$ (the \emph{ternions}  $\mathbb{T}$);
\item[\emph{(b)}] $\{M(a,b,c,d,0,y,z,0) \mid a,b,c,d,y,z \in \K\}$ (the \emph{sextonions} $\mathbb{S}$);
\end{compactenum}
If $\B$ is split and $\dim_\K(\A) < 8$, then $\A$ is isomorphic to a subalgebra of the split octonions~$\bO'$.
\end{prop}

\begin{rem}\em The fact that possibilities $(a)$ and $(b)$ occur is because $b \mapsto tb$ is a not necessarily injective linear map between $\B$ and $t\B$. Note that, if $tb=0$ for some  $b \in \B\setminus \{0\}$, then $\mathsf{N}(b) = 0$ (hence this does not occur for $b\neq 0$ when $\B$ is division). \end{rem}

 We will refer to the algebras $\A=\B\oplus t\B$ of the above proposition as \emph{``split quadratic alternative algebras with a level 1 degeneracy''}.

\subsection{Veronese varieties associated to $\mathbb{T}$, $\mathbb{H}'$, $\mathbb{S}$ and $\mathbb{O}'$}\label{splitgeom}

Let $\A$ be a split quadratic alternative $\K$-algebra with a level 1 degeneracy, i.e., an algebra as in Proposition~\ref{qadd}:  $\mathbb{T}$, $\CD(\L',0)$, $\mathbb{S}$ or $\CD(\bH',0)$. To each of those, we associate a plane Veronese variety, i.e., we consider the Veronese representation of an abstract ring projective plane over $\A$. The points and lines of the latter geometry are given by the triples in $\A$ such that there is a left resp.\ right $\A$-linear combination that gives 1.  However, since $\A$ contains many non-invertible elements, it is hard to list these triples. Instead, we start with an affine part of the abstract geometry and use the following \emph{partial Veronese map} $\rho$, with $d:=\dim_\K(\A)$.

\[\rho:\A\times\A \rightarrow \mathbb{P}^{3d+2}(\K):(B,C)\mapsto (1,B\overline{B},C\overline{C},BC, C,B)\]

\begin{rem} \em 
 Usually, the Veronese map takes $(y,z)$ to $(1,y\overline{y},z\overline{z},y\overline{z}, z,\overline{y})$, but we can change $z$ to $\overline{z}$, and then obtain $(1,y\overline{y},z\overline{z},yz,\overline{z},\overline{y})$, which linearly transforms into the above definition. 
\end{rem}

 If $|\K|>2$, a calculation shows that a line $L$ of $\mathbb{P}^{3d+2}(\K)$ containing three points of $\rho(\A\times \A)$ has all its points in $\rho(\A\times \A)$, except for the unique point on $L$ in the hyperplane $H_0$ given by the equation $X_0=0$. So, as a first step, we add the points $L\cap H_0$ for such lines $L$. Repeated steps of this process (in fact, one could show that two steps suffice) yield a point-set which is \emph{projectively closed}: each line of $\mathbb{P}^{3d+2}(\K)$ is either contained in it, or meets it in at most two points. If $|\K|=2$, one can also define this closure, but we do not do this effort as $\mathbb{F}_2$ is the only field we will not consider.

\begin{defi}\label{vv} \em For $|\K|>2$, we define the Veronese variety $\mathcal{V}_2(\K,\A)$ as the \emph{projective closure} of $\rho(\A \times \A)$.
\end{defi}

%\begin{rem}\label{rhoenV} \em Note that all points of  $\mathcal{V}_2(\K,\A)\setminus \rho(\A \times \A)$ belong to $H_0$, and hence each line of $\mathcal{V}_2(\K,\A)$ through a point of $\rho(\A \times \A)$ has all its points but one in  $\rho(\A \times \A)$.\end{rem}

\begin{prop}[\cite{thirdrow}]\label{thegeom} The geometries $\mathcal{V}_2(\K,\L')$,  $\mathcal{V}_2(\K,\bH')$, $\mathcal{V}_2(\K,\bO')$ are isomorphic to the Segre variety $S_{2,2}(\K)$, the line Grassmannian $\mathcal{G}_{6,2}(\K)$, $\mathcal{E}_{6,1}(\K)$, respectively. \end{prop}

We  briefly discuss the geometries $\mathcal{V}_2(\K,\A)$ for each $\A \in \{\mathbb{T}, \CD(\L',0), \mathbb{S},\CD(\bH',0)\}$. 

\subsubsection{The case $\A=\mathbb{S}$} We use $\bO' \cong \{M(a,b,c,d,x,y,z,u)\mid a,b,c,d,x,y,z,u \in \K\}$ and $\mathbb{S} \cong \{M(a,b,c,d,0,y,z,0)\mid a,b,c,d,y,z \in \K\}$.  Moreover, note that $\mathbb{S}':= M(a,b,c,d,x,0,0,u)\mid a,b,c,d,x,u \in \K\} \cong \mathbb{S}$ and $\mathbb{H}' \cong  \{M(a,b,c,d,0,0,0,0)\mid a,b,c,d \in \K\}$, so we may asume that $\mathbb{H}' = \mathbb{S} \cap \mathbb{S}' \subseteq \mathbb{O}'$.
Clearly, if $\A' \subseteq \A$ then $\rho(\A' \times \A') \subseteq \rho(\A \times \A)$ and hence also $\mathcal{V}_2(\K,\A') \subseteq \mathcal{V}_2(\K,\A)$. So $\mathcal{V}_2(\K,\bH')$ is a subgeometry of both $\mathcal{V}_2(\K,\mathbb{S})$ and $\mathcal{V}_2(\K,\mathbb{S}')$, and the latter two are subgeometries of $\mathcal{V}_2(\K,\mathbb{O}')$, which is isomorphic to $\mathcal{E}_6(\K)$ by Proposition~\ref{thegeom}.

\begin{figure}
\centering
\includegraphics[width=0.7\textwidth]{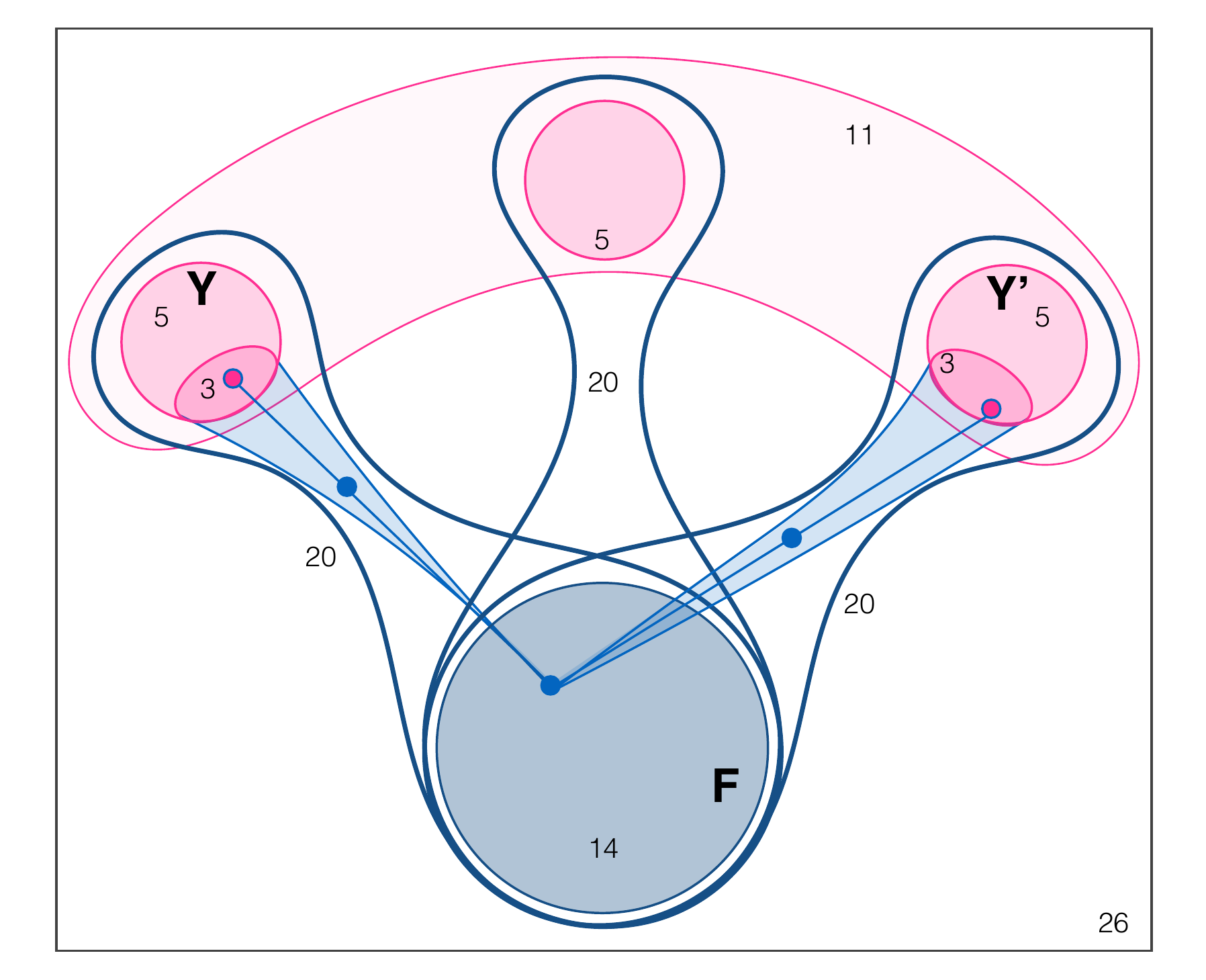}
\caption{A schematic representation of varieties isomorphic to $\mathcal{V}_{2}(\K,\mathbb{S})$, sharing the variety $\mathcal{V}_{2}(\K,\mathbb{H}')$, viewed  inside $\mathcal{V}_{2}(\K,\O')$ or $\mathcal{V}_{2}(\K,\CD(\bH',0))$.\label{e6sub}}
\end{figure}

Take $B=M(a,b,c,d,x,y,z,u)$ and $C=M(a',b',c',d',x',y',z',u')$ in $\mathbb{O}'$. Then $\rho(B,C)=(X_{0-26})$, where $X_{i-j}=(x_i,...,x_j)$, and we have
\begin{align*}
X_{0-2}&=(1,ad-bc-zx-uy, a'd'-b'c'-z'x'-u'y'),\\
X_{3-6}&=(aa'+bc'+xz'+yu', d'b+ab'+zu'-z'u, a'c+dc'+x'y-xy',dd'+b'c+x'z+y'u),\\
X_{7-10}&=(ax'+d'x+uc'-cu',d'y+ay'+cz'-c'z,a'z+dz'+by'-b'y,a'u+du'+b'x-bx'),\\
X_{11-26}&=(a',b',c',d',x',y',z',u',a,b,c,d,x,y,z,u).\end{align*}

Note that, if $B,C \in \mathbb{S}$ (i.e., $x=y=x'=y=0$), then $x_i=0$ for $i\in \{7,10,15,18,23,26\}=:J$; likewise, if $B,C \in \mathbb{S}'$, then $x_i=0$ for $i\in \{8,9,16,17,24,25\}=:J'$.  
Let $(e_0,...,e_{26})$ be the standard basis of $\mathbb{P}^{26}(\K)$. Put $I=\{0,...,26\}\setminus J$ and $I'=\{0,...,26\}\setminus J'$ and define $Y:=\<e_i \mid i \in J\>$ and $Y'=\<e_i\mid i\in J'\>$, and finally $F:=\<e_i:i\in I\cap I'\>$. Clearly, $\mathcal{V}_2(\K,\mathbb{S})$ is contained in the $20$-dimensional subspace $\<F,Y\>=\<e_i \mid i \in I\>$ and $\mathcal{V}_2(\K,\bH')$ is contained in $F$. We list our findings, the proofs of which can be found in the author's Ph.D. thesis (Section 5.2 of~\cite{thesis}); and as the pairs $\mathcal{V}_2(\K,\mathbb{S}), Y$ and $\mathcal{V}_2(\K,\mathbb{S}'), Y'$ play the same role, we only mention the former.

\begin{compactenum}[$(1)$]
\item The variety $\mathcal{V}_2(\K,\mathbb{S})$ is the intersection of the $20$-space $\<F,Y\>$ and $\mathcal{V}_2(\K,\bO')$.
\item Likewise,  $F \cap \mathcal{V}_2(\K,\bO')=F \cap \mathcal{V}_2(\K,\mathbb{S})$ coincides with $\mathcal{V}_2(\K,\mathbb{H}')$ and is hence isomorphic to the line Grassmannian $\mathcal{G}:=\mathcal{G}_{6,2}(\K)$ by Proposition~\ref{thegeom}. 
\item The subspace $Y$ is a singular subspace of $\mathcal{V}_2(\K,\mathbb{S})$, which is not contained in $\rho(\mathbb{S} \times \mathbb{S})$, so it is ``at infinity''.
\item Each point $p$ of  $\mathcal{V}_2(\K,\mathbb{S})\setminus Y$ corresponds to a unique point $p_G$ of $\mathcal{G}$ and  the set of points of $\mathcal{V}_2(\K,\mathbb{S})\setminus Y$ corresponding with this point $p_G$ forms an affine $4$-space, whose $3$-space at infinity $U_p$ belongs to $Y$. 
\item In view of the foregoing and the transitivity properties in $\mathcal{G}$, we have that transitivity on the set of pairs of collinear points of $\mathcal{V}_2(\K,\mathbb{S})\setminus Y$ and on the set of non-collinear points of $\mathcal{V}_2(\K,\mathbb{S})\setminus Y$.
\item The correspondence between $\mathcal{G}$ and $Y$, taking a point $p$ of $\mathcal{G}$ to the $3$-space $U_p$, is a linear duality. 
\end{compactenum}

In general, we  call two points of $\mathcal{V}_2(\K,\A)$ \emph{collinear} if the line of the ambient projective space determined by them is fully contained in  $\mathcal{V}_2(\K,\A)$.
In $\mathcal{V}_2(\K,\bO') \cong \mathcal{E}_{6}(\K)$, the convex closure of any pair of non-collinear points is isomorphic to a hyperbolic quadric in $\mathbb{P}^{9}(\K)$ (called a \emph{symp}, following the parapolar spaces terminology). The following proposition, whose proof is largely based on the correspondence given in Observation (6) above, lists the possibilities for the intersection of $\mathcal{V}_2(\K,\mathbb{S})$ with a symp of $\mathcal{V}_2(\K,\mathbb{O}')$. 

\begin{prop}[\cite{thesis}]\label{symps}
Let $\Sigma$ be a symp of $\mathcal{V}_2(\K,\mathbb{O}')$ such that $\zeta:=\Sigma\cap \mathcal{V}_2(\K,\mathbb{S})$ contains a pair of non-collinear points of $\mathcal{V}_2(\K,\mathbb{S})$. Then either
\begin{compactenum}[$(i)$]
\item $Y \cap \Sigma$ is a line $L$, in which case $\zeta$ is a cone with $1$-dimensional vertex $L$ and base isomorphic to a hyperbolic quadric in $\mathbb{P}^{5}(\K)$, and hence $\zeta=L^\perp \cap \Sigma$);
\item $Y \cap \Sigma$ is a $4$-space, in which case $\zeta=\Sigma$.
\end{compactenum}
In both cases, the convex closure (viewed in $\mathcal{V}_2(\K,\mathbb{S})$) of two non-collinear points of $\zeta$ is~$\zeta$.
\end{prop} 

We conclude that the set of symps of $\mathcal{V}_2(\K,\bO')$ containing a pair of non-collinear points of $\mathcal{V}_2(\K,\mathbb{S})$ gives rise to two types of symps of $\mathcal{V}_2(\K,\mathbb{S})$, which we distinguish as follows:
\[\Xi:=\{ \Sigma \cap \mathcal{V}_2(\K,\mathbb{S}) \mid \dim(Y \cap \Sigma)=1\},\quad \Theta:=\{\Sigma \cap  \mathcal{V}_2(\K,\mathbb{S}) \mid \dim(Y \cap \Sigma)=4\}.\]

\subsubsection{The case $\A=\CD(\L',0)$}
Here we use that  $\bH'':=\CD(\L',0)$ is isomorphic to $\{M(a,0,0,d,0,y,z,0)\mid a,d,y,z\in\K\}$. Clearly, $\bH'' \subseteq \mathbb{S}$ and hence $\mathcal{V}_2(\K,\bH'')$ belongs to $\mathcal{V}_2(\K,\mathbb{S})$ and is, as one can easily verify, contained in the $14$-space $\<e_i \mid i\in I \setminus \{4, 5,12,13,20,21\}\>$, which can also be given as $\<Y,F'\>$ where $F'$ is an $8$-space in $F$. 
The subspace $Y$ is a singular $5$-space of $\mathbb{V}_2(\K,\bH'')$ as well. As in the first case, $\mathbb{V}_2(\K,\mathbb{S}) \cap \<Y,F'\>=\mathbb{V}_2(\K,\bH'')$ and $\mathbb{V}_2(\K,\mathbb{H}'') \cap F'=\mathbb{V}_2(\K,\L')=:\mathcal{S}$, and the latter is isomorphic to the Segre variety $\mathcal{S}_{2,2}(\K)$ by Proposition~\ref{thegeom}. 

Now, let $U \cong \mathbb{P}^{5}(\K)$ be (an abstract) projective space whose line Grassmannian gives $\mathcal{G}$. Then the Segre variety $\mathcal{S}$, as sub-variety of $\mathcal{G}$, arises as the set of lines of $U$ intersecting two given disjoint planes  $\pi_1$ and $\pi_2$  non-trivially. The correspondence between $\mathcal{G}$ and $Y$ (as given in Observation (6))  implies that $Y$ is isomorphic to the dual of $U$, and hence the lines of $U$ intersecting both $\pi_1$ and $\pi_2$ non-trivially correspond to $3$-spaces having a line in common with two planes $Z_1$ and $Z_2$ in $Y$, and these $3$-spaces all arise as $U_p=p^\perp \cap Y$, for some point $p\in \mathcal{S}$.

We then have  the analogue of Proposition~\ref{symps}.

\begin{prop}[\cite{thesis}]\label{sympssegre}
Let $\Sigma$ be a symp of $\mathcal{V}_2(\K,\mathbb{O}')$ such that $\zeta:=\Sigma\cap \mathcal{V}_2(\K,\bH'')$ contains at least two non-collinear points of $(\mathcal{V}_2(\K,\bH'')\setminus Y) \cup Z_1 \cup Z_2$. Then either 
\begin{compactenum}[$(i)$]
\item $Y \cap \Sigma$ is a line $V$, in which case $\zeta$ is a cone with vertex $V$ and base isomorphic to a hyperbolic quadric in $\PAG^3(\K)$ (so $\zeta \subseteq L^\perp \cap \Sigma$); or,
\item $Y \cap \Sigma$ is a $4$-space $W$ generated by a line $V_i$ in $Z_i$ and the plane $Z_j$, with $\{i,j\}=\{1,2\}$, in which case $\zeta$ is a cone with vertex $V_i$ and base isomorphic to a Klein quadric over $\K$.
\end{compactenum}
In both cases, the convex closure (viewed in $\mathcal{V}_2(\K,\mathbb{H}'')$) of two non-collinear points of $\zeta$ is~$\zeta$.
\end{prop} 

Again, we define
\[\Xi:=\{ \Sigma \cap \mathcal{V}_2(\K,\bH'') \mid \dim(Y \cap \Sigma)=1\},\, \Theta:=\{\Sigma \cap  \mathcal{V}_2(\K,\bH'') \mid Y \cap \Sigma \text{ is a } 4\text{-space containing } Z_1 \text{ or } Z_2\}.\] 

\subsubsection{The case $\A=\mathbb{T}$}

Here we use $\mathbb{T} \cong\{M(a,0,0,d,0,y,0,0)\mid a,d,y \in \K\}$, and we obtain that $\mathcal{V}_2(\K,\mathbb{T})$ arises as the intersection of $\mathcal{V}_2(\K,\bH'')$ with the $11$-space generated by the Segre-variety $\mathcal{S}$ in $F$ and by the plane $Z_1$ in $Y$. Again, we have:

\begin{prop}[\cite{thesis}]\label{sympshalfsegre}
Let $\Sigma$ be a symp of $\mathcal{V}_2(\K,\mathbb{O}')$ such that $\zeta:=\Sigma\cap \mathcal{V}_2(\K,\mathbb{T})$ contains at least two non-collinear points of $\mathcal{V}_2(\K,\mathbb{T})$. Then either
\begin{compactenum}[$(i)$]
\item $\Sigma \cap Z_1$ is a point $V$, in which case $\zeta$ is a cone with vertex $V$ and base isomorphic to a grid quadric over $\K$;
\item $\Sigma \cap Z_1 =Z_1$, in which case  $\zeta$ is isomorphic to a Klein quadric over $\K$,.
\end{compactenum}
In both cases, the convex closure (viewed in $\mathcal{V}_2(\K,\mathbb{T})$) of two non-collinear points of $\zeta$ is~$\zeta$.
\end{prop} 

Also here, we define:
\[\Xi:=\{ \Sigma \cap \mathcal{V}_2(\K,\mathbb{T}) \mid \dim(Z_1 \cap \Sigma)=0\},\quad \Theta:=\{\Sigma \cap  \mathcal{V}_2(\K,\mathbb{T}) \mid  Z_1\subseteq \Sigma \}.\] 

\subsubsection{The case $\A=\CD(\mathbb{H}',0)$}
As alluded to before, $\mathcal{V}_2(\K,\CD(\mathbb{H}',0))$ does not exhibit the same behaviour as its three siblings $\mathcal{V}_2(\K,\mathbb{T})$, $\mathcal{V}_2(\K,\CD(\mathbb{L}',0))$ and $\mathcal{V}_2(\K,\mathbb{S})$. One thing that goes wrong for instance, is the following. Firstly, the schematic representation (cf.\ Figure~\ref{e6sub}) of the embedding of varieties isomorphic to $\mathcal{V}_2(\K,\mathbb{S})$ and containing the same line Grassmannian variety $\mathcal{V}_2(\K,\bH')$, is still applicable in this case. Consider the $5$-spaces $Y$, which are pairwise disjoint singular $5$-spaces in an $11$-dimensional subspaces. A calculation shows that these are on a regulus of $\mathcal{V}_2(\K,\CD(\mathbb{H}',0))$ (meaning that for each point of $Y$ there is a unique line of $\mathcal{V}_2(\K,\CD(\mathbb{H}',0))$ which meets the other such $5$-spaces), whereas in $\mathcal{V}_2(\K,\mathbb{S})$, these $5$-spaces were pairwise opposite (meaning that no point of one of them is on a line of $\mathcal{V}_2(\K,\mathbb{S})$ with a point of one of the others).  This is the reason why the convex closure of two non-collinear points of  $\mathcal{V}_2(\K,\bH')$ is not a quadric, as opposed to the situation in $\mathcal{V}_2(\K,\mathbb{S})$ (cf.\ Propositions~\ref{symps},~\ref{sympssegre} and~\ref{sympshalfsegre}).

\subsection{Properties of the Veronese varieties $\mathcal{V}_2(\K,\mathbb{T})$, $\mathcal{V}_2(\K,\CD(\L',0))$ and $\mathcal{V}_2(\K,\mathbb{S})$}

We show some properties satisfied by each of the varieties $\mathcal{V}_2(\K,\mathbb{T})$, $\mathcal{V}_2(\K,\CD(\L',0))$ and $\mathcal{V}_2(\K,\mathbb{S})$ (which we will later on use as their characterising properties). % Our main aim lies in the converse statement, which informally says that the above point-quadric geometries are the only of ``their kind'' to satisfy these two. We make this more precise in the next section.

In case of $\mathcal{V}_2(\K,\mathbb{T})$ and $\mathcal{V}_2(\K,\mathbb{S})$, we define $Z$ as the points of the subspace $Y$; in  $\mathcal{V}_2(\K,\CD(\L',0))$ we define $Z$ as the union of the two subspaces $Z_1$ and $Z_2$ and $Y$ as $\<Z_1,Z_2\>$. In the three varieties $\mathcal{V}_2(\K,\cdot)$, we set $X$ equal to the points in $\mathcal{V}_2(\K,\cdot)\setminus Y$. Recall the definitions of $\Xi$ and $\Theta$.

\begin{prop}\label{standardex} The Veronese  varieties $\mathcal{V}_2(\K,\mathbb{T})$, $\mathcal{V}_2(\K,\CD(\L',0))$ and $\mathcal{V}_2(\K,\mathbb{S})$ satisfy the following three properties:
\begin{compactenum}
\item[\emph{(S1)}] Each pair of distinct points $p_1,p_2 \in X \cup Z$ is contained in a member of $\Xi \cup \Theta$;
\item[\emph{(S2)}] for each pair of distinct members $\zeta_1, \zeta_2 \in \Xi \cup \Theta$, the intersection $\zeta_1 \cap \zeta_2$ is a singular subspace;
\item[\emph{(S3)}] for each point $x \in X$,  there exists $\xi_1,\xi_2$ in $\Xi$ such that $T_x=\<T_x(\xi_1),T_x(\xi_2)\>$.
\end{compactenum}
\end{prop}
\begin{proof}
Consider $\mathcal{V}_2(\K,\mathbb{A})$, where $\A \in \{\mathbb{T},\CD(\L',0),\mathbb{S}\}$.

\begin{enumerate}
\item[(S1)] If $p_1$ and $p_2$ are non-collinear, then they determine a unique symp $\Sigma$ of $\mathcal{V}_2(\K,\bO')$, which by assumption intersects $\mathcal{V}_2(\K,\mathbb{A})$ in two non-collinear points, so $\Sigma \cap \mathcal{V}_2(\K,\bO') \in \Xi \cup \Theta$ by Propositions~\ref{sympshalfsegre},~\ref{sympssegre} and~\ref{symps}. 

If $p_1$ and $p_2$ are on a line, then we can always find a point $p_3$ in $X\cup Z$ which is collinear to $p_1$ and not to $p_2$. Then the symp of $\mathcal{V}_2(\K,\mathbb{O}')$ containing $p_2$ and $p_3$ also contains $p_1$ and the same argument as above applies.

\item[{(S2)}] This is immediate as each member of $\Xi \cup \Theta$ is contained in a symp of $\mathcal{V}_2(\K,\bO')$ and two symps of the latter intersect in a singular subspace, which at its turn will intersect $\mathcal{V}_2(\K,\mathbb{A})$ in a singular subspace.

\item[{(S3)}] This can be shown by using the correspondence between the ``base'' variety (the part contained in $F$) and the subspace $Y$ (see Observation (6)), together with the properties of the base variety. 
\end{enumerate}
\end{proof}

\section{Dual split Veronese sets}\label{sVV}
We  introduce dual split Veronese sets formally. 
\subsection{Definition}\label{Defvv}
We work with a point set in $\mathbb{P}^N(\K)$ (where $\K$ is still an arbitrary field) and a family of quadrics:

\begin{defi}\label{svs} \em Let $R,V$ be integers with $V\geq -1$ and $R \geq 1$. An \emph{$(R,V)$-cone} $C$ is a cone with a $V$-dimensional vertex and as base a hyperbolic quadric of rank $R+1$ (i.e., a non-degenerate quadric of maximal Witt index in $\mathbb{P}^{2R+1}(\K)$); $C$ without its vertex is called an \emph{$(R,V)$-tube}.
\end{defi}

As can be seen in Propositions~\ref{symps}, \ref{sympssegre} and~\ref{sympshalfsegre}, we need to work with  two families of $(R,V)$-tubes with a different behaviour (called ``ordinary tubes'' and ``special tubes''); likewise, we also work with two types of point sets (called ``ordinary points'' and ``special points''). Informally speaking, the ``special points'' are points belonging to vertices of ``ordinary tubes'' and the ``special tubes'' are special in the sense that their base also contains points which are contained in the vertex of ``ordinary tubes''.

 Let $r,v,r',v',N$ be integers which are at least $-1$ with $r'>r \geq 1$. Suppose that $X \cup Z$ is a spanning point set of $\mathbb{P}^N(\K)$. We define $Y$ as the subspace spanned by the points of $Z$. Put  $d:=2r+v+1$ and $d':=2r'+v'+1$. Let $\Xi$ be a  collection of $(d+1)$-dimensional subspaces of $\mathbb{P}^N(\K)$ with $|\Xi|>1$ and $\Theta$ a possibly empty collection of $(d'+1)$-dimensional subspaces of $\mathbb{P}^N(\K)$ such that:
 \begin{itemize}
 \item For each $\xi \in \Xi$, the intersection $XY(\xi):=(X\cup Y)  \cap \xi$ is an $(r,v)$-cone $C_\xi$, $X(\xi):=X\cap \xi$ is a $(r,v)$-tube $T_\xi$ and $Y(\xi):=Y\cap \xi$ is the vertex of $C_\xi$;
 \item for each $\theta \in \Theta$, the intersection $XY(\theta):=(X \cup Y) \cap \theta$ is an $(r',v')$-cone $C_\theta$, $Y(\theta):=Y \cap \theta$ is precisely a generator $M$ of the quadric $C_\theta$ (which in particular contains the vertex $V_\theta$ of $C_\theta$),  and  $Z(\theta):=Z \cap \theta$ is the (disjoint) union of $V_\theta$ and some $r'$-space of $M$ complementary to it; lastly, $X(\theta):=X \cap \theta$ is $C_\theta\setminus M$.
\end{itemize}

A subspace $S$ of $\mathbb{P}^N(\K)$ is called \emph{singular} if all its points are contained in $X \cup Y$. For each point $x\in X$, we denote by $T_x$  the subspace spanned by all singular lines through $x$. 
\par\medskip

\begin{tcolorbox}
A quadruple $(X,Z,\Xi,\Theta)$ is called a \textbf{dual split Veronese set (DSV for short)}, with parameters $(r,v,r',v')$ if $\Theta$ non-empty and parameters $(r,v)$ if $\Theta$ empty, if the following axioms are satisfied:
\begin{itemize}
\item[(S1)] Each pair of distinct points $p_1,p_2 \in X \cup Z$ is contained in a member of $\Xi \cup \Theta$;
\item[(S2)] the intersection $\zeta_1 \cap \zeta_2$ of two distint members $\zeta_1, \zeta_2 \in \Xi \cup \Theta$ is singular;
\item[(S3)] for each point $x \in X$,  there exist $\xi_1,\xi_2$ in $\Xi$ such that $T_x=\<T_x(\xi_1),T_x(\xi_2)\>$;
\end{itemize}
If $(X,Z,\Xi,\Theta)$ satisfies (S1) and (S2), then we call it a \textbf{dual split \emph{pre}-Veronese set (pre-DSV for short)}. If $\Theta$ is empty, we use the adjective \textbf{mono-symplectic}, and otherwise \textbf{duo-symplectic}.
\end{tcolorbox}
\par\medskip

The goal is to classify the DSVs. We list a class of varieties of (pre-)DSVs first.

\subsection{Mono-symplectic DSVs}
In Lemma~\ref{existencesupersymps} and Proposition~\ref{trivialcase}, we show that a mono-symplectic DSV is projectively equivalent to a cone (possibly with empty vertex) over a mono-symplectic DSV with parameters $(r,-1)$ in which $Z$ is the empty set. Such sets have been classified by Schillewaert and Van Maldeghem in \cite{SVM}, who call these sets \emph{Mazzocca-Melone sets of split type $d$} (recall $d=2r+v+1$). There is a slight difference in the axioms used in \cite{SVM}, indeed, they consider the following two variants of (S3):

\begin{compactenum}
\item[(S3')] for each point $x \in X$, $\dim(T_x) \leq 2d$;
\item[(S3'')] for each point $x \in X$, $\dim(T_x)=2d$.
\end{compactenum}

If we filter out the varieties of the main result of \cite{SVM} that not only satisfy (S1), (S2) and (S3'), but also the stronger requirement (S3), we obtain that the mono-symplectic DSV with parameters $(r,-1)$ are the following:
\begin{compactenum}
\item[($r=1$)] The Segre varieties $\mathcal{S}_{1,2}(\K)$ and $\mathcal{S}_{2,2}(\K)$; 
\item[($r=2$)] the line Grassmannians  $\mathcal{G}_{5,2}(\K)$ and $\mathcal{G}_{6,2}(\K)$;
\item[($r=4$)]  the exceptional $\mathcal{E}_{6,1}(\K)$-variety.
\end{compactenum}

\subsection{Segre varieties, line Grassmannians and their projections}
We zoom in on the Segre varieties $\mathcal{S}_{k,\ell}(\K)$ (for general but finite $k,\ell$) and the line Grassmannians $\mathcal{G}_{n+1,2}(\K)$ of a projective space $\mathbb{P}^N(\K)$,  as they will be key ingredients for the examples of duo-symplectic (pre-)DSVs.

\subsubsection{Segre varieties with two families of maximal singular subspaces}\label{s}\label{HDS}\label{ds}
Let $\ell$ and $k$ be natural numbers with $\ell,k \geq 1$. The Segre variety $\mathcal{S}_{\ell,k}(\K)$ is the set of points in the image of the \emph{Segre map} $\sigma$, where $m:=(\ell+1)(k+1)-1$:
\[\sigma: \mathbb{P}^{\ell}(\K) \times \mathbb{P}^{k}(\K) \rightarrow  \mathbb{P}^{m}(\K): ((x_0,..,x_\ell),(y_0,...,y_k)) \mapsto  (x_iy_j)_{0 \leq i \leq \ell, 0 \leq j \leq k}.\]
This product can be visualised in $\mathbb{P}^{m}(\K)$ by taking an $\ell$-space $\Pi_\ell$ and a $k$-space $\Pi_k$ intersecting each other in precisely a point, and considering their direct product of $\Pi_\ell$ and $\Pi_k$ geometrically. There are two natural families of maximal singular subspaces, if we say that two maximal singular subspaces belong to the same family if and only if they are disjoint.

\subsubsection{Line Grassmannians of projective spaces}\label{DLG}
Let $n$ be a natural number with $n \geq 2$. The line Grassmannian $\mathcal{G}_{n+1,2}(\K)$ of $\mathbb{P}^n(\K)$ is the set of points in $\mathbb{P}^{\frac{1}{2}(n^2+n) -1}(\K)$ obtained by taking the images of all lines of $\mathbb{P}^n(\K)$ under the Pl\"ucker map
\[ \mathsf{pl}: (\<x_0,x_1,...,x_n),(y_0,y_1,...,y_n)\> \mapsto \left( \begin{vmatrix} x_i & x_j \\ y_i & y_j \end{vmatrix} \right)_{0 \leq i<j \leq n}.\] 

\subsubsection{Legal and injective projections}
We will encounter more general versions of the above defined geometries. This is caused by the fact that the varieties $\mathcal{S}_{\ell,k}(\K)$ and $\mathcal{G}_{n+1,2}(\K)$ are embeddings of the abstract geometries $\mathsf{A}_{\ell,1}(\K)\times\mathsf{A}_{k,1}(\K)$ and $\mathsf{A}_{n,2}(\K)$, respectively, but the latter geometries could admit other embeddings. The varieties $\mathcal{S}_{\ell,k}(\K)$ and $\mathcal{G}_{n+1,2}(\K)$  are however  their (absolutely) universal embeddings, as follows from the main results in Wells (\cite{Wel}) and Zanella (\cite{Zan}):

\begin{fact}\label{factuniemb}
Let $(\cP,\cL)$ be a point-line geometry isomorphic to either $\mathsf{A}_{\ell,1}(\K)\times\mathsf{A}_{k,1}(\K)$, for $\ell,k\geq 1$, or $\mathsf{A}_{n,2}(\K)$, for $n\geq 2$, such that $\cP$ be a spanning subset of a projective space $\mathbb{P}^{m}(\K)$, and each member of $\mathcal{L}$ is the set of all points on a certain line of $\mathbb{P}^{m}(\K)$.   Then $\cP$ arises as an injective projection of $\mathcal{S}_{\ell,k}(\K)$, or $\mathcal{G}_{n+1,2}(\K)$, respectively. \end{fact}

We could, in the above descriptions of the (half) dual Segre varieties and dual line Grassmannians, replace the Segre varieties and line Grassmannians by injective projections of them, if any, and apply  the same construction. However, if we want the resulting geometry to satisfy (S2), we cannot use \emph{any} injective projection, as we will now explain.

Let $\Omega$ be either a Segre variety $\mathcal{S}_{\ell,k}(\K)$ or a line Grassmannian  $\mathcal{G}_{n+1,2}(\K)$. Let $X$ be its point set and $\Xi$ be the set of subspaces spanned by the quadrics one obtains by taking the convex closures of points of $X$ at distance 2. As can be verified from their respective algebraic definitions, $(X,\Xi)$ is a mono-symplectic split pre-DSV with $Z=\emptyset$. We will only consider projections of them which preserve the fact that they are pre-DSVs. In general, suppose $(X,\Xi)$ is a mono-symplectic pre-DSV with $Z =\emptyset$ and $\<X\>=\mathbb{P}^{M}(\K)$ for some $M\in \mathbb{N}$. 

\begin{defi}\label{legal} \em We say that a subspace $S$ of $\mathbb{P}^{m}(\K)$ is \emph{legal with respect to $(X,\Xi)$} if $S$ is disjoint from $\<\xi_1,\xi_2\>$ for each pair of symps $\xi_1,\xi_2\in\Xi$. The projection of $(X,\Xi)$ from $S$ onto a subspace of $\mathbb{P}^{m}(\K)$  complementary to $S$ is called a \emph{legal projection} of $(X,\Xi)$. \end{defi}

Note that a legal projection is automatically injective. By definition, any legal projection of $(X,\Xi)$ is also a mono-symplectic pre-DSV.

\subsection{Examples of duo-symplectic (pre-)DSVs}\label{HDS}\label{ds}\label{DLG}
We  use the above geometries to construct duo-symplectic pre-DSVs.

\textbf{Half dual Segre varieties.} Inside $\mathbb{P}^{m+\ell+1}(\K)$,  we consider a Segre variety $\mathsf{S}:=\mathcal{S}_{\ell,k}(\K)$ and an $\ell$-space $Y$ complementary to $\<\mathsf{S}\>$. Let $S$ be any $\ell$-space of $\mathsf{S}$ and $\chi_S$  a linear duality  between $\Pi$ and $Y$, which hence takes a point of $S$ to a hyperplane of $Y$. We extend $\chi_S$ to a map $\chi$ from all points of $\mathsf{S}$ to $Y$  by defining, for a point $x\in \mathsf{S} \setminus S$, its image $\chi(x)$ as $\chi_S(x_S)$, where $x_S$ is the unique point in $S$ collinear to $x$.  The union of all points in  $\{\<x,\chi(x)\> \setminus \chi(x) \mid x\in \mathsf{S}\}$ is the point set $X$ of what we call a \emph{half dual Segre variety} and denote by $\mathcal{HDS}_{\ell,k}(\K)$. 

\textbf{Dual Segre varieties.}  Inside $\mathbb{P}^{m+\ell+k+2}(\K)$, we consider a Segre variety $\mathsf{S}:=\mathcal{S}_{\ell,k}(\K)$, and in an $(\ell+k+1)$-space $Y$ complementary to it, we take two disjoint subspaces $Z_1$ and $Z_2$ of respective dimensions $\ell$ and $k$.
As above, let $S_1$ be any $\ell$-space of $\mathsf{S}$, and also take any $k$-space $S_2$ of $\mathsf{S}$ which intersects $S_1$ in a point. For $i=1,2$, let $\chi_{S_i}$ be a linear duality between $S_i$ and $Z_i$, thus taking a point of $S_i$ to a hyperplane of $Z_i$. We extend the maps $\chi_{S_1}$ and $\chi_{S_2}$ to a map $\chi$ from all points of $\mathsf{S}$ to $\<Z_1,Z_2\>$  by defining, for a point $x$ of $\mathsf{S}$, its image $\chi(x)$ as $\<\chi_{S_1}(x_{S_1}),\chi_{S_2}(x_{S_2})\>$, where $x_{S_i}$ is equal to $x$ if $x\in S_i$, or, if not, it is the unique point in $S_i$ collinear to $x$. 
The union of all points in $\{\<x,\chi(x)\> \setminus \chi(x) \mid x\in \mathsf{S}\}$ is the point set $X$ of a \emph{dual Segre variety}, which we will denote by $\mathcal{DS}_{\ell,k}(\K)$. 

\begin{rem} \em The half dual Segre variety $\mathcal{S}_{\ell,k}(\K)$ is the projection of the dual Segre variety $\mathcal{S}_{\ell,k}(\K)$ from the subspace $Z_2$. 
\end{rem}

\textbf{Dual line Grassmannians.}
Consider, inside $\mathbb{P}^{\frac{1}{2}(n^2+3n)}(\K)$, an $n$-space $Y$ and a complementary subspace $F$ of dimension $\frac{1}{2}(n^2+n)-1$. In $F$,  take a line Grassmannian $\mathcal{G}:=\mathcal{G}_{n+1,2}(\K)$, which is the image under $\mathsf{pl}$ of a certain $n$-dimensional projective space $\mathbb{P}$. Let $\chi':\mathbb{P}\rightarrow Y$ be a linear duality, and note that each line of $\mathbb{P}$ corresponds to a $(n-2)$-space of $Y$. As such, we can define a map $\chi$ between $\mathcal{G}$ and $Y$ which is defined by, for each point $x\in \mathcal{G}$, taking $x$ to $\chi'(\mathsf{pl}^{-1}(x))$.
The union of all points in $\{\<x,\chi(x)\> \setminus \chi(x) \mid x\in \mathcal{G}\}$ is the point set $X$ of a \emph{dual line Grassmannian}, which we will denote by $\mathcal{DG}_{n+1,2}(\K)$. 

\par\medskip
Each of these three classes of geometries contains a duo-symplectic DSV (where the convex closures of two points at distance 2 gives the members of $\Xi \cup \Theta$, and such a convex closure belongs to $\Xi$ if and only if it only shares its vertex with $Y$):
\par\medskip
\begin{prop}\label{vervar}
The varieties $\mathcal{HDS}_{2,2}(\K)$, $\mathcal{DS}_{2,2}(\K)$ and $\mathcal{DG}_{6,2}(\K)$ are isomorphic to $\mathcal{V}_2(\K,\mathbb{T})$, $\mathcal{V}_2(\K,\CD(\L',0))$ and $\mathcal{V}_2(\K,\mathbb{S})$, respectively, and hence they are dual split Veronese sets, with respective parameters $(1,0,2,-1)$, $(1,1,2,1)$, $(2,1,4,-1)$. 
\end{prop}

\begin{proof}
The first assertion follows from the description of $\mathcal{V}_2(\K,\mathbb{T})$, $\mathcal{V}_2(\K,\CD(\L',0))$ and $\mathcal{V}_2(\K,\mathbb{S})$ in Section~\ref{splitgeom}: the decomposition into a subspace $Y$ and a ``base variety'' (which is the Segre variety $\mathcal{S}_{2,2}(\K)$ in the smallest two cases and the line Grassmannian $\mathcal{G}_{6,2}(\K)$ in the largest case) is given, together with the structure of $p^\perp \cap Y$ for each point $p$ in the base variety. The second assertion then follows from Proposition~\ref{standardex}.
\end{proof}

\begin{rem}
\em The (half) dual Segre varieties and dual line Grassmannians other than these in Proposition~\ref{vervar} do not satisfy Axiom (S3). %We do not want to claim that they satisfy Axioms (S1) and (S2) because this is not our main concern, but probably they do.
\end{rem}

%We do not claim that the (half) dual Segre varieties and the dual line Grassmannians with general parameters are pre-DSVs too, as this is not our main concern, we merely describe these geometries as they will arise naturally when studying the pre-DSV. Yet it should be conceivable that they satisfy (S1) and (S2), for they exhibit similar behaviour as their smaller siblings occurring in the Main theorem, except for (S3), which does have a clear dependence on the parameters (soon, the tangent space in a point becomes larger than the subspace generated by the tangent spaces in that point of two symps through that point). 

\begin{defi}\label{defmut}\em Consider the (half) dual variety $\mathsf{(H)D\Omega}$ associated to $\Omega$. Then we may replace $\Omega$ by a legal projection $\Pi$ of $\Omega$, and we may re-position $Y$ in such a way that, after applying the same construction as before (which does not depend on the mutual position of $\Pi$ and $Y$) and obtaining a point set $X$, the projection of $\<\Pi\>\cap X$ from $\<\Pi\> \cap Y$ (onto a subspace of $\<\Pi\>$ complementary to $\<\Pi\> \cap Y$) yields an injective projection of $\Pi$; the ambient projective space is afterwards restricted to $\<\Pi,Y\>$. The resulting structure is called a \emph{mutant} of  $\mathsf{(H)D\Omega}$.\end{defi}

 Luckily, for the DSVs,  (S3) forces the occurring Segre varieties and line Grassmannians to be rather small, in which case they do not admit proper legal projections:
 
\begin{prop}\label{noproj}The following varieties do not admit proper legal projections:\begin{compactenum}[$(i)$] 
\item  The Segre varieties  $\mathcal{S}_{\ell,k}(\K)$ with $\ell \leq 3$ and $k\geq 1$; 
\item the line Grassmannians $\mathcal{G}_{5,2}(\K)$ and $\mathcal{G}_{6,2}(\K)$.\end{compactenum}
\end{prop}

\begin{proof}
Since $\mathcal{S}_{\ell,k}(\K)$ with $\ell <3$ is contained in $\mathcal{S}_{3,k}(\K)$ it suffices to show that the latter does not admit proper legal projections. The Segre variety $\mathcal{S}_{3,k}(\K)$ is defined by the $4\times (k+1)$ matrices over $\K$ of rank 1 in the projective space defined by the vector space of all $4\times (k+1)$ matrices over $\K$. If $A$ is such a matrix of rank 4, then $A$ is the sum of four rank $1$ matrices $A_1$, $A_2$, $A_3$ and $A_4$ which are pairwise not collinear. Let $\xi_1$ and $\xi_2$ be the respective members of $\Xi$ determined by the pairs  $(A_1,A_2)$ and $(A_3,A_4)$.  Then $A \in \<\xi_1,\xi_2\>$ is not a legal point w.r.t.\ $\mathcal{S}_{3,k}(\K)$. If $A$ has rank $2$ or $3$, $M$ is already the sum of two or three rank 1 matrices, respectively, and the same conclusion can be reached analogously. 

The second assertion is a direct consequence of the main result in \cite{SVM}. 
\end{proof}

Moreover, in these small cases we will, in Lemmas~\ref{1':FX},~\ref{1:FX} and~\ref{2:FX},  be able to show that $\<\Pi\> \cap Y$ needs to be empty, implying that no mutants occur for the DSVs. The following two lemmas will later on help us with that:

\begin{lemma}\label{s22p}
Let $\Sigma$ be a $4$-space contained in the $8$-dimensional projective space generated by the Segre variety $S:=\mathcal{S}_{2,2}(\K)$ and suppose that $\Sigma$ intersects $S$ in exactly a grid $G$. Then there exists a grid $G'$ of $S$ such that $\<G'\>$ intersects $\Sigma \setminus G$ non-trivially.
\end{lemma}
\begin{proof}
Let $\pi_1$ and $\pi_2$ be two planes of $S$ intersecting each other in a unique point $p$, and which are disjoint from $G$.  The $4$-space $\<\pi_1,\pi_2\>$  has a point $q$ in common with $\Sigma$, which does not belong to $G$ as $\<\pi_1,\pi_2\> \cap S=\pi_1 \cup \pi_2$. Then $q$ is contained in a unique plane $\<L_1,L_2\>$ with $L_i$ a line of $\pi_i$ through $p$, for $i=1,2$. Hence $q$ belongs to the subspaces spanned by the  grid $G'$ of $S$ determined by $L_1$ and $L_2$.
\end{proof}

\begin{lemma}\label{a52p}
Let $\Sigma$ be a $10$-space contained in the $14$-dimensional projective space generated by the line Grassmannian $A:=\mathsf{A}_{5,2}(\K)$. If $\Sigma \cap A$ contains a line Grassmannian $A':=\mathsf{A}_{4,2}(\K)$, then $A' \subsetneq \Sigma \cap A$.\end{lemma}
\begin{proof}
Note that $A' \subseteq A$ corresponds to a  $4$-dimensional subspace $P'$ of a projective $5$-space $P$; and consider any $4$-space $V$ in $A$ corresponding to the set of lines through a point of $P \setminus P'$. Then $V$ is disjoint from $A'$ and  $V \cap \Sigma$ is non-empty.
\end{proof}

\subsection{Main result}\label{mainsub}
We are ready to state the main result.

\begin{main}\label{main}
Let $(X,Z,\Xi,\Theta)$ be a dual split Veronese set with parameters $(r,v,r',v')$ where $\<X,Z\>=\mathbb{P}^N(\K)$ for some arbitrary field $\K$ with $|\K|>2$. If mono-symplectic, then  $X$ is projectively equivalent to a cone with a vertex of dimension $v^*$ (possibly, $v^*=-1$), whose points are those of $Z$, over one of the following geometries:
\begin{enumerate}[$(i)$ ]
\item A Segre variety $\mathcal{S}_{1,2}(\K)$ or $\mathcal{S}_{2,2}(\K)$, a line Grassmannian  $\mathcal{G}_{5,2}(\K)$  or $\mathcal{G}_{6,2}(\K)$, or the variety $\mathcal{E}_{6,1}(\K)$; in this case $v=v'=v^*$.
\end{enumerate}

If duo-symplectic, then $X$ is either projectively equivalent to a cone with a vertex of dimension~$v^*$ (possibly, $v^*=-1$)  over one of the following geometries:

\begin{enumerate}
\item[$(ii)$] A half dual Segre variety $\mathcal{HDS}_{2,k}(\K)$, where $k\in \{1,2\}$, which is a dual split Veronese set with parameters $(1,0,2,-1)$;% and in this case $r=1$, $v=r'+v^*-1$ and $v'=v^*$;
\item[$(iii)$] A dual line Grassmannian variety $\mathcal{DG}_{6,2}(\K)$, which is a dual split Veronese set with parameters $(2,1,4,-1)$,
\end{enumerate}
or projectively equivalent to the following geometry:
\begin{itemize}
\item[$(iv)$] A dual Segre variety $\mathcal{DS}_{2,2}(\K)$, with parameters $(1,1,2,1)$.
\end{itemize}
In particular, the varieties in $(i)$ up to $(iv)$ are subvarieties of the Veronese variety $\mathcal{V}_2(\K,\bO')$ over the split octonions $\bO'$, and apart from $\mathcal{S}_{1,2}(\K)$,  $\mathcal{G}_{5,2}(\K)$ and $\mathcal{HDS}_{2,1}(\K)$, all of them are a Veronese variety $\mathcal{V}_2(\K,\A)$ for some split quadratic alternative algebra $\A$ whose radical is either empty or generated by a single element $t$. 
\end{main}

\par\medskip

\begin{rem}\label{f2}\em 
We exclude the field of two elements in the above result because already one of the very preliminary lemmas (Lemma~\ref{uniquesymp}) might fail if $|\K|=2$. An alternative approach is required, and seeing the high cost and low benefits, we did not pursue this. No counterexamples are known however. \end{rem}

\subsection{Structure of the proof}\label{sotp}

In \textbf{Section~\ref{bp}}, we deduce some general properties and set up the inductive proof. In Lemma~\ref{existencesupersymps} we show that in a \emph{mono-symplectic} (pre-)DVS all members of $\Xi$ have the same vertex. Proposition~\ref{trivialcase} then reduces this case to the result of~\cite{SVM}, which leads to Main Result~\ref{main}$(i)$.  From that point onwards, we assume that the (pre-)DSV is \textbf{duo-symplectic}, in which case it easily follows (cf.\ Lemma~\ref{xinSS}) that through each $x\in X$, there is a member of $\Theta$.

The proof uses \textbf{induction on $r$}, the base case being $r=1$. For a (pre-)DSV with parameters $(r,v,r',v')$ with $r >1$, we show in  Lemma~\ref{res} that its $X$-point-residues are pre-DSVs with parameters $(r-1,v,r'-1,v')$, not necessarily satisfying (S3); whence the need to study \textbf{pre}-DSVs.

In \textbf{Section~\ref{r=1,2SS}} we deal with  (pre-)DSVs $(X,Z,\Xi,\Theta)$ with parameters $(1,v,r',v')$, with the additional assumption that through each point $x \in X$, there are at least two members of~$\Theta$. This case leads to the  \textbf{dual Segre varieties}: Theorem~\ref{1':projuni} shows that $X$ is a mutant of the dual Segre variety $\mathcal{DS}_{r',r'}(\K)$, and if (S3) holds, $r'=2$ ($\rightarrow$ Main Result~\ref{main}$(iv)$). 

In \textbf{Section~\ref{1SS/point}}, we treat (pre-)DSVs $(X,Z,\Xi,\Theta)$ in which there is a point in $X$ through which there is a unique member of $\Theta$, in which case it turns out that $r=1$ (cf.\ Lemma~\ref{1:r=1}). % and that there  is  a {unique} member of $\Theta$ through \emph{each} $x \in X$.   
 This time, we obtain the \textbf{half dual Segre varieties}: Theorem~\ref{1:projuni} shows that, up to projection from a subspace of $Y$ contained in each member of $\Xi \cup \Theta$, $X$ is the point set of a mutant of the half dual Segre variety $\mathcal{HDS}_{r',k}(\K)$ and, if (S3) holds, then $(r',k) \in \{(2,1), (2,2)\}$  ($\rightarrow$ Main Result~\ref{main}$(ii)$).

In \textbf{Section~\ref{r>1}}, we turn our attention to the (pre-)DSV $(X,Z,\Xi,\Theta)$ with parameters $(r,v,r',v')$ where $r > 1$. We first show in Lemma~\ref{2:uniSS} that each singular line is contained in at least one member of $\Theta$. In view of the above, we distinguish between the case where there is a singular line contained in a unique member of $\Theta$ (Section~\ref{cas1}) and the case where no such line exists (Section~\ref{cas2}). In the first case, we immediately deduce $r=2$ and  we obtain \textbf{dual line Grassmannians}: Theorem~\ref{2:projuni} shows that, up to projection from a subspace of $Y$ contained in each member of $\Xi \cup \Theta$, $X$ is the point set of a mutant of the dual line Grassmannian $\mathcal{DG}_{r'+2,2}(\K)$ and, if (S3) holds, $r'=4$ ($\rightarrow$ Main Result~\ref{main}$(iii)$).  Finally, Proposition~\ref{r>2kanni} shows, relying on Case 1, that a DSV with parameters $(r,v,r',v')$ with $r>1$ always has a singular line contained in a unique member of $\Theta$, leading us to Case 1. 

\section{Preliminaries}\label{bp}

Let $(X,Z,\Xi,\Theta)$ be a pre-DSV with parameters $(r,v,r',v')$ with $\<X,Z\>=\mathbb{P}^N(\K)$ for an arbitrary field $|\K| >2$.

\subsection{Basic properties}

Recall that a subspace $S$ of $\mathbb{P}^N(\K)$ is called \emph{singular} if its points are contained in $X \cup Y$. If moreover $S \subseteq X$, then we call $S$ an $X$-space. Two subspaces are called \emph{collinear} if there is a singular subspace containing them. 

\begin{lemma}\label{XYpointsonline}
A line $L$ of $\mathbb{P}^N(\K)$ containing at least three points of $X\cup Y$ is singular. Moreover, a singular line contains at most one point in $Y$.
\end{lemma}
\begin{proof}
If $L$ contains two points of $Y$, then $L$ belongs to $Y$ since the latter is a subspace by definition. So, if $|L \cap (X \cup Y)| \geq 3$, then we may assume that $L$ contains at least two points $x_1,x_2$ of $X$. By (S1), these are contained in a member $\zeta$ of $\Xi \cup \Theta$. Since $XY(\zeta)$ is a quadric containing $L$, it follows that $L$ is singular.
\end{proof}

\begin{lemma}\label{uniquesymp}
If $p_1$ and $p_2$ are non-collinear points of $X\cup Z$, then there is a unique member of $\Xi \cup \Theta$ containing them, which is denoted by $[p_1,p_2]$.
\end{lemma}

\begin{proof}
By (S1), there exists a member of $\Xi \cup \Theta$ containing $p_1$ and $p_2$. If there were at least two of them, then by (S2) their intersection is a singular subspace containing $p_1p_2$. But then the line $p_1p_2$ is singular, a contradiction.
\end{proof}

\begin{lemma}\label{planes}
Let $L_1$ and $L_2$ be two singular lines not entirely contained in $Y$, intersecting each other in a unique point $s$. Then either $\<L_1,L_2\>$ is a singular plane or $L_1$ and $L_2$ are contained in a unique member of $\Xi \cup \Theta$, denoted $[L_1,L_2]$.
\end{lemma}

\begin{proof}
Take $X$-points $x_i \in L_i \setminus \{s\}$, $i=1,2$.  Suppose first that $x_1$ is not collinear to $x_2$. By Lemma~\ref{XYpointsonline}, $x_1$ and $x_2$ are the only points of $X\cup Y$ on $x_1x_2$. Since $|\K|>2$, we can take $X$-points $x'_i \in L_i \setminus\{s\}$ distinct from $x_i$, $i=1,2$. Since $\<L_1,L_2\>$ is a plane of $\mathbb{P}^N(\K)$, the line $x'_1x'_2$ intersects $x_1x_2$ in a point $p$, distinct from $x_1,x_2$ and hence, by the foregoing, $p \notin X\cup Y$. Therefore, the line $x'_1x'_2$ is not singular. If $[x_1,x_2]$ and $[x'_1,x'_2]$ are distinct, then by (S2) we get $p \in X\cup Y$ after all, a contradiction. Hence $[x_1,x_2]$ and $[x'_1,x'_2]$  coincide and contain both $L_1$ and $L_2$. 

So we may assume that $x_1 \perp x_2$ for each pair of $X$-points $x_i \in L_i\setminus\{s\}$, $i=1,2$. Let $p \in \<L_1,L_2\>\setminus (L_1 \cup L_2)$ be arbitrary. As $|\K|>2$, there is a line through $p$ meeting $L_1$ and $L_2$ in distinct $X$-points. By assumption, these $X$-points are collinear and hence $p \in X \cup Y$. We conclude that the plane $\<L_1,L_2\>$ is singular indeed. 
\end{proof}

\begin{defi}\em
For each point $p \in X \cup Y$, we denote by $p^\perp$ the union of all singular lines through $p$ with at most one point in $Y$ (i.e., not entirely contained in $Y$). \end{defi}

\begin{lemma}\label{convexclosure}
Let $\zeta \in \Xi \cup \Theta$ and $p \in (X \cup Y)\setminus \zeta$ arbitrary. Then the set $p^\perp \cap \zeta$ contains no two non-collinear points.  Consequently, $XY(\zeta)$ is a convex subspace of $X \cup Y$ with respect to singular lines not entirely contained in $Y$, whose vertex is the subspace of $Y$ collinear to any two non-collinear $X$-points of $\zeta$.
\end{lemma}

\begin{proof}
Suppose  $p_1,p_2$ are non-collinear points in $p^\perp \cap \zeta$. Put $L_i:=pp_i$. Then $L_1$ and $L_2$ are singular lines not entirely contained in $Y$ and hence, Lemma~\ref{planes} implies that $p \in [p_1,p_2]=\zeta$, a contradiction. For the second assertion, it suffices to note that the unique line between two collinear points of $XY(\zeta)$ is contained in $XY(\zeta)$. 
\end{proof}

We can  extend Lemma~\ref{planes} to higher-dimensional subspaces.

\begin{lemma}\label{kspaces}
Let $S_1$ and $S_2$ be two singular subspaces of dimension $k$, with $k\geq 1$, not entirely contained in $Y$, intersecting each other in a $(k-1)$-space $S$. Then either $\<S_1,S_2\>$ is a singular $(k+1)$-space, or $S_1$ and $S_2$ are contained in a unique member of $\Xi \cup \Theta$.
\end{lemma}

\begin{proof}
If $k=1$, this follows from Lemma~\ref{planes}, so let $k > 1$. We may assume that each pair of $X$-points $x_1,x_2$ with $x_1 \in S_1 \setminus S$ and $x_2\in S_2 \setminus S$ is collinear: if not,  Lemma~\ref{convexclosure} implies that $[x_1,x_2]$ contains $\<S_1,S_2\>$. 
Let $p$ be any point in $\<S_1,S_2\>\setminus(S_1\cup S_2)$. For any $X$-point $x_1 \in S_1 \setminus S$, the line $x_1p$ intersects $S_2 \setminus S$ in a point $p_2$. If $p_2\in X$, then $p\in x_1p_2 \subseteq X\cup Y$. If $p_2 \in Y$,  take an $X$-point $x'_1$ in $S_1 \setminus (S\cup\{x_1\})$ such that the line $x_1x'_1$ intersects $S$ in an $X$-point $x$ (note that $S\nsubseteq Y$ since $p_2\in Y$ and $S_2 \nsubseteq Y$). Let $p'_2$ be the point in $S_2\setminus S$ on $x'_1p$.  Then $p'_2$ belongs to $xp_2$ and, as $x\in X$, also $p'_2\in X$.  So $p\in X \cup Y$ as before. We conclude that $\<S_1,S_2\>$ is singular indeed.
\end{proof}

\begin{defi}\label{YS}\em
For each $X$-space $S$, we denote by $Y_S$ the set of points of $Y$ collinear to (all points $p$ of) $S$, i.e., $Y_S:= \bigcap_{p \in S} (p^\perp \cap Y)$. \end{defi}

\begin{cor}\label{collY}
For each $X$-space $S$ of dimension $k-1\geq 0$, $Y_S$ is a subspace of $Y$.
\end{cor}
\begin{proof}
Let $S_1$ and $S_2$ be singular $k$-spaces through $S$ such that $S\setminus S_i$ contains a point $y_i \in Y$, $i=1,2$. By Lemma~\ref{kspaces}, $\<S_1,S_2\>$ is either singular or contained in a member $\zeta$ of $\Xi \cup \Theta$. In both cases we conclude that $y_1y_2$ is singular, noting that $Y(\zeta)$ is by definition a singular subspace of~$\zeta$. \end{proof}

\begin{lemma}\label{subspaceinsymp}
Let $S$ be a singular $k$-space with $k \geq 1$ and suppose $\zeta \in \Xi \cup \Theta$. If $S \cap \zeta$ is a hyperplane of $S$ not entirely contained in $Y$ and not a maximal singular subspace of $\zeta$, then there is a $\zeta' \in \Xi \cup \Theta$ through $S$ such that $\zeta' \cap \zeta$ is not collinear to $S$.
\end{lemma}

\begin{proof}
The assumptions on $S\cap \zeta$ yield a pair of non-collinear singular subspaces $S_1$ and $S_2$ of $\zeta$ through $S \cap \zeta$ that are not entirely contained in $Y$.  If both $\<S,S_1\>$ and $\<S,S_2\>$ were singular, then an $X$-point of $S\setminus \zeta$ (which exists because $S \nsubseteq Y$) would be collinear to a pair of non-collinear $X$-points of $S_1, S_2$, contradicting Lemma~\ref{convexclosure}. So we may assume that $\<S,S_1\>$ is not singular. By Lemma~\ref{kspaces}, $\<S,S_1\>$ is contained in a member of $\Xi \cap \Theta$.
\end{proof}

We record a special case of the previous lemma. Observe that the crucial difference between $\Xi$ and $\Theta$ is that for each  $\xi\in \Xi$, each point of $Y(\xi)$ is collinear to each point of $X(\xi)$, whereas for each $X$-point of $\theta\in\Theta$  there is a point $y \in Y(\theta)$ not collinear to it. 

\begin{cor}\label{supersymp/plane}
Let $L$ be a singular line with a unique point $y\in Y$, meeting some $\zeta \in \Xi\cap\Theta$ in an $X$-point $x$. Then there is a $\theta \in \Theta$ through $L$ sharing an $X$-line with $\zeta$.
\end{cor}
\begin{proof}
Lemma~\ref{subspaceinsymp} implies that $L$ is contained in a member $\zeta'$ of $\Xi \cup \Theta$, $i=1,2$, with $\zeta \cap \zeta'$ a singular subspace through $x$ not collinear to $L$. Since all lines through $x$ containing a point of $Y$ are collinear to $L$ by Corollary~\ref{collY}, this implies that $\zeta \cap \zeta'$ contains an $X$-line $L'$ not collinear to $L$. In particular, $y$ is not collinear to $L'$, and hence $\zeta' \in \Theta$.
\end{proof}

\subsection{Projections of $(X,Z,\Xi,\Theta)$}
The next lemma explains why, in the main theorem, we speak of ``a cone with vertex $V^*$''.

\begin{lemma}\label{resY}
If $V^*$ is a subspace of $Y$ of dimension $v^*$ collinear to all points of $X$, then the projection of a (pre-)DSV $(X,Z,\Xi,\Theta)$ with parameters $(r,v,r',v')$ from $V^*$ onto a complementary subspace of $\mathbb{P}^N(\K)$ is a (pre-)DSV with parameters $(r,v-v^*-1,r',v'-v^*-1)$ inside $\mathbb{P}^{N-v^*-1}(\K)$.
\end{lemma}
\begin{proof}
If all points of $X$ are collinear to $V^*$, then all members of $\Xi$ and $\Theta$ have $V^*$ in their vertex (cf.\ Lemma~\ref{convexclosure}). A straightforward verification shows that the projection of $(X,Z,\Xi,\Theta)$ from $V^*$ satisfies (S$i$) if $(X,Z,\Xi,\Theta)$ does, for each $i\in\{1,2,3\}$.
\end{proof}

There is another projection that we will frequently make use of. Let $F$ be a subspace of $\mathbb{P}^N(\K)$ complementary to $Y$.

\begin{defi}\label{RHO} \em The projection of $(X,Z,\Xi,\Theta)$ onto $F$ is induced by the following map.
\[\rho: X \rightarrow F: x \mapsto \<Y,x\> \cap F.\]
\end{defi}

\begin{defi}\label{CHI} \em
The \emph{connection} map between $\rho(X)$ and $Y$ (recalling $Y_x= x^\perp \cap Y$) is defined as follows:
\[\chi: \rho(X) \mapsto Y: \rho(x) \mapsto Y_x.\]
\end{defi}

We show some general properties on $\rho$ and $\chi$ (in particular that $\chi$ is well defined).

\begin{lemma}\label{proprhochi}\label{1:inj}\label{2':inj}
\begin{compactenum}[$(i)$]
\item For each $x\in X$, $\rho^{-1}(\rho(x))=\<x,Y_x\> \cap X$ and hence $\chi$ is well defined;
\item for each $\xi \in \Xi$, $\rho(X(\xi))$ is a non-degenerate hyperbolic quadric of rank $r+1$;
\item  for each $\theta \in \Theta$, $\rho(X(\theta))$ is a singular subspace of dimension $r'$.
\end{compactenum}
\end{lemma}
\begin{proof}
$(i)$ If $\rho(x)=\rho(x')$ for points $x,x' \in X$ with $x\neq x'$, then $\<x,x',Y\>$ contains $Y$ as a hyperplane. Therefore, the line $xx'$ meets $Y$ in a point $y\in Y$, which by Lemma~\ref{XYpointsonline} means that $xx'$ is singular. In particular, $y\in Y_x$,  and so $x' \in \<x,Y_x\>\cap X$ indeed. Conversely it is clear that all points of the latter set are mapped onto the same point by $\rho$.  Consequently, $\<x',Y_{x'}\>=\<x,Y_x\>$, so in particular  $Y_{x'}=Y_x$, from which we conclude that $\chi$ is well defined.

Assertions $(ii)$ and $(iii)$ are obvious noting that, for each $\xi \in \Xi$, $Y \cap \xi$ is the vertex of $\xi$ and for each $\theta \in \Theta$, $Y \cap \theta$ is a maximal singular subspace of $\theta$.
\end{proof}

\begin{rem}\em
Notwithstanding the fact that $Y$ is a subspace, we cannot just immediately use the projection $\rho$ and study the pair $(\rho(X),\rho(\Xi))$. Indeed, if $p_1,p_2$ are two non-collinear points of $\rho(X)$, and $x_i,x'_i \in \rho^{-1}(p_i)$ for $i=1,2$, then $x_1$ and $x_2$ and also $x'_1$ and $x'_2$ are non-collinear points of $X$ for sure, but we do not even know whether $\rho([x_1,x_2])=\rho([x'_1,x'_2])$. Moreover, to establish the inverse image of a line of $\rho(X)$, we need to know more on the structure of $Y$, et cetera.
\end{rem}

\subsection{Mono-symplectic DSVs}
We  express the fact that $(X,Z,\Xi,\Theta)$ is mono-symplectic (i.e., $|\Theta|=0$) in terms of the vertices of members of $\Xi$.

\begin{lemma}\label{existencesupersymps}  The following are equivalent:
\begin{compactenum}[$(i)$]
\item $\Theta$ is empty;
\item each member of $\Xi$ has $Y$ as its vertex;
\item there is a member of $\Xi$ having $Y$ as its vertex.
\end{compactenum}
\end{lemma}
\begin{proof} Since $(ii) \Rightarrow (iii)$ is trivial, it suffices to show  $(i) \Rightarrow (ii)$ and  $(iii) \Rightarrow (i)$.

$(i) \Rightarrow (ii)$: Suppose that $\Theta$ is empty. Take any member $\xi \in \Xi$, say with vertex $V$, and let $x$ be one of its $X$-points. If $x$ would be collinear to some point $y \in Y\setminus V$ then, by Corollary~\ref{supersymp/plane}, $xy$ is contained in a member of $\Theta$ together with an $X$-line of $\xi$ through $x$, contradicting the assumption. So $Y_x=V$. If there would be a $z\in Z$ outside $Y_x$, then (S1) implies a member of $\Theta$ through $x$ and $z$, contradicting the assumption. Since $Y=\<Z\>$, we have $V=Y_x=Y$. As each member of $\Xi$ has a $v$-dimensional vertex, $(ii)$ follows. 
%Hence $Y_x=V$. Now, take  $x' \in X\setminus\{x\}$ arbitrary. Then (S1) and our assumption imply that $x$ and $x'$  are contained in a member of $\Xi$, whose vertex $V'$ is contained in $Y_x$ and hence coincides with it since $\dim(V)=\dim(V')=v$. As a consequence, all points of $X$ are collinear to $V$,  and therefore also all members of $\Xi$ have vertex $V$. We now claim that  this implies that $Y=V$. Indeed, if not, then since $\<Z\>=Y$, there is a point $z \in Z\setminus V$. But then, for any point $x\in X$ we have by (S1) that there is a member $\zeta \in \Xi \cup \Theta$ containing $x$ and $z$. Since $x$ is not collinear to $z$, this member should be an element of $\Theta$, a contradiction. Hence $Y=V$ indeed, showing $(ii)$.

$(iii) \Rightarrow (i)$:  Suppose that there is some $\xi \in \Xi$ having $Y$ as its vertex. In particular, all $X$-points of $\xi$ are collinear to $Y$. Let $x$ be a point of $X\setminus X(\xi)$. We show that $x \perp Y$ as well. By Lemma~\ref{convexclosure}, $X(\xi)$ contains a point $x'$ non-collinear to $x$ and hence $[x,x'] \in \Xi \cup \Theta$. If $[x,x'] \in \Xi$, then it has vertex  $Y$ (since $v= \dim(Y)$); if $[x,x'] \in \Theta$ then $Y([x,x'])$ should contain a point non-collinear to $x'$, a contradiction. Hence $Y$ is collinear to all points in $X$ and consequently, $\Theta$ is empty.
\end{proof}

Since we are dealing with mono-symplectic sets, $\Theta$ is empty. Moreover, by the previous lemma, all members of $\Xi$ have $Y=\<Z\>$ is their vertex. Hence, when considering the projection $\rho$ of $(X,Z,\Xi,\Theta)$ from $Y$, only $\rho(X)$ and $\rho(\Xi)$ carry information:

\begin{prop}\label{trivialcase} Let $(X,Z,\Xi,\Theta)$ be a mono-symplectic dual split Veronese set. Then the projection $(\rho(X),\rho(\Xi))$ is isomorphic to one of the following point-quadric varieties: a Segre variety $\mathcal{S}_{1,2}(\K)$ or  $\mathcal{S}_{2,2}(\K)$  $(r=1)$, a line Grassmannian  $\mathcal{G}_{5,2}(\K)$ or $\mathcal{G}_{6,2}(\K)$ $(r=2)$ or the variety $\mathcal{E}_{6,1}(\K)$ ($r=4$). 
\end{prop}
\begin{proof}
As $\Theta$ is empty by assumption, Lemma~\ref{existencesupersymps} implies that each member of $\Xi$ has $Y$ as its vertex. In particular,  for each point $x\in X$ holds that $x \perp Y$.  By Lemma~\ref{resY}, the projection $(\rho(X),\rho(\Xi))$ satisfies axioms (S1), (S2) and (S3) as well and as such, it is a \emph{Mazzocca-Melone set of split type $2r$}. The result follows from \cite{SVM}.
\end{proof}
\par\medskip

We have shown Main Result~\ref{main}$(i)$. Henceforth we assume that $(X,Z,\Xi,\Theta)$ is a {duo-symplectic} (pre-)DSV with the property that no point in $Y$ is collinear to all points of $X$ (cf.\ Lemma~\ref{resY}).

\subsection{Local properties}
Given $|\Theta| \geq 1$, one can show that each $X$-point belongs to a member of $\Theta$.

\begin{lemma}\label{xinSS}
For each $X$-point $x$, there is a point $z\in Z$ not collinear to $x$. In particular, $x$ is contained in at least one member of $\Theta$. 
\end{lemma}
\begin{proof}
Suppose for a contradiction that $x$ is collinear to each point of $Z$. Since $Y=\<Z\>$, Corollary~\ref{collY} implies that $x \perp Y$. By (S1), this means that $x$ is contained in some $\xi \in \Xi$, say with vertex $V$. By Lemma~\ref{existencesupersymps} and $|\Theta|\geq 1$,  there is a point $y \in Y\setminus V$. Corollary~\ref{supersymp/plane} implies that the singular line $xy$ is contained in some $\theta\in\Theta$ together with an $X$-line of $\xi$ through $x$. But then $Y(\theta)$ contains a point non-collinear to $x$ after all, a contradiction. We conclude that there is a point $z \in Z$ non-collinear to $x$, and then $x\in[x,z]\in\Theta$. 
\end{proof}

We take a closer look at the $X$-lines.
\begin{defi} \em
An $X$-line contained in $0$, $1$ or at least $2$ members of $\Theta$ is called a $0$-line, a $1$-line or a $2$-line, respectively. \end{defi}

 It turns out that the nature of an $X$-line $L$ can be expressed in terms of $Y_L$ and $Y_x$ with $x\in L$ (cf.\  Definition~\ref{YS} and Corollary~\ref{collY}). 

\begin{lemma}\label{lineSS}
Let $L$ be an $X$-line and $x$ any of its points. If $\theta \in \Theta$ contains $L$, then $Y_L \subseteq \theta$ is a hyperplane in  $\theta \cap Y_x$. Conversely, for each subspace $H \subseteq Y_x$ in which $Y_L$ is a hyperplane, there is a unique $\theta_{H,L} \in \Theta$ through $L$ with $\theta_{H,L} \cap Y_x = H$. Consequently:
\begin{compactenum}[$(i)$]
\item $L$ is a $0$-line if and only if $Y_x= Y_L$, in which case $Y_x=Y_{x'}$ for each $x' \in L$;
\item $L$ is a $1$-line if and only if $Y_L$ is a hyperplane of $Y_x$;
\item $L$ is a $2$-line if and only if $Y_L$ is strictly contained in a hyperplane of $Y_x$.
\end{compactenum}
\end{lemma}
\begin{proof}
Suppose that there is a point $y \in Y_x \setminus Y_L$ and put $H_y:=\<Y_L,y\>$. By Lemma~\ref{planes} and $y \notin Y_L$, $[xy,L]$ is the unique member of $\Theta$ containing $L$ and $y$. According to  Lemma~\ref{convexclosure}, $[xy,L]$ contains each $X$-point of $\<x,Y_L\>$, and hence also $Y_L$ and $H_y$. Inside the quadric $[xy,L]$, the subspace $Y_L \cap [xy,L]$ is a hyperplane of $Y_x \cap [xy,L]$, implying that $[xy,L] \cap Y_x = H_y$.   Therefore $\theta_{H_y,L}:=[xy,L]$ is the unique member of $\Theta$ through $L$ with $\theta_{H_y,L} \cap Y_x=H_y$. Conversely, each $\theta\in \Theta$ through $L$ arises as $\theta_{H,L}$ with $H=Y_x \cap \theta$.

We conclude that the number of members of $\Theta$ through $L$ depends on the number of subspaces of $Y_x$ containing $Y_L$ as a hyperplane. Note that, if $Y_L=Y_x$, so if there are no members of $\Theta$ through $L$, then, as $x\in L$ was arbitrary, $Y_{x'}=Y_L=Y_x$ for each  $x'\in L$. %By Lemma~\ref{Xlinevertex}, $L^\perp \cap Y(\theta_i)$ is exactly $\theta_1 \cap \theta_2$, $i=1,2$. If $L^\perp \cap Y$ were strictly bigger than $L^\perp \cap Y(\theta_i)$, then there is a point $z\in Z$ collinear to $L$ outside $\theta_i$, $i=1,2$. Take an arbitrary $X$-plane $\pi_i=\<x_i,L\>$ in $\theta_i$, with $x_i \in X(\theta_i)$. By Lemma~\ref{planes}, $z$ is collinear to the maximal singular subspace $\<L,L^\perp \cap Y(\theta_i)\>$ of $\theta_i$, so by Lemma~\ref{convexclosure}, $z$ is not collinear to $x_i$. Therefore, $[x_i,z]$ is a member of $\Theta$ intersecting $\theta_i$ in the $X$-plane $\pi_i$ while not coinciding with it (as $z\notin \theta_i$), contradicting the first paragraph. We conclude that, if $L$ is an $X$-line contained in at least two members of $\Theta$, then these members contain $L^\perp \cap Y$.
\end{proof}

\begin{cor}\label{noXplane}\label{Xlinevertex}
If $\theta \in \Theta$ and $\zeta \in \Xi \cup \Theta\setminus \{\theta\}$ share an $X$-line $L$, then $\zeta \cap \theta = \<L,L^\perp \cap Y(\zeta)\>$, i.e., $\zeta \cap \theta =\<L,Y(\zeta)\>$ if $\zeta \in \Xi$ and $\zeta \cap \theta = \<L,Y_L\>$ if $\zeta \in \Theta$.  Moreover, two members of $\Theta$ sharing an $X$-plane coincide.
\end{cor}
\begin{proof}
By Lemma~\ref{lineSS}, $\theta$ contains $\<L,Y_L\>$, so in particular it contains $\<L,L^\perp \cap Y(\zeta)\>$. If $\zeta \in \Xi$, then $L^\perp \cap Y(\zeta) = Y(\zeta)$; if $\zeta \in \Theta$ then $Y_L \subseteq \zeta$ too, and as $\<L,Y_L\>$ is a maximal singular subspace of both $\theta$  (because $\<L,L^\perp \cap Y(\theta)\>$ is maximal in $\theta$), likewise for $\zeta$, we get that $\theta \cap \zeta = \<L,Y_L\>$.

Two distinct members of $\Theta$ sharing an $X$-plane, in particular share an $X$-line $L$ and hence the maximal singular subspace $\<L,Y_L\>$, which contains no $X$-planes, a contradiction. \end{proof}

\begin{lemma}\label{collbetweenS}
Let $\zeta_1$ and $\zeta_2$ be two members of $\Xi \cup \Theta$, and put $S=\zeta_1 \cap \zeta_2$. Then two $X$-points $x_1,x_2$ in $\zeta_1\setminus S$ and $\zeta_2\setminus S$, respectively, with $x_1^\perp \cap S \neq x_2^\perp \cap S$ are not collinear.
\end{lemma}
\begin{proof}
Suppose for a contradiction that $x_1 \perp x_2$ while $x_1^\perp \cap S \neq x_2^\perp \cap S$. Then there is a point $p_1 \in S \cap x_1^\perp\setminus x_2^\perp$ and by Lemma~\ref{convexclosure}, $x_1 \in[x_2,p_1]=\zeta_2$, a contradiction. 
\end{proof}

\begin{lemma}\label{Hperp}
Suppose $\theta_1,\theta_2 \in \Theta$ are such that $\theta_1 \cap \theta_2$ is a maximal singular subspace of both $\theta_1$ and $\theta_2$ of the form $\<x,H\>$, with $x\in X$ and $H \subseteq Y$. For $i\in\{1,2\}$, let $L_i \in \theta_i$ be an $X$-line through $x$ and put $H_i :=L_i^\perp \cap H$. Then $v=v'+r'-2$ and
\begin{compactenum}[$(i)$]
\item if $H_1 \neq H_2$, then $[L_1,L_2]\in\Xi$ (and there are always $L_1$ and $L_2$ such that this occurs);
\item if $H_1 = H_2$, then $L_1 \perp L_2$ (and there are  $L_1$ and $L_2$ such that this occurs $\Leftrightarrow V_1=V_2$).
\end{compactenum}
\end{lemma}

\begin{proof} Let $V_1$ and $V_2$ denote the respective vertices of $\theta_1$ and $\theta_2$ and note that these belong to $x^\perp$ and hence to $H$.  

$(i)$ We first show that there always is a pair of $X$-lines $L_1$ and $L_2$ with $H_1 \neq H_2$. If not, then all pairs  $L_1,L_2$ are collinear to the same hyperplane of $H$, that hence coincides with both $V_1$ and $V_2$, whereas $\dim(H)=v'+r' \geq v'+2$.  

Take $L_1$ and $L_2$ with $H_1 \neq H_2$. By Lemma~\ref{collbetweenS}, $L_1$ and $L_2$ are not collinear. Suppose for a contradiction that $[L_1,L_2] \in \Theta$. By Corollary~\ref{Xlinevertex}, $H_1$ and $H_2$ belong to $[L_1,L_2]$. However, since $L_1$ and $H_2$ are not collinear, $[L_1,L_2]=\theta_1$,  contradicting $L_2 \subsetneq \theta_1$. So $[L_1,L_2]$ belongs to $\Xi$. Recall that the vertex $V$ of $[L_1,L_2]$ consists of all $Y$-points collinear with both $L_1$ and $L_2$ (cf.\ Lemma~\ref{convexclosure}). By Corollary~\ref{Xlinevertex}, $V \subseteq Y(\theta_1) \cap Y(\theta_2)=H$, so $V=H_1 \cap H_2$. In particular, $v=\dim(H)-2=v'+r'-2$.

$(ii)$ Next, suppose that $L_1$ and $L_2$ are such that $H_1 = H_2=:T$. Suppose for a contradiction that $L_1$ and $L_2$ are not collinear and put $\zeta:=[L_1,L_2] \in \Xi \cup \Theta$. If $\zeta \in \Xi$, then, as in the previous paragraph, we get that $T$ is its vertex, violating $\dim(T)=\dim(H)-1>v$. If $\zeta \in \Theta$, then $\<L_1,T\>$ and $\<L_2,T\>$ are two maximal singular subspaces of $\zeta$ by Corollary~\ref{Xlinevertex}. However, these two subspaces go through the submaximal subspace $\<T,x\>$ and contain $X$-lines through $x$, whereas $\zeta$ only contains one such subspace, a contradiction. We conclude that $L_1 \perp L_2$.

By definition, $Z(\theta_i)$ is the disjoint union of $V_i$ and some $r'$-dimensional subspace, say $R_i$ ($i=1,2$). Recall that $V_1 \cup V_2 \subseteq H$. Suppose first that $V_1 \neq V_2$. Then $V_1 = R_2 \cap H=R_2 \cap x^\perp$ and $V_2 = R_1 \cap H=R_1 \cap x^\perp$; in particular, $V_1 \cap V_2 = \emptyset$.  Each $X$-line $L_1$ in $\theta_1$ through $x$ is not collinear to $V_2$, whereas  each $X$-line $L_2$ in $\theta_2$ through $x$ is collinear to $V_2$. So if $V_1\neq V_2$, there are no $L_1,L_2$ with $H_1=H_2$. On the other hand, if $V_1=V_2$, then for each $X$-line $L_1$ in $\theta_1$ through $x$, there is a unique maximal singular subspace in $\theta_2$ through $\<x,H_1\>$ containing  $X$-lines through $x$, and for any such $X$-line $L_2$, clearly $H_2=H_1$.
\end{proof}

\subsection{Point-residues and the inductive approach}

 \begin{defi}\label{ptres} \em For each $x\in X$, we define the \emph{point residue $\Res_X(x):=(X_x,Z_x,\Xi_x,\Theta_x)$} as follows. Take any  hyperplane $H_x$ of $T_x$ containing $Y_x$ and not containing $x$. We let $X_x$ be the points in $H_x \cap  X$ on an $X$-line with  $x$ and $Z_x$  as $Y_x \cap Z$; furthermore, $\Xi_x$ is the set $\{T_x(\xi) \cap H_x \mid x\in \xi\in \Xi\}$ and $\Theta_x$ as $\{ T_x(\theta) \cap H_x \mid x\in\theta\in \Theta\}$. \end{defi}
 
Note that, for $\zeta \in \Xi \cup \Theta$ and $x\in X(\zeta)$, we have that $X_x(\zeta):=X_x \cap \zeta$ corresponds precisely to the point residue $\Res_{X(\zeta)}(x)$ of $X(\zeta)$ as a (degenerate) hyperbolic quadric. As such, it follows that the definition is independent of the choice of the hyperplane $H_x$.  Moreover, each member of $\Xi_x \cup \Theta_x$ contains a pair of non-collinear points of $X_x$ and hence corresponds to a unique member of $\Xi \cup \Theta$ through $x$. 
 
To show that $\Res_X(x)$ is a pre-DSV as well, we first need that $\<Z_x\>=\<x,Y_x\>$. 
 \begin{lemma}\label{resYx}
 For each $x\in X$, $\<Z_x\>=Y_x$.
 \end{lemma}
 \begin{proof}
 Recall that $Y_x$ is indeed a subspace by Corollary~\ref{collY}. Suppose for a contradiction that $\<Z_x\>$ is a strict subspace $S$ of $Y_x$. Let $\theta \in \Theta$ through $x$ be arbitrary (cf.\ Lemma~\ref{xinSS}). Take an $X$-line $L$ through $x$ in $\theta$. By Lemma~\ref{lineSS}$(i)$, $Y_L<Y_x$ and hence there is a point $y \in Y_x\setminus (Y_L \cup S)$. According to Lemma~\ref{planes}, $\theta':=[xy,L] \in \Theta$. The definition of $\theta'$ implies that $x^\perp \cap Y(\theta')=\<x^\perp \cap Z(\theta')\>$. As the former subspace contains $y$ and the latter is contained in  $S$, this contradicts $y \notin S$.
 \end{proof}
 
\begin{lemma}\label{res}
Let $(X,Z,\Xi,\Theta)$ be a duo-symplectic pre-DSV with parameters $(r,v,r',v')$ with $r \geq 2$. For each $x\in X$, $(X_x,Z_x,\Xi_x,\Theta_x)$ is a duo-symplectic pre-DSV in $\mathbb{P}^{N_x}(\K)$ with $N_x=\dim(T_x)-1$, with parameters $(r-1,v,r'-1,v')$. Furthermore, $(X_x,Z_x,\Xi_x,\Theta_x)$ contains no $2$-lines, and, if (S3) holds in $(X,Z,\Xi,\Theta)$, then $N_x \leq 2d-1$.
\end{lemma}
\begin{proof}
Recall that $T_x$ is generated by all singular lines through $x$. By definition of $X_x$ and $Y_x$, these lines are precisely the ($X$-)lines $\<x,x'\>$ where $x'$ is a point of $X_x$ and the lines $\<x,y\>$ where $y$ is a point of $Y_x$.  So $T_x=\<x,X_x,Y_x\>$. By Lemma~\ref{resYx} we obtain that  $X_x$ and $Z_x$ generate the projective space $H_x$ and hence $H_x\cong  \mathbb{P}^{N_x}(\K)$ with $N_x=\dim(T_x)-1$ (so if (S3) holds, then indeed $N_x\leq 2d-1$).

A straightforward verification tells us that each $\zeta \in \Xi \cup \Theta$ with $x\in \zeta$, intersects the sets $X_x$, $Y_x$ and $Z_x$ as described in Definition~\ref{svs}, and that the associated parameters are $(r-1,v,r'-1,v')$. 

Since (S2) holds in $(X,Z,\Xi,\Theta)$, it also holds in $\Res_X(x)$. We show that this is also the case for (S1).  Two points $p_1,p_2$ in $X_x\cup Z_x$ correspond to two singular lines $L_1$ and $L_2$ through $x$. 
 By Lemma~\ref{planes}   either $[L_1,L_2] \in \Xi \cup \Theta$ and $x\in [L_1,L_2]$ by Lemma~\ref{convexclosure}, in which case (S1) follows; or  $\<L_1,L_2\>$ is a singular plane $\pi$. In the latter case, (S1) implies the existence of a member $\zeta$ of $\Xi \cup \Theta$ containing $L_1$. If $\zeta$ also contains $L_2$, we are good, so suppose it does not. Since $r,r' \geq 2$, there is a singular plane $\pi'$ in $\zeta$ through $L$ not collinear to $L_2$ (cf. Lemma~\ref{convexclosure}). By Lemma~\ref{kspaces}, $\pi$ and $\pi'$ determine a unique member of  $\Xi \cup \Theta$ containing $L_1$ and $L_2$, so also in this case (S1) follows. 

By Corollary~\ref{noXplane}, $(X,Z,\Xi,\Theta)$ has no $X$-planes contained in multiple members of $\Theta$. Consequently, $(X_x,Z_x,\Xi_x,\Theta_x)$ contains no $2$-lines.
\end{proof}

Recall that we assume that there are no $Y$-points collinear to all points of $X$ (cf.\ Lemma~\ref{resY}). This is a residual property:

\begin{lemma}\label{projV}
If, for some $x\in X$, $y\in Y_x$ is collinear to all points of $X_x$,  then $y$ is collinear to all points of $Y$.\end{lemma}
\begin{proof}
Suppose that $y\in Y_x$ collinear to all points of $X_x$. First note that this implies, since $x \perp y$, that $y$ is collinear to all points in $X$ collinear to $x$. 
Take any $x' \in X$ not collinear to $x$. Then $[x,x'] \in \Xi \cup \Theta$. Then $[x,x'] \cap X_x$ contains two non-collinear points, which by assumption are collinear to $y$, so the vertex of $[x,x']$ contains $y$ and hence $x' \perp y$. We conclude that all $X$-points are collinear to $y$.
\end{proof}

In the next two sections, we will treat the ``induction hypothesis'': the cases where $r=1$, in which case we will deal with \emph{pre}-DSVs, unless we are sure that they cannot occur as a residue (for example, when there are $2$-lines), then we may also use (S3). In the final section we treat the general case with $r \geq 2$. In all cases, Lemmas~\ref{resY} and~\ref{projV} allow us to assume that no point of $Y$ is collinear to all points of $X$.

\section{The dual Segre varieties}\label{r=1,2SS}

In this section we suppose that $(X,Z,\Xi,\Theta)$ is a duo-symplectic pre-DSV with $r=1$ such there are at least two members of $\Theta$ through each $X$-point and such that no $Y$-point is collinear to all points of $X$.

  We will show that, as soon that there is a $1$-line, all $X$-lines are $1$-lines (cf.\ Proposition~\ref{1':xlines}), but first we exclude the possibility that there are no $1$-lines (cf.\ Proposition~\ref{1':nonex}).  

\subsection{The occurrence of $1$-lines}

\begin{prop}\label{1':nonex}
Let $(X,Z,\Xi,\Theta)$ be a duo-symplectic pre-DSV with $r=1$ such there are at least two members of $\Theta$ through each $X$-point. Then there is a $1$-line.
\end{prop}

This proposition is the key to the classification. Its proof is rather long and hence divided into several lemmas over a couple of subsections.

Suppose for a contradiction that $(X,Z,\Xi,\Theta)$ has no $1$-lines.

\subsubsection{The values $v'$ and $v$ if there are no $1$-lines}
\begin{lemma}\label{S3}\label{2L}
There is at least one $2$-line through each $X$-point. In particular,  (S3) holds.
\end{lemma}
\begin{proof}
By assumption, there is at least one member $\theta\in\Theta$ containing $x$. As all $X$-lines in $\theta$ are $2$-lines, this shows the first assertion. Since there are $2$-lines, Lemma~\ref{res} assures that $(X,Z,\Xi,\Theta)$ cannot occur as a point-residue of a DSV, and therefore we may assume that (S3) holds. 
\end{proof}

The next aim is to show that $v'=-1$ (given the assumption that no $Y$-point is collinear to all $X$-points, cf.\ Lemma~\ref{resY}).

\begin{lemma}\label{1':r'=2}
Any two members of $\Theta$ sharing a $2$-line have the same vertex and $r'=2$.
\end{lemma}

\begin{proof} Let $L$ be any 2-line and take two members $\theta_1,\theta_2 \in \Theta$ containing $L$. Recall that $\theta_1 \cap \theta_2 = \<L,Y_L\>$ by Corollary~\ref{Xlinevertex}. Take $X$-planes $\pi_i$ through $L$ in $\theta_i$ for $i=1,2$. Then $\pi_1$ and $\pi_2$ are collinear, for they can neither be contained in a member of $\Xi$ (because $r=1$), nor in a member of $\Theta$  (by Corollary~\ref{noXplane}). Since $\pi_1^\perp \cap\theta_2$ is singular by Lemma~\ref{convexclosure}, this implies that there cannot be two non-collinear $X$-planes  in $\theta_2$ through $L$. As such, $r'=2$. 

Note that $r'=2$ implies that $\pi_i^\perp \cap Y_L$ is precisely the vertex $V_i$ of $\theta_i$, $i=1,2$. As $\pi_1$ and $\pi_2$ are collinear, we get that $\pi_1^\perp \cap Y_L = \pi_2^\perp \cap Y_L$ (cf.\ Lemma~\ref{collbetweenS}), and hence $V_1 = V_2$. 
%Since $r'=2$, $R_i$ contains a unique point $r_i$ collinear to $L$. So $Y_L = \<V_i,r_i\>$. If $V_1 \neq V_2$, then necessarily $V_1 = r_2$ and $V_2=r_1$. In particular, $Y_L$ is the line $\<V_1,V_2\>$. So, an $X$-plane through $L$ in $\theta_i$ is collinear to precisely the point $V_i$ of $Y_L$ (since $\<L,V_1,V_2\>=\<L,Y_L\>$ is a maximal singular subspace of $\theta_i$), contradicting the first assertion of this paragraph.  We conclude that $V_1 = V_2$.
\end{proof}

Combining the two previous lemmas, we can easily put an upper bound on $v'$. Afterwards we show that $v'=-1$ (in view of Lemma~\ref{resY}), which requires some more work.

\begin{lemma}\label{1':v,v'}
We have $v' \leq 2v-1$.
\end{lemma}
\begin{proof}
Let $x\in X$ be arbitrary. By Lemma~\ref{S3}, there are $\xi_1,\xi_2 \in \Xi$ through $x$ with ${T_x=\<T_x(\xi_1),T_x(\xi_2)\>}$. Since $r=1$, we get $\dim(T_x) \leq 2(v+4)=2v+8$. On the other hand, Lemma~\ref{S3} also implies that there are $\theta_1,\theta_2 \in \Theta$ through $x$ sharing a $2$-line $L$. By Corollary~\ref{Xlinevertex},  $\theta_1 \cap \theta_2 = \<L,Y_L\>$ and hence, recalling $r'=2$ (cf.\ Lemma~\ref{1':r'=2}), we get $\dim(T_x) \geq 2(v'+6) - (v' + 3) = v'+9$. Combining these inequalities yields $v' \leq 2v-1$. 
\end{proof}

\begin{lemma}\label{1':v'=-1}
All members of $\Theta$ have the same vertex $V$. As $V \perp x$ for all $x\in X$,  $v'=-1$.
\end{lemma}

\begin{proof}
If all members of $\Theta$ share a vertex $V$, then through each $x\in X$  there is a member of $\Theta$, so $x \perp V$. Since we assume that no $Y$-point is collinear to all points of $X$, we get $v'=-1$. Suppose for a contradiction that $v' \geq 0$. 

 Let $x\in X$ be arbitrary and take any $\theta_1\in \Theta$ through $x$, say with vertex $V_1$.  By Lemma~\ref{projV}, there is a singular line through $x$ not collinear to $V_1$. Clearly, each $0$-line $L$ through $x$ is collinear to $V_1$ since $V_1 \subseteq Y_x=Y_L$ by Lemma~\ref{lineSS}$(i)$; and a line through $x$ with a point in $Y_x$ is collinear to $V_1$ too by Corollary~\ref{collY}. Hence there is a $2$-line $M_2$ through $x$ with $M_2$ not collinear to $V_1$. Take any $\theta_2 \in \Theta$ containing $M_2$ and denote its vertex by $V_2$. Clearly $V_1 \neq V_2$. Lemma~\ref{1':r'=2} says that $\theta_1 \cap \theta_2=\<x,H\>$ with $H \subseteq Y_x$ and implies that no member of $\Theta$ meets both $\theta_1$ and $\theta_2$ in respective $X$-lines.

\textit{Claim 1:  Each pair of $X$-lines $L_1$ and $L_2$ through $x$ in $\theta_1\setminus\theta_2$ and $\theta_2\setminus\theta_1$, respectively, is non-collinear and $[L_1,L_2]\in\Xi$ has vertex $V(L_1,L_2)=Y_{L_1} \cap Y_{L_2}\subseteq H$.}\\
First note that, by the previous paragraph, no member of $\Theta$ contains $L_1\cup L_2$. Suppose for a contradiction that $\<L_1,L_2\>$ is a singular plane $\pi$. As $L_2^\perp \cap \theta_1$ is singular (cf.\ Lemma~\ref{convexclosure}), there is a singular plane $\pi'$ in $\theta_1$ through $L_1$ not collinear to $\pi$. Since $r=1$, there is a member of $\Theta$ containing $\pi \cup \pi'$, contradicting the beginning of this paragraph. So $[L_1,L_2] \in \Xi$ indeed, say with vertex $V(L_1,L_2)$. Corollary~\ref{Xlinevertex} implies that $V(L_1,L_2)$ belongs to $Y(\theta_1) \cap Y(\theta_2)=H$, and then it follows by Lemma~\ref{convexclosure} that $V(L_1,L_2)=Y_{L_1} \cap Y_{L_2}\subseteq H$. This shows the claim. Observe that $H$ is necessarily non-empty, as $v \geq 0$.

%\textit{Claim 2: There are no $0$-lines.}\\ By our choice of $x\in X$, it suffices to show that there are no $0$-lines through $x$. So suppose for a contradiction  that $K$ is a $0$-line through $x$. Then of course $L_1$ and $K$ are not contained in a member of $\Theta$, and the same arguments as used in showing Claim 1 show that $L_1$ and $K$ are not collinear either and, as such, they are contained in a member of $\Xi$, say with vertex $V$. Then $V=Y_{L_1} \cap Y_K$, which is just $Y_{L_1}$ as $Y_{L_1}\subseteq Y_K=Y_x$. In particular, $v=\dim(Y_{L_1})=v'+1$  (recall $Y_{L_1}=L_1^\perp \cap Y(\theta_1)$ and $r'=2$). But then, the vertex $V(L_1,L_2)$ of $\xi(L_1,L_2)$ has to coincide with $Y_{L_1}$ too, which is not possible as $Y_{L_1}$ contains $V_1$ and $L_2 \perp V(L_1,L_2)$, whereas we assume that $L_2$ is not collinear to $V_1$. This shows the claim.

By definition, $Y(\theta_i) \cap Z$ is the disjoint union of $V_i$ and some plane $R_i$ (as $r'=2$) and $\<V_i,R_i\>=Y(\theta_i)$ for $i=1,2$.  So $H=Y(\theta_1) \cap Y(\theta_2)$ also contains two (possibly empty) disjoint subspaces in $Z$ (not necessarily spanning $H$ though), and therefore we either have $(V_1\cap H,R_1 \cap H)=(V_2 \cap H,R_2\cap H)$ or $(V_1 \cap H,R_1 \cap H)=(R_2 \cap H,V_2 \cap H)$. `

\textit{Claim 2: $V_i \cap H$ is non-empty for $i=1,2$.}\\
Suppose $V_1 \cap H$ is empty (the case where $V_2 \cap H$ is the same). Then $H$, being contained in $Y(\theta_1) \cap x^\perp$ is at most a line. Let $L_1$ and $L_2$ be $X$-lines through $x$ in $\theta_1 \setminus \theta_2$ and $\theta_2\setminus \theta_1$. By Claim 1, $V(L_1,L_2)=Y_{L_1} \cap Y_{L_2} \subseteq H$. Note that $Y_{L_1} \cap H$ is at most a point, as $\<L_1,H\>$ is skew from $V_1$ and $r'=2$. We conclude that $v=0$. However, Lemma~\ref{1':v,v'} then yields $v'=-1$, contradicting the assumption. The claim follows.

\textit{Claim 3: $(V_1 \cap H,R_1 \cap H)=(R_2 \cap H,V_2 \cap H)$ and $R_i \cap H  = x^\perp \cap R_i$ for $i=1,2$.}\\
Denote by $Z_1$ the unique subspace of $Z(\theta_i) \cap x^\perp$ through $H \cap V_2$ (which is non-empty by Claim 2). Note that, if $(V_1 \cap H,R_1 \cap H)=(R_2 \cap H,V_2 \cap H)$, then $Z_1 \setminus H$ always contains a point $z$ since  $V_1 \cap H \subsetneq V_1$ (otherwise $V_1 = V_2$), and if $(V_1 \cap H,R_1 \cap H)=(R_2 \cap H,V_2 \cap H)$, then $Z_1\setminus H$ contains a point $z$ if  $R_1 \cap H  \subsetneq x^\perp \cap R_1$.

By Lemma~\ref{convexclosure}, $z^\perp \cap X(\theta_2)$ contains no two non-collinear points, and hence there is an $X$-line $L_2$ in $\theta_2 \setminus \theta_1$ through $x$ not collinear to $z$. As such, $\theta'_2:=[zx,L_2] \in \Theta$. By Lemma~\ref{1':r'=2}, $\theta'_2$ has vertex $V_2$; and Lemma~\ref{convexclosure} implies that $\theta'_2$ contains $\<x,H \cap V_2\>$ (since $L_2 \perp \<x,H \cap V_2\> \perp xz$). 
Inside $Y(\theta'_2)$,  $z \notin V_2$ implies $z\in R'_2$ (using similar notation as above) and hence $\<z,H \cap V_2\> \cap Z = z \cup (H \cap V_2)$, whereas in $Y(\theta_1)$ we have that $\<z,H \cap V_2\> \subseteq Z_1 \subseteq Z$. This is only possible  if $V_2 \cap H$ is non-empty, contradicting Claim 2. We conclude that $(V_1 \cap H,R_1 \cap H)=(R_2 \cap H,V_2 \cap H)$ and $R_1 \cap H  = x^\perp \cap R_1$. Changing the roles of $\theta_1$ and $\theta_2$, the claim follows.

Note that Claim 1 and 3 imply that, for $X$-lines $L_1,L_2$ in $\theta_1$ and $\theta_2$, respectively, the vertex $V(L_1,L_2)$ of $[L_1,L_2]$ is a line joining the points $z_1 = R_1 \cap L_1^\perp$ and $z_2 = R_2 \cap L_2^\perp$, so $v=1$. By Lemma~\ref{1':v,v'}, this gives us $v' \leq 1$, so $V_j$ coincides with the line $R_i \cap x^\perp$ for $\{i,j\}=\{1,2\}$. Consequently, $H=\<V_1,V_2\>=\theta_i \cap x^\perp$, for $i=1,2$.

\textit{Claim 4: $Y_x=\<V_1,V_2\>$.}\\
Take any $y\in Y_x$. Up to renumbering, we have $y\notin V_1$.  Lemma~\ref{convexclosure} yields an $X$-line $L_1$ in $\theta_1$ through $x$ not collinear to $xy$, and hence $\theta'_1:=[L_1,xy] \in \Theta$. By Lemma~\ref{1':r'=2}, $\theta'_1$ has vertex $V_1$ too, and hence $\theta'_1$ plays the same role with respect to $\theta_2$ as $\theta_1$. By the foregoing, this means that $y\in Y(\theta'_1) \cap x^\perp=\<V_1,V_2\>$, so $Y_x= \<V_1,V_2\>$ indeed.

As $Y_x=\<V_1,V_2\>$, we have for any $X$-line $L_1$ in $\theta_1$ through $x$ that $Y_{L_1}$ is a hyperplane of $Y_x$. By Lemma~\ref{lineSS}, we get that $L_1$ is a 1-line, a contradiction.
\end{proof}

An immediate consequence of this is the following. 

\begin{cor}\label{1':y_L} For each $2$-line $L$, $Y_L$ is a point belonging to $Z$.
\end{cor}
\begin{proof}
If $\theta_L \in \Theta$ contains $L$, then $Y_L = L^\perp \cap Y(\theta_L)$ (cf.\ Lemma~\ref{lineSS}). Since $v'=-1$ and $r'=2$ it follows that $Y_L$ is a point, which belongs $Z$. 
\end{proof}

\subsubsection{Non-compatible structures if there are no $1$-lines}

We deepen the connections between the  $2$-lines $L$ and the points $Y_L$ and show that this leads to non-compatible structures.

\begin{lemma}\label{1':0and2lines}
Let $L$ and $M$ be two intersecting $2$-lines. Then $L \cup M$ belongs to a member of $\Theta$. Moreover, $Y_L = Y_M$ if and only if $\<L,M\>$ is singular and contains a  point of~$Y$. 
\end{lemma}
\begin{proof}
Suppose for a contradiction that  $L$ and $M$ belong to some $\xi \in \Xi$. Since $r=1$, $L$ and $M$ are not collinear and hence  $Y(\xi)=Y_L \cap Y_M$. As $v \geq 0$, we get $Y_L=Y_M=:y$. By Lemma~\ref{lineSS}, $Y_x$ contains a point $y' \neq y$. Then $\theta_L:=[L,xy'], \theta_M:=[M,xy'] \in \Theta$ and  $\theta_L \cap \theta_M=\<x,y,y'\>$ is a maximal singular subspace of both $\theta_L$ and $\theta_M$. As such, we can apply Lemma~\ref{Hperp}, which says that $L \perp M$, a contradiction. We conclude that, if $L$ and $M$ are not collinear, then $[L,M]\in \Theta$. 

Suppose that $\<L,M\>$ is a singular plane $\pi$. Take a member $\theta_L \in \Theta$ through $L$. If $\pi$ has a point $y\in Y$, then $y=Y_L=Y_M$, so $\pi=\<L,Y_L\>\subseteq \theta_L$ by Corollary~\ref{1':y_L} and Lemma~\ref{lineSS}. So let $\pi$ be an $X$-plane.  If $\pi$ were collinear to $Y_L$, then $\pi$ is not collinear to the unique $X$-plane $\pi'$ in $\theta_L$ through $L$, implying that $\pi$ and $\pi'$ determine an element of $\Xi$ or $\Theta$, which however violates $r=1$ and Corollary~\ref{noXplane}, respectively. So $\pi$ is not collinear to $Y_L$ and hence $\pi$ and $\<L,y_L\>$ determine a member of $\Theta$ containing $\pi$.

For two intersecting $X$-lines inside a member of  $\Theta$, it is clear that $y_L=y_M$ if and only if $\<L,M\>$ is a singular plane containing $Y_L = Y_M$.  The lemma follows. 
\end{proof}

Every degenerate hyperbolic quadric contains two natural systems of maximal singular subspaces (that all contain the vertex), called generators. Each such system is called a \emph{regulus}.

\begin{lemma}\label{1':symps} 
For each $\xi \in \Xi$, the set of $2$-lines contained in $\xi$ are all $X$-lines contained in the members of a fixed regulus of $X(\xi)$. Moreover, $v=0$. 
\end{lemma}
\begin{proof}
Take $x\in X(\xi)$ arbitrary and let $L_1$ and $L_2$ be $X$-lines through $x$ belonging to different reguli of $X(\xi)$. By Lemma~\ref{1':0and2lines}, at least one of $L_1, L_2$ is a $0$-line. Suppose for a contradiction that both are  $0$-lines. Then $\xi$ has vertex $Y_{L_1} \cap Y_{L_2} = Y_x$ by Lemma~\ref{lineSS}$(i)$. By Lemma~\ref{1':r'=2}, there is a $2$-line $L$ through $x$, which is collinear to at most one of $L_1,L_2$, say $L$ and $L_1$ are not collinear. By definition, $0$-lines are contained in no member of $\Theta$, so we get that $\xi_1:=[L,L_1] \in \Xi$. Then $\xi$ has $Y_L$ as its vertex, but  as $L$ is a $2$-line, $\dim(Y_L)<\dim(Y_x)=v$, a contradiction. Hence two $X$-lines of $X(\xi)$ belonging to members of different reguli, have different types, from which the first assertion follows. If $L$ is one of the $2$-lines of $X(\xi)$, we get that $Y_L$ is the vertex of $\xi$ and as such $v=0$.
\end{proof}

\begin{cor}\label{1':Y=Z}
The set $Z$ coincides with $Y$.
\end{cor}
\begin{proof}
Let $z_1$ and $z_2$ be any two points of $Z$. By (S1), there is a $\zeta \in \Xi \cup \Theta$ through them. Since $v=0$,  $\zeta \in \Theta$ and in $\zeta$,  $z_1z_2 \subseteq Z$ because $v'=-1$. We get $Z=\<Z\>=Y$.
\end{proof}

We  make use of the maps $\rho$ and $\chi$, as defined in Definitions~\ref{RHO}~and~\ref{CHI}. Note that, for each $X$-line $L$, $\rho(L)$ is a line, for $X$-lines have no point in $Y$. 

\begin{lemma}\label{1':inj2lines}
Let $L$ and $M$  be $X$-lines with $\rho(L)=\rho(M)$.  If $L$ is a $2$-line, then so is $M$, and  $Y_L = Y_M$.
\end{lemma}
\begin{proof}
As $\rho(L)=\rho(M)$, the map  taking a point $x$ on $L$ to the unique point  $\overline{x}$ on $M$ with $\rho(\overline{x})=\rho(x)$ is a bijection.  If $x =\overline{x}$ for some $x\in L$, i.e., if $L$ and $M$ share a point, then either $L=M$, or $\<M,L\>$ contains a point in $Y$ and hence coincides with the plane $\<L,Y_L\>$. In both cases, the assertion follows.

So suppose $x \neq \overline{x}$ for all $x\in L$, in particular, $\<L,M\>$ is a $3$-space.  By Lemma~\ref{proprhochi}, $x\overline{x}$ is a singular line with a unique point $y_x$ in $Y$.  Put $K:=\{y_x \mid x \in L\}$. If $L \perp M$, then for each $x\in L$, we have that $y_x \in \<x,M\>$ and hence $y_x=Y_M$, contradicting the fact that $\<L,M\>$ is a $3$-space. So $L$ contains a point $x_1$ not collinear to $M$.  Let $x_2$ be a point of $L\setminus\{x_1\}$, and note that $\overline{x}_1 \neq \overline{x}_2$. Then $x_1$ and $\overline{x}_2$ are not collinear and therefore $\zeta:=[x_1,x_2]\in \Xi \cup \Theta$. Since $x_2,\overline{x}_1 \in x_1^\perp \cap \overline{x}_2^\perp$, $\zeta$ contains $K \cup L \cup M$. Hence, as $y_{x_1} \neq y_{x_2}$ (otherwise $L$ and $M$ intersect) and $v=0$,  we get that $\zeta \in \Theta$. This means that $M$ is a $2$-line too. In the quadric $XY(\theta)$, we see that $K$ is a line (inside $Y(\theta)$) and that the point $Y_L$, being collinear to both $L$ and $K$, is also collinear to $M$, so $Y_L=Y_M$ indeed. 
\end{proof}
\par\bigskip

By the previous lemma, it makes sense to keep speaking about $0$-lines and $2$-lines in $\rho(X)$ and of the unique point $y_L$ of $Y$ collinear to such a $2$-line $L$ in $\rho(X)$. We do not claim that each line in $\rho(X)$ is the image of an $X$-line; this does not matter.

\begin{lemma}\label{1':pi}
Let $L$ and $M$ be $2$-lines such that $\rho(L)\cap \rho(M)$ is a point $p$. Then each line in the plane $\pi:=\<\rho(L),\rho(M)\>$  is a $2$-line, $Y_L \neq Y_M$ and the set of lines $K$ through $p$ in $\pi$ and the set of points on $\<Y_L,Y_M\>$ are in bijective correspondence by the map $K \mapsto Y_K$. \end{lemma}
\begin{proof}
Since $p \in \rho(L) \cap \rho(M)$, there are points $p_L$ and $p_M$ on $L$ and $M$, respectively, with $\rho(p_L)=\rho(p_M)=p$. Suppose that $p_L \neq p_M$. Then the line $\<p_L,p_M\>$ contains a point $y \in Y$ by Lemma~\ref{1:inj}. Let $p'_L$ be a point on $L\setminus \{p_L\}$. Either $y=Y_L$ or $y$ is not collinear to $L$; anyhow, in both cases there is a member $\theta \in \Theta$ containing the lines $L$ and $\<p_L,p_M\>$. In $\theta$, there is a line $L'$ through $p_M$ with $\rho(L)=\rho(L')$ (and hence also $Y_L=Y_{L'}$ by Lemma~\ref{1':inj2lines}). Replacing $L$ by $L'$, we may assume that $p_L=p_M=:x$.

 By Lemma~\ref{1':0and2lines}, $L$ and $M$ are contained in some $\theta \in \Theta$. As such, the lines $\rho(L)$ and $\rho(M)$ span the singular plane $\rho(\theta)$ (cf.\ Lemma~\ref{1:inj}) and each line in this plane is reached by some $X$-line in $\theta$. This shows the first assertion. By the same lemma, $Y_M \neq Y_L$, as otherwise $\rho(L)=\rho(M)$.  In the quadric $XY(\theta)$, collinearity gives a bijective correspondence between the $X$-lines $K$ through $x$ in $\<L,M\>$ and the points on $Y_K$ on $\<Y_L,Y_M\>$, and since all $X$-lines in $\rho^{-1}(\rho(K))$ correspond to the same point $Y_K$ by Lemma~\ref{1':inj2lines}, this gives us the required bijective correspondence.
\end{proof}

\par\bigskip
Take a connected component $\Pi$ of $\rho(X)$ with respect to $2$-lines intersecting each other in points and let $Y_\Pi$ be the subset of $Y$ consisting of all points of $\{Y_L \mid \rho(L) \text{ a line of } \Pi\}$. 

\begin{lemma}\label{1':bij}
The above defined connected component $\Pi$ is a singular subspace of $\rho(X)$, whose lines are in bijective correspondence to the points of $Y_\Pi$. 
\end{lemma}
\begin{proof}
Let $p$ be any point of $\Pi$. We claim that all other points of $\Pi$ are on a $2$-line with $p$. If not, then there are points $p',p''\in\Pi$ such that $pp'$ and $p'p''$ are $2$-lines, and $pp''$ is not. But then,  looking in $p'$, it follows from Lemma~\ref{1':pi} that $\<p,p',p''\>$ is a singular plane all of whose lines are $2$-lines and hence $p$ and $p'$ are on a $2$-line after all. So $\Pi$ is indeed a singular subspace of $\rho(X)$. 

Suppose that there are two distinct $2$-lines $\rho(L)$ and $\rho(M)$ in $\Pi$ with $Y_L = Y_M$. It follows from Lemma~\ref{1':pi} that $\rho(L)$ and $\rho(M)$ do not share a point.  Let $\rho(K)$ be a $2$-line in $\Pi$ joining a point of $\rho(L)$ and a point of $\rho(M)$. Then $Y_L=Y_M$ is collinear to two distinct points of $K$ (those corresponding to the intersection points $\rho(K) \cap \rho(L)$ and $\rho(K) \cap \rho(M)$) and hence $Y_K = Y_L=Y_M$, contradicting Lemma~\ref{1':pi}. As by definition, each line of $\Pi$ corresponds to a unique point of $Y_\Pi$, this shows the lemma.
\end{proof}

\begin{lemma}\label{1':Y_Pi}
Let $L$ and $M$ be $2$-lines with $Y_L, Y_M\in Y_\Pi$ distinct. Then $\rho(L)\cap\rho(M) \neq \emptyset$.
\end{lemma}

\begin{proof}
We claim that there is a point $x\in X$ with $\rho(x)\in \Pi$ such that $x$ is collinear to $Y_LY_M$. Since $Y=Z$ (cf.\ Lemma~\ref{1':Y=Z}) and $v=0$, there is a $\theta\in\Theta$ through $Y_L$ and $Y_M$ by (S1). In $\theta$, there is a point $x \in X$ collinear to the line $y_Ly_M$. If $\rho(x) \in \Pi$ we are good, so suppose it is not. Let $x' \in X$ be any point with $\rho(x') \in \Pi$. Suppose first that $x$ and $x'$ are collinear. As $\rho(x) \neq \rho(x')$, the line $xx'$ is an $X$-line (cf.\ Lemma~\ref{1:inj}). From $\rho(x)\notin \Pi$, it follows that $xx'$ is a $0$-line, and hence $Y_{x'} = Y_x$ (cf.\ Lemma~\ref{lineSS}$(i)$). In particular, $x'$ is also collinear to $Y_LY_M$ and hence is a valid choice. Secondly, suppose $x'$ and $x$ are not collinear.
Then they are contained in a member $\zeta$ of $\Xi$ or of $\Theta$. In the first case, Lemma~\ref{1':symps} implies that there is a $2$-line through $x'$ in $\zeta$ meeting a $0$-line through $x$, say in a point $x''$. Then $x''$ is a good choice: it is collinear to $Y_LY_M$  as it is on a $0$-line with $x$, and $\rho(x'')\in\Pi$ since it is on a $2$-line with $x'$. If $\zeta \in \Theta$, then since $x$ and $x'$ are joined by two intersecting $2$-lines in $\theta$, we obtain $\rho(x) \in \Pi$, a contradiction. The claim follows.

So, let $x\in X$ be a point collinear to $Y_LY_M$ with $p:=\rho(x) \in \Pi$. By Lemma~\ref{1':bij}, $\rho(L)$ and $\rho(M)$ are the unique respective lines in $\Pi$ collinear to $Y_L$ and $Y_M$. Let $p'$ be any point of $\rho(L)$, distinct from $p$ if $p$ were on it. Then $p$ and $p'$ are on a $2$-line $\rho(K)$ of $\Pi$  by Lemma~\ref{1':bij}. Since both $p$ and $p'$ are collinear to $Y_L$, so is $K$. Consequently, $\rho(K)=\rho(L)$ since this was the unique line in $\Pi$ collinear to $Y_L$, so $p \in \rho(L)$. Likewise, $\rho(M)$ contains $p$ and hence $\rho(L)$ and $\rho(M)$ intersect in $p$. 
\end{proof}

Finally we can  show that there have to be $1$-lines.

\textbf{Proof of Proposition~\ref{1':nonex}} \,
Let $L$ be any $2$-line and let $\theta_1,\theta_2$ be two elements of $\Theta$ containing $L$. Let $\pi_1$ and $\pi_2$ be the unique $X$-planes through $L$ in $\theta_1$ and $\theta_2$, respectively. Then, as noted before, $\pi_1$ and $\pi_2$ span a singular $3$-space $S$ (as they cannot be contained in a member of $\Xi$ nor of $\Theta$). If $S \cap Y$ were non-empty, then $S \cap Y$ is a point $y$ (since the planes $\pi_1$ and $\pi_2$ are $X$-planes).  But then $y \perp L$ and hence $y=Y_L$. Consequently $\pi_2 \subseteq \<\pi_1,Y_L\> \subseteq \theta_1$, a contradiction. We conclude that $\rho(S)$ is $3$-dimensional. Moreover, since each line $K$ in $S$ is contained in a plane together with two $2$-lines $M_1$ and $M_2$ of $\pi_1 \cup \pi_2$, Lemma~\ref{1':0and2lines} implies that $K$ belongs to a member of $\Theta$ containing $M_1$ and $M_2$ and hence is a $2$-line as well.   This however implies that the connected component $\Pi$ of $\rho(X)$ containing $\rho(S)$ contains a pair of disjoint $2$-lines, contradicting Lemma~\ref{1':Y_Pi}. We conclude that the assumption that there are no $1$-lines must be false. \qed

\subsubsection{All $X$-lines have to be $1$-lines}
The next goal is to show that \emph{all} $X$-lines are $1$-lines. We need a couple of lemmas for this.

\begin{prop}\label{1':xlines}
All $X$-lines are $1$-lines.
\end{prop}

\begin{lemma}\label{1':xperp}
Suppose $x\in X$ is on a $1$-line $L$. Then for each $\theta \in \Theta$ through $x$,  $Y_x \subseteq Y(\theta)$.\end{lemma}
\begin{proof}
Let $\theta_L$ be the unique member of $\Theta$ containing $L$. By Lemma~\ref{lineSS}$(ii)$, $Y_x=x^\perp \cap Y(\theta_L)$. In particular, $\dim(Y_x)=r'+v'$ and hence, for each $\theta \in \Theta$ containing $x$ we obtain that $x^\perp \cap Y(\theta)$, which also has dimension $r'+v'$, coincides with $Y_x$. 
\end{proof}

Recall that we assume that there are at least two members of $\Theta$ through any point of $X$, 

\begin{lemma}\label{1':2SS}
Let $\theta^x_1$ and $\theta^x_2$ be two members of $\Theta$ through a point $x$ which is on a $1$-line. Let $V^x_1$ and $V^x_2$ denote their respective vertices, and suppose $Z(\theta^x_i)=V^x_i \cup R^x_i$ for some $r'$-space $R^x_i$, $i=1,2$. Then  $V^x_1=R^x_2 \cap Y_x$ and $V^x_2=R^x_1 \cap Y_x$ are disjoint and generate $Y_x$. In particular, $v'=r'-1$ and there are exactly two members of $\Theta$ through $x$.
\end{lemma}
\begin{proof}
By Lemma~\ref{1':xperp}, $\theta^x_1$ and $\theta^x_2$ both contain the maximal singular subspace $\<x,Y_x\>$ and hence $\theta^x_1 \cap \theta^x_2 = \<x,Y_x\>$. Take $X$-lines $L_1$ and $L_2$ through $x$ in $\theta^{x}_1$ and $\theta^{x}_2$, respectively. By Lemma~\ref{Hperp}, $L_1$ and $L_2$ are collinear if and only if $H_1:=L_1^\perp \cap Y_x= L_2^\perp \cap Y_x=:H_2$. We claim that this is not possible. Suppose for a contradiction that $L_1$ and $L_2$ span a singular plane $\pi$. Then $L_2$ is collinear to the maximal singular subspace $\<L_1,H_1\>$ of $\theta^{x}_1$, and hence for any $X$-plane $\pi_1$ through $L_1$ in $\theta^{x}_1$, $L_2$ is not collinear to $\pi_1$. As such, $\pi$ and $\pi_1$ are contained in some $\zeta \in \Xi \cup \Theta$. However, if $\zeta \in \Xi$ this contradicts $r=1$ and if $\zeta \in \Theta$ this contradicts Corollary~\ref{noXplane}, showing the claim. 

We deduce from Lemma~\ref{Hperp} that the respective vertices $V^{x}_1$ and $V^{x}_2$ of $\theta^{x}_1$ and $\theta^{x}_2$ do not coincide. Since the points of $Z$ in $\theta^{x}_i$, $i=1,2$, are precisely those of $V^{x}_i$ and $R^{x}_i$, we have that $V^{x}_1=R^x_2 \cap Y_x$ and $V^{x}_2 = R^x_1 \cap Y_x$.  In particular, $v'=r'-1$ and $V^{x}_1$ and $V^{x}_2$ are disjoint subspaces spanning  $Y_x$. 

We conclude that the members of $\Theta$ through $x$ have pairwise disjoint vertices, which are contained in $Z \cap Y_x$.   Since the latter  only contains two $v'$-spaces, there are precisely  two members of $\Theta$ through $x$. 
\end{proof}

\par\bigskip
\begin{defi} \label{not} \em For each $X$-point $x$ on a $1$-line, we henceforth denote the unique two members of $\Theta$ through ${x}$ by $\theta^{x}_1$ and $\theta^{x}_2$, their respective vertices by $V^x_1$ and $V^x_2$ and the $r'$-spaces in $Z$ which are complementary to $V^x_i$ inside $Y(\theta^x_i)$, $i=1,2$, by $R^x_1$ and $R^x_2$, respectively.
\end{defi}

\begin{cor}\label{1':symp}
Suppose $x\in X$ is on a $1$-line. Let $L_i$  be an $X$-line through $x$ in $\theta^x_i$, $i=1,2$. Then $L_1$ and $L_2$ are non-collinear and  $[L_1,L_2]\in\Xi$, and $v=2v'-1=2r'-3$.
\end{cor}
\begin{proof}
By Lemma~\ref{1':2SS}, the respective vertices $V^x_1$ and $V^x_2$ of $\theta^x_1$ and $\theta^x_2$ do not coincide. So, according to Lemma~\ref{Hperp}, the lines $L_1$ and $L_2$ are not collinear and $[L_1,L_2] \in \Xi$. Secondly, we have $v=v'+r'-2$ and $v'=r'-1$.
\end{proof}

\begin{lemma}\label{1':xlinesx}
Suppose $x\in X$ is on a $1$-line. Then all $X$-lines through ${x}$ belong to $\theta^{x}_1 \cup \theta^{x}_2$.
\end{lemma}
\begin{proof}
Suppose for a contradiction that $K$ is an $X$-line through ${x}$ not contained in $\theta^{x}_1 \cup \theta^{x}_2$. Then $K$ is a $0$-line, for otherwise an element of $\Theta$ through $K$ would coincide with one of $\theta^{x}_1,\theta^{x}_2$. So, by Lemma~\ref{lineSS}$(i)$, $K$ is collinear to $Y_x$. Take any $X$-line $L$ through ${x}$ in $\theta^{x}_1$. Then $L$ and $K$ are not collinear, because $K^\perp \cap \theta^x_1$ is the maximal singular subspace $\<x,Y_x\>\not\supseteq L$. As such, $[L,K]$ belongs to $\Xi$  and has vertex $Y_L$. But then $v=r'+v'-1=2v'$, whereas we deduced before that $v=2v'-1$ (cf.\ Corollary~\ref{1':symp}). This contradiction shows the lemma.
\end{proof}

\textbf{Proof of Proposition~\ref{1':xlines}.}
By Proposition~\ref{1':nonex}, there is at least one $1$-line $L$. Let $x$ be any point on $L$. Firstly, consider any other $X$-line $M$ through $x$. Then $M$ is contained in $\theta^x_1$ or $\theta^x_2$ by Lemma~\ref{1':xlinesx} and, in there, it is clear that $Y_M$ is a hyperplane of $Y_x$. It follows from Lemma~\ref{lineSS} that $M$ is a $1$-line indeed. Since $x$ was just any point on a $1$-line, we obtain by connectivity (via $X$-lines) that all $X$-lines are indeed $1$-lines. \hfill $\square$

\subsection{The structure of $Y$ and its connections to $\rho(X)$}\label{1':1line}

Knowing that all $X$-lines are $1$-lines, we can start a structure analysis. We begin by examining the structure of $Y$.

\subsubsection{Properties of $Y$ in terms of $\{Y_x \mid x \in X\}$ and $\{V \mid V \text{ vertex of } \theta \in \Theta\}$}

Recall the notation introduced in Definition~\ref{not}. We show that $R^x_1 \cup R^x_2$ does not depend on $x$. 

\begin{lemma}\label{1':YZ}
For any  $x\in X$, we have $Z=R^{x}_1 \cup R^{x}_2$ and  $\dim(Y)=2r'+1$. Renumbering if necessary,  $R^x_1=R^{x'}_1$ and $R^x_2=R^{x'}_2$ for each $x'\in X\setminus\{x\}$. \end{lemma}
\begin{proof}
Take a point $z\in Z\setminus Y_x$. Then $[x,z] \in \Theta$, so, by Lemma~\ref{1':2SS}, $[x,z] \in \{\theta^{x}_1, \theta^{x}_2\}$, and hence $z \in R^{x}_1 \cup R^{x}_2$. So $Z=R^x_1 \cup R^x_2$ indeed. We claim that $R^x_1 \cap R^x_2=\emptyset$. Firstly, $R^x_1 \cap R^x_2 \subseteq Y_x$ (for otherwise $\theta^1_x=\theta^2_x$) and secondly, Lemma~\ref{1':2SS} says that  $R^x_1 \cap Y_x=V^x_2$, $R^x_2 \cap Y_x=V^x_1$  and $V^x_1 \cap V^x_2 = \emptyset$. This shows the claim, and $Y=\<Z\>$ then implies that $\dim(Y)=2r'+1$.

A union of two $r'$-spaces only contains two $r'$-spaces, so for each pair of points $x,x'\in X$ we have $\{R_1^x,R_2^x\} = \{R^{x'}_1,R^{x'}_2\}$. So far, the numbering was arbitrary but of course this gives a canonical numbering: for each $x'\in X\setminus\{x\}$, we choose it such that $R_1^x = R^{x'}_1$ and $R_2^x=R^{x'}_2$. 
\end{proof}

The previous lemma allows us to put $R_i:=R^x_i$, $i=1,2$, for any $x\in X$. This hence divides the set $\Theta$ in two: for $i=1,2$, we define $\Theta_i$ as the set $\{\theta \in \Theta \mid R_i \subseteq Z(\theta)\}$. It also divides the set of $X$-lines inside each member of $\Xi$ in two natural reguli, as we show in Lemma~\ref{1':sympx}$(i)$.

\begin{cor}\label{V}
For each $x\in X$, we have $V^x_1=R_2 \cap Y_x$ and $V^x_2= R_1 \cap Y_x$.
\end{cor}
\begin{proof}
This follows immediately from Lemmas~\ref{1':YZ} and~\ref{1':2SS}.
\end{proof}

\begin{lemma}\label{1':sympx}
Let $\xi \in \Xi$  be arbitrary and denote its vertex by $T$. Then, for each $x\in X(\xi)$:
\begin{compactenum}[$(i)$]
\item $\theta^i_x$ and $\xi$ share a generator $G_i=\<T,L_i\>$, for some $X$-line $L_i$ of $\xi$, $i=1,2$, and $T=\<L_1^\perp \cap V^x_2, L_2^\perp \cap V^x_1\>$;
\end{compactenum}
moreover, for each $x'\in X(\xi)$,  putting $(x_1,x_2):=(x',x)$ if $x' \in G_1$;   $(x_1,x_2):=(x,x')$ if $x' \in G_2$; and $(x_1,x_2):=(x'^\perp \cap L_1, x'^\perp \cap L_2)$ if $x'\notin G_1 \cup G_2$, we have 
\begin{compactenum}[$(ii)$]
\item $(V^{x'}_1,V^{x'}_2)=(V_1^{x_2},V_2^{x_1})$ and hence $Y_{x'} = \<R_1 \cap x_1^\perp, R_2 \cap x_2^\perp\>$.
\end{compactenum}
\end{lemma}
\begin{proof}
$(i)$ Let $L_1$ and $L_2$ be any two non-collinear $X$-lines through $x$ inside $\xi$. Since all $X$-lines through $x$ are contained in $\theta_1^x\cup\theta_2^x$ by Lemma~\ref{1':xlinesx}, we have (renumbering if necessary) that $L_i \subseteq \theta^x_i$, $i=1,2$. Furthermore we know that $T$ is determined as $Y_{L_1} \cap Y_{L_2}$, which belongs to $Y_x=\<V^x_1,V^x_2\>$. Firstly, this implies that $G_i=\<T,L_i\> = \xi \cap \theta^x_i$ for $i=1,2$; secondly, it implies that $T=\<L_1^\perp \cap V^x_2, L_2^\perp \cap V^x_1\>$ (noting that $L_i \perp V^x_i$, $i=1,2$). 

$(ii)$ Recall from Corollary~\ref{V} that $V_1^{x'}= R_2 \cap Y_{x'}$ and $V_2^{x'} =  R_1 \cap Y_{x'}$ for each $x'\in X$. Firstly, take $x' \in G_1$.  Then $V_1^{x'}=V_1^{x}=V_1^{x_2}$ since  $x,x' \in \theta^x_1$ and $x=x_2$;  $V_2^{x'}=V_2^{x_1}$ is trivial since $x_1=x'$. Likewise, the statement is true if $x' \in G_2$, so suppose $x' \notin G_1 \cup G_2$. Note that this implies that $x_i=x'^\perp \cap L_i$ is indeed a unique point of $L_i$ distinct from $x$ for $i=1,2$, as one can see inside $X(\xi)$; in particular, $x' \notin \theta^x_1 \cup \theta^x_2$.
As such, the line $x_1x'$ is not contained in $\theta^x_1=\theta^{x_1}_1$, and by Lemma~\ref{1':xlinesx} this means that $x_1x' \subseteq \theta_2^{x_1}$, implying that $V_2^{x'}=V_2^{x_1}$. Likewise, we obtain $V_1^{x'}=V_1^{x_2}$.
By Lemma~\ref{1':2SS},  $Y_{x'} = \<V_1^{x'},V_2^{x'}\>$ and by the above this equals $\<V_1^{x_2},V_2^{x_1}\>$, which at its turn coincides with $\<x_2^\perp \cap R_2,x_1^\perp \cap R_2\>$ by Corollary~\ref{V}.
\end{proof}

This allows us to provide a counterpart for Corollary~\ref{V}.

\begin{lemma}\label{1':V1V2}
Let $V_i$ by any hyperplane of $R_i$, $i=1,2$. Then there is a point $x\in X$ such that $Y_x=\<V_1,V_2\>$ and all points of $X$ collinear to $\<V_1,V_2\>$ are precisely those of $\<x,Y_x\> \cap X$.
\end{lemma}
\begin{proof}
Let $x\in X$ be arbitrary. Suppose first that  $x$ is collinear to $\<V_1,V_2\>$. We claim that $\<x,V_1,V_2\> \cap X$ is precisely the set of $X$-points collinear to $\<V_1,V_2\>$. Clearly, all points of $\<x,V_1,V_2\> \cap X$ are collinear to $\<V_1,V_2\>$. Suppose for a contradiction that $x' \notin \<x,V_1,V_2\>$ is an $X$-point collinear to $\<V_1,V_2\>$. If $x$ and $x'$ determine a unique member of $\Xi$, then $\<V_1,V_2\>$ would be contained in its vertex, contradicting the fact that $\dim(\<V_1,V_2\>)=2v'+1>v$ by Corollary~\ref{1':symp}. In all other cases,  Lemma~\ref{1':xlinesx} or (S1) implies that $x$ and $x'$ belong to some  $\theta\in \Theta$. By Lemma~\ref{1':xperp}, $\theta$ contains $Y_x= \<V_1,V_2\>$, and $\<x,V_1,V_2\>$ is the unique maximal singular subspace in $\theta$ not contained in $Y$ through $\<V_1,V_2\>$, which means that $\<x',V_1,V_2\>=\<x,V_1,V_2\>$ after all.  The claim follows. 

Next, suppose that $x$ is not collinear to $\<V_1,V_2\>$. Hence $V_2^x \neq V_1$ or $V^x_1 \neq V_2$. Without loss of generality, the first option holds. Then there is an $X$-point $x_1$ in $\theta_1^x$ on an $X$-line with $x$ that is  collinear to $V_1\subseteq R_1$.  In case $V_1^x = V_2$, then $x_1 \perp V_2$ since $x,x_1 \in \theta^x_1$, and hence  $x_1$ is collinear to $\<V_1,V_2\>$ and we are done. In case $V_1^x \neq V_2$,  we can likewise find a point $x_2 \in X(\theta_2^x)$ on an $X$-line with $x$ that is collinear to $V_2$.  Putting $L_i:=xx_i$ for $i=1,2$, we obtain from Corollary~\ref{1':symp} that $L_1$ and $L_2$ are contained in a member $\xi$ of $\Xi$. Let $x'$ be a point on $X(\xi)$ collinear to $x_1$ and $x_2$, but not equal or collinear to $x$. Then Lemma~\ref{1':sympx} gives $Y_{x'}=\<x_1^\perp \cap R_1, x_2^\perp \cap R_2\>=\<V_1,V_2\>$. \end{proof}
\par\bigskip
There are some interesting consequences.

\begin{lemma}\label{1':int12}
Let $\theta$ and $\theta'$ be two members of $\Theta$, with respective vertices $V$ and $V'$. Then $\theta \cap \theta'$ contains  an $X$-point $x$ if and only if $\theta$ and $\theta'$ belong to the different classes $\Theta_1$ and $\Theta_2$ of $\Theta$. If this is the case, then $\theta \cap \theta'=\<x,V,V'\>$; if not, say if $\theta,\theta'\in\Theta_i$ for some $i\in\{1,2\}$, then  $\theta \cap \theta'$ is $\<R_i,V \cap V'\>$ and $V \cap V'$ is a hyperplane of~$V$ and~$V'$. 
\end{lemma}
\begin{proof}
If $\theta \cap \theta'$ contains an $X$-point $x$, then without loss, $\theta = \theta^x_1 \in \Theta_1$ and $\theta'=\theta^x_2 \in \Theta_2$, so they belong to different classes indeed and $\theta_1 \cap \theta_2 = \<x,V,V'\>=\<x,Y_x\>$ since $Y_x \subseteq \theta, \theta'$ by Lemma~\ref{1':xperp}. Let $\theta_1 \in \Theta_1$ and $\theta_2 \in \Theta_2$  be arbitrary and denote their respective vertices by $V_1$ and $V_2$.
By definition,  $\theta_i$ contains $R_i$ for $i=1,2$, and therefore $V_1 \subseteq R_2 \subseteq \theta_2$ and $V_2 \subseteq  R_1 \subseteq \theta_1$. So $\<V_1,V_2\> \subseteq \theta_1 \cap \theta_2$. In $\theta_1$ and $\theta_2$, there are (unique) maximal singular subspace through $\<V_1,V_2\>$ containing a point of $X$. By Lemma~\ref{1':V1V2}, these two subspaces coincide, implying that $\theta_1 \cap \theta_2$ contains an $X$-point.

Secondly, take two arbitrary members $\theta, \theta' \in \Theta_1$ ($\Theta_2$ plays the same role). Again, $R_1$ is contained in $\theta \cap \theta'$ by definition. The vertices $V$ and $V'$ are hyperplanes of $R_2$. Let $x \in X(\theta)$ and $x'\in X(\theta')$ be points with $x^\perp \cap R_1 = x'^\perp \cap R_1$. Then, if $V=V'$, also $Y_x=Y_{x'}$, which by Lemma~\ref{1':V1V2} implies that $x' \in \<x,Y_x\> \subseteq \theta$. However, as $\theta$ and $\theta'$ belong to the same class, they cannot share a point of $X$. We conclude that $V\neq V'$ and hence (looking inside $R_2$) we see that $V \cap V'$ is a hyperplane in both $V$ and $V'$.
\end{proof}

\begin{cor}\label{1':SSV}
For each hyperplane $V$ of $R_i$, there is a unique member of $\Theta_j$ having $V$ as its vertex, $\{i,j\}=\{1,2\}$.
\end{cor}
\begin{proof}
Without loss, $V \subseteq R_1$. Let $V'$ by any hyperplane of $R_2$. Then by Lemma~\ref{1':V1V2}, there is a point $x\in X$ with  $Y_x=\<V,V'\>$. Let $z\in R_2\setminus V'$ be arbitrary. Then $[x,z]$ is a member of $\Theta$ containing $R_2$, i.e., $[x,z]\in \Theta_2$. Moreover, $[x,z]$ contains $Y_x$ by Lemma~\ref{1':xperp} and therefore, $[x,z]$ has $V$ as its vertex. By Lemma~\ref{1':int12}, there is no other member of $\Theta_2$ having $V$ as vertex. 
\end{proof}

\par\bigskip

The relation between two $X$-points can be expressed in terms of the subspaces of $Y$ they are collinear to.
\begin{lemma}\label{1':rel}
Take two distinct points $x_1,x_2 \in X$. Then
\begin{compactenum}[$(i)$]
\item $x_1x_2$ is a singular line with a unique point in $Y$ $\Leftrightarrow V_i^{x_1} = V_i^{x_2}$ for all $i\in\{1,2\}$;
\item $x_1$ and $x_2$ are either on an $X$-line or non-collinear points of a member of $\Theta$ $\Leftrightarrow V_i^{x_1} = V_i^{x_2}$ for precisely one $i\in\{1,2\}$;
\item $x_1$ and $x_2$ are non-collinear points of a member of $\Xi$  $\Leftrightarrow V_i^{x_1} \neq V_i^{x_2}$ for all $i\in\{1,2\}$.
\end{compactenum}
\end{lemma}
\begin{proof}
Since the possibilities for $x_1,x_2$ described in $(i)$, $(ii)$ and $(iii)$ exhaust the mutual positions between $x_1$ and $x_2$, it suffices to verify the ``$\Rightarrow$''s. 

$(i), \Rightarrow$: This is clear.

$(ii), \Rightarrow$: Observe that $x_1,x_2$ are contained in some $\theta\in\Theta$ in both cases (since each $X$-line is a $1$-line by Proposition~\ref{1':xlines}). Recall that $Y_{x_1} \cup Y_{x_2} \subseteq Y(\theta)$ (cf.\ Lemma~\ref{1':xperp}) and that $\theta$ has $R_1$ or $R_2$ as its vertex, say $R_1$. As such, $V_2^{x_1}=V_2^{x_2}$. If also $V_1^{x_1} = V_1^{x_2}$, i.e., $x_1^\perp \cap R_2 = x_2^\perp \cap R_2$, then this would either yield a singular $(r'+1)$-space in $\theta$ (if $x_1 \perp x_2$) or yield three singular $r'$-spaces through a singular $(r'-1)$-space (if $x_1$ and $x_2$ are non-collinear), a contradiction. So

$(iii), \Rightarrow$: Put $\xi=[x_1,x_2]$. Then $Y(\xi)=Y_{x_1} \cap Y_{x_2}$. Since $v=2r'-3$, the assertion follows.
%Assertion $(i)$ follows  from Lemma~\ref{1':V1V2}. The  ``$\Rightarrow$''-implications of $(ii)$ and $(iii)$ are also easily verified. So, next, suppose  that $V_i^{x_1} = V_i^{x_2}$ for precisely one $i\in\{1,2\}$. Then $x_1$ and $x_2$ cannot be contained in a member of $\Xi$ for this would violate our deduced value of $v$; nor can they be on a singular line with a unique point of $Y$ by assertion $(i)$, and hence ``$\Leftarrow$'' of assertion $(ii)$ follows. By lack of other options, also ``$\Leftarrow$'' of assertion $(iii)$ now follows. 
\end{proof}
\par\bigskip

We again consider the maps $\rho$ and $\chi$ (cf.\ Definitions~\ref{RHO} and~\ref{CHI}). 

\subsubsection{The projection $\rho(X)$ and its connection to $Y$}

\textbf{Notation.} Denote by $\mathsf{L}$ the set of $X$-lines.

The next lemma shows in particular that, for each $\rho(X)$-line $L'$, there is an $X$-line $L$ with $\rho(L)=L'$. 

\begin{lemma}~\label{1':rhoXlines}
Suppose $\rho(x_1)$ and $\rho(x_2)$ determine a singular line of $\rho(X)$, for $x_1,x_2 \in X$. Let $x'_i  \in \rho^{-1}(\rho(x_i))$ be arbitrary, for $i=1,2$. Then:
\begin{compactenum}[$(i)$]
\item There is a unique member  $\theta\in\Theta$ containing $x'_1 \cup x'_2$, and  $\rho^{-1}(\rho(x_1)) \cup  \rho^{-1}(\rho(x_2))\subseteq \theta$;
\item there is an $x''_2\in\rho^{-1}(\rho(x_2))$ such that $\<x'_1,x''_2\>$ is an $X$-line. 
\item $\{Y_x \mid \rho(x) \in \<\rho(x_1),\rho(x_2)\>\}$ is the set of all $(2r'-1)$-spaces through the $(2r'-2)$-space $Y_{x_1}\cap Y_{x_2}$ inside the $2r'$-space $Y(\theta)$. 
\end{compactenum}
\end{lemma}
\begin{proof}
$(i)$ Also here, for $i=1,2$, $x'_i\in\rho^{-1}(\rho(x_i)) = \<x_i,Y_{x_i}\> \cap X$  by Lemma~\ref{proprhochi}; in particular $Y_{x_i}=Y_{x'_i}$. By Lemma~\ref{1':rel}, it suffices to show that $\dim(Y_{x_1} \cap Y_{x_2})=2r'-2$, as this implies that there is a member $\theta \in \Theta$ containing $x'_1$ and $x'_2$ and hence also $\<x_1,Y_{x_1}\>$ and $\<x_2,Y_{x_2}\>$ (cf.\ Lemma~\ref{1':xperp}). So suppose for a contradiction that $\dim(Y_{x_1} \cap Y_{x_2})\neq2r'-2$. By Lemma~\ref{1':V1V2}, $Y_{x_1} \neq Y_{x_2}$ because $\rho(x_1)\neq \rho(x_2)$ by assumption, and hence,  by Lemma~\ref{1':rel}, the unique other option is $\dim(Y_{x_1} \cap Y_{x_2}) = 2r'-3$, in particular $\<Y_{x_1},Y_{x_2}\>=Y$.

Consider the $(2r'+3)$-space $\<x_1,x_2,Y\>$, which, as noted above, equals $\<x_1,x_2,Y_{x_1},Y_{x_2}\>$. Recall that $\dim(\<x_i,Y_{x_i}\>)=2r'$. We claim that there is an $X$-line $L$ with $\rho(L)=\<\rho(x_1),\rho(x_2)\>$. Indeed, the fact that $\rho(x_1)$ and $\rho(x_2)$ are on  a $\rho(X)$-line, implies that there is a point $x_3 \in X$ with $\rho(x_3)$ on $\rho(x_1)\rho(x_2)\setminus\{\rho(x_1),\rho(x_2)\}$, and hence $x_3$ is a point of $\<x_1,x_2,Y\>$. In the latter subspace, we see that $\<x_3,x_1,Y_{x_1}\>$ (dimension $2r'+1$) intersects $\<x_2,Y_{x_2}\>$ (dimension $2r'$) in a subspace of dimension $(2r'-2)$, which is therefore generated by $Y_{x_1}\cap Y_{x_2}$ and some $X$-point, say $x''_2$. The line $\<x''_2,x_3\>$ then intersects $\<x_1, Y_{x_1}\>$ in a point $x''_1$. Since $x_3$ does not belong to  $Y$, neither does $\<x''_1,x''_2\>$.  As $\<x''_1,x''_2\>$ contains three points of $X\cup Y$, it is singular (cf.\ Lemma~\ref{XYpointsonline}), implying that it is an $X$-line (since $x_3 \notin \<x_1,Y_{x_1}\>$). This shows the claim. The $X$-line $L$ is a $1$-line by Proposition~\ref{1':xlines}, and as such it belongs to a member of $\Theta$ that also contains $x_1$ and $x_2$. By Lemma~\ref{1':rel}, this contradicts the assumption on $Y_{x_1} \cap Y_{x_2}$. Assertion $(i)$ follows.

$(ii)$  Let $\theta$ be  the unique member  containing $x'_1$ and $x'_2$. In $\theta$, we see that $x'_1$ is collinear to a hyperplane of $\<x_2,Y_{x_2}\>$, which does not coincide with $Y_{x_2}$ (since $Y_{x_1} \neq Y_{x_2}$) and hence contains an $X$-point $x''_2$. 

$(iii)$ Let $L$ be an $X$-line in $\theta$ with $\rho(L)=\<\rho(x_1),\rho(x_2)\>$ (possible by $(ii)$). Clearly, $\rho$ gives a bijective correspondence between the points of $L$ and the points of  $\<\rho(x_1),\rho(x_2)\>$. Furthermore, looking in $\theta$, it is also clear that the collinearity relation $x \mapsto Y_x$ is a bijection between the points on $L$ and the $(2r'-1)$-spaces of $Y(\theta)$ containing $Y_{x_1} \cap Y_{x_2}$. Composing these two bijections, the assertion follows. 
\end{proof}

Let $\xi \in \Xi$ be arbitrary. We already noted in Lemma~\ref{proprhochi} that $\rho(X(\xi))$ is a hyperbolic quadric $Q$ in $\rho(X)$ of rank $r=1$. Below, we basically show that $Q$ corresponds in a ``nice'' way to the members of $\Xi$ determined by pairs of non-collinear $X$-points in $\rho^{-1}(Q)$: they all have the same vertex and the same image under $\rho$. We introduce some notation first.

\textbf{Notation.} Let $V \subseteq Y$ be the vertex of any $\xi \in \Xi$. Then we denote by  $\Xi_V$ the subset of $\Xi$ whose members have vertex $V$ and we denote by $X_V$ the $X$-points collinear to $V$. Clearly, $X(\xi) \subseteq X_V$ for each $\xi\in \Xi_V$. 

\begin{lemma}~\label{1':welldef}
Suppose $\xi=[x_1,x_2]\in \Xi$ for points $x_1,x_2\in X$ and put $T=Y(\xi)$ and $T_i:=R_i\cap T$. Then:
\begin{compactenum}[$(i)$]
\item $\sigma_\xi: \rho(X(\xi)) \rightarrow \{\<H_1,H_2\> \mid T_i \subsetneq H_i \subsetneq R_i, i=1,2\}: \rho(x) \mapsto Y_x$ is a isomorphism;
\item If $x'_i \in \rho^{-1}(\rho(x_i))$ for $i=1,2$, then $\xi':=[x'_1,x'_2]$ belongs to $\Xi_T$;
\item for each $\xi' \in \Xi_T$, $\rho(X(\xi'))=\rho(X(\xi))$.
\end{compactenum}
\end{lemma}

\begin{proof} 
$(i)$ First of all, note that $\sigma_\xi$ is well-defined by Lemma~\ref{proprhoci} and the fact that $Y_x$ contains $T$ and shares a hyperplane with each of $R_1$, $R_2$. Since $\dim(T)=2r'-3$ and $\dim(Y)=2r'+1$,  the residue $\Res_Y(T)$ is isomorphic to  a projective $3$-space over $\K$, say $\Pi_T(\K)$, in which $R_i$ corresponds to a line $L_i$ and $Y_x$ to a line $L(x)$ meeting both $L_1$ and $L_2$ in a point. Let $x,x'$ be two points of $X(\xi)$. By Lemma~\ref{1':rel}, $L(x)=L(x')$ if and only if $x$ and $x'$ belong to the same generator of $X(\xi)$, i.e., if and only if $\rho(x)=\rho(x')$; $L(x)$ and $L(x')$ intersect in precisely a point (which belongs to $L_1$ or $L_2$) if and only if $xx'$ is an $X$-line in $X(\xi)$ and $L(x)$ and $L(x')$ are disjoint if and only if $x$ and $x'$ are non-collinear. Moreover, Lemma~\ref{1':rhoXlines} implies that each $X$-line of $X(\xi)$ corresponds to a full planar point pencil in $\Pi_T$. 

On the other hand, the point-line geometry whose point set is the set of lines of  $\Pi_T(\K)$  meeting both $L_1$ and $L_2$ non-trivially and whose lines are the full planar line pencils contained in it, is (as can be seen by dualising) isomorphic to the point-line geometry associated to a hyperbolic quadric $Q$ in $\mathbb{P}^{3}(\K)$. Moreover,  $\rho(X(\xi))$ is a hyperbolic quadric in $\mathbb{P}^{5}(\K)$ as $r=1$.  By the previous paragraph, $\sigma_\xi(\rho(X(\xi)))$  is embedded isometrically into $Q$. Since the fields of definition are the same, $\sigma_\xi(\rho(X(\xi)))=Q$, i.e., $\sigma_\xi$ is an isomorphism. Assertion $(i)$ follows.

$(ii)$ By Lemma~\ref{proprhochi}, $x'_i \in \<x_i, Y_{x_i}\> \cap X$, for $i=1,2$. Since $\xi=[x_1,x_2]$ and $Y_{x_i}=Y_{x'_i}$ for $i\in\{1,2\}$, Lemma~\ref{1':rel} implies that also $x'_1$ and $x'_2$ are non-collinear points of some $\xi' \in \Xi$. Moreover, $Y(\xi')=Y_{x'_1} \cap Y_{x'_2}=Y_{x_1}\cap Y_{x_2}=T$. 

$(iii)$ Let $\xi'$ be any member of $\Xi$ with vertex~$T$. By $(i)$, $\rho(X(\xi'))$ is isomorphic to $\rho(X(\xi))$ via $\sigma^{-1}_{\xi'} \circ \sigma_\xi$. This means that, for each point $\rho(x')\in \rho(X(\xi'))$, there is a unique point $\rho(x)\in \rho(X(\xi))$ with $Y_x=Y_{x'}$, so by Lemma~\ref{1':V1V2}, $\rho(x)=\rho(x')$. 
\end{proof}

We use $Y$ to define the following point-line geometry $(\mathsf{P},\mathsf{B})_Y$.

\begin{defi}\label{1':PB}
Let  $\mathsf{P}$ denote the set $\{\<H_1,H_2\> \mid H_i \subseteq R_i, \dim H_i=r'-1\}$. For subspaces $T_1$ and $T_2$ of $R_1$ and $r_2$, respectively, with $\dim T_i=r'-2$, and a subspace $H_j$ with $T_j \subsetneq H_j \subsetneq T_j$ for $j\in \{1,2\}$,  we define the pencil $P_j(T_1,T_2)$ as the set $\{P \in \mathsf{P} \mid \<H_j,T_1,T_2\> \subseteq P\}$. Then we denote by $\mathsf{B}$ the set $\bigcup_{j\in\{1,2\}} \{P_j(T_1,T_2) \mid T_i \subseteq R_i, \dim T_i = r'-2\}$. 
\end{defi}

\begin{lemma}\label{1':Ystructuur}
The point-line geometry $(\mathsf{P},\mathsf{B})_Y$ is isomorphic to an injective projection of the Segre geometry $\mathcal{S}_{r',r'}(\K)$.
\end{lemma}

\begin{proof}
By dualising, we see that the point-line geometry $(\mathsf{P},\mathsf{B})_Y$ (with natural incidence relation) is isomorphic to the direct product of two projective $r'$-spaces (namely, $R_1$ and $R_2$), and hence by Fact~\ref{factuniemb},   $(\mathsf{P},\mathsf{B})_Y$ is an injective projection of $\mathcal{S}_{r',r'}(\K)$. 
\end{proof}

\begin{prop}\label{1':projsegre}
The point-line geometry $\mathcal{S}:=(\rho(X),\rho(\mathsf{L}))$ is isomorphic to an injective projection of the Segre geometry $\mathcal{S}_{r',r'}(\K)$. Moreover, we have
\begin{compactenum}
\item[$(i)$] for each singular $r'$-space $S$ in $\mathcal{S}$, there is a unique $\theta_S \in \Theta$ with
 $\rho^{-1}(S)=X(\theta_S)$.
 \item[$(ii)$] the sets $\mathcal{S}i:=\{\rho(X(\theta)) \mid \theta \in \Theta_i\}$, for $i=1,2$, are the two natural families of singular $r'$-spaces of $\mathcal{S}$.
\item[$(iii)$] for each grid $G$ of $\mathcal{S}$, there is a unique  $v$-space $V$ in $Y$ with $\rho^{-1}(G)=X_V$ and vice versa.
\end{compactenum}
\end{prop}
\begin{proof} We  claim that $\chi$ induces an isomorphism between the abstract point-line geometries $(\rho(X),\rho(\mathsf{L}))$ and $(\mathsf{P},\mathsf{B})_Y$. Indeed, the fact that $\chi: \rho(X) \rightarrow \mathsf{P}: x \mapsto \chi(x)=Y_x$ is a bijection between $\rho(X)$ and $\mathsf{P}$ follows immediately from Lemma~\ref{1':V1V2} and the fact that $\rho^{-1}(x)=\<x,Y_x\>\cap X$ for each $x\in X$. The fact that a member $\rho(\mathsf{L})$ is mapped by $\chi$ to a member of $\mathsf{B}$ follows from Lemma~\ref{1:rhoXlines}$(iii)$. This shows the claim. By Fact~\ref{factuniemb} and Lemma~\ref{1':Ystructuur}, $(\rho(X),\rho(\mathsf{L}))\subseteq F$ arises as an injective projection of the Segre geometry $\mathcal{S}_{r',r'}(\K)$.

$(i), (ii)$ Let $S$ be a maximal singular subspace of $\rho(X)$ of dimension $r'$ and take a line $L$ in $S$. By Lemma~\ref{1:rhoXlines}$(i)$, there is a unique $\theta\in \Theta$ containing $L$. The properties of the Segre variety $\mathcal{S}_{r',r'}(\K)$ imply that there is a unique $r'$-space of $\rho(X)$ through $L$; as such, the $r'$-space $\rho(X(\theta))$ coincides with $S$.

$(iii)$ Lastly, let $G$ be any grid of $\mathcal{G}$.  By Lemma~\ref{1:welldef}$(iii)$, it suffices to show that $G$  coincides with $\rho(X(\xi))$ for some $\xi \in \Xi$.  Let $p_1$ and $p_2$ be non-collinear points of $G$ and take points $x_1,x_2 \in X$ with $\rho(x_i)=p_i$. Then, since $p_1$ and $p_2$ are distinct and non-collinear,  $\xi:=[x_1,x_2]\in \Xi$. Now, $\rho(X(\xi))$ is a grid in $\rho(X)$ containing the points $p_1$ and $p_2$, and since two non-collinear points determine a unique grid in $\rho(X)$ (by taking the convex closure), we obtain $\rho(X(\xi))=G$. 
\end{proof}

\begin{cor}\label{NDS}
We have $N\leq r'^2+4r'+2$. 
\end{cor}
\begin{proof}
By Lemma~\ref{1':YZ}, we know $\dim(Y)=2r'+1$ and by Proposition~\ref{1':projsegre}, $\dim(F)\leq (r'+1)^2-1$. Since $F$ and $Y$ generate $\mathbb{P}^N(\K)$, we obtain $N\leq r'^2+4r'+2$. 
\end{proof}

Recall that $(\rho(X),\rho(\mathsf{L}))$ is a \emph{legal} projection of $\mathcal{S}_{r',r'}(\K)$ if (S2) also holds here (cf.\ Definition~\ref{legal}).

\begin{prop}\label{1':FX}
The set $X$ contains a legal projection $\Omega$ of $\mathcal{S}_{r',r'}(\K)$ which is such that: 
\begin{compactenum}[$(i)$]
\item $\bigsqcup_{x\in \Omega} \<{x},Y_{{x}}\>\setminus Y_x=X$ and hence, putting $F^*=\<\Omega\>$, $\<F^*,Y\>=\mathbb{P}^N(\K)$;
\item Re-choosing the subspace $F$ so that it is inside $F^*$, the projection $\rho^*$ of $F^* \cap X$ from $F^*\cap Y$ onto $F$ is the restriction of $\rho$ to $F^* \cap X$;
\item  Containment gives a bijection between the two natural families of $r'$-spaces of $\Omega$ and the sets $\Theta_1$ and $\Theta_2$;
\item If $r' =2$, then $F^* \cap Y = \emptyset$.
\end{compactenum} \end{prop}

\begin{proof}
By Proposition~\ref{1':projsegre}, $\rho(X)$ is the point set of an injective projection of  a Segre geometry $\mathcal{S}_{r',r'}(\K)$, and the elements $\theta \in \Theta$ are in $1-1$-correspondence to the set of $r'$-spaces $S_\theta$. We  construct a legal projection of $\mathcal{S}_{r',r'}(\K)$ \emph{inside} $X$ (where we know that (S2) holds), by choosing a set of points $X'\subseteq X$ such that $\{Y_x \mid x \in X'\}$ is a basis of $(\mathsf{P},\mathsf{B})_Y$ (cf.\ Definition~\ref{1':PB}). 

To that end, take a basis of hyperplanes $V_{1,0},...,V_{1,r'}$ of $R_2$ and a basis of hyperplanes $V_{2,0},...,V_{2,r'}$ of $R_1$. By Corollary~\ref{1':SSV}, there is, for each $0 \leq t \leq r'$, $i=1,2$, a unique member  $\theta_{i,t} \in \Theta_i$ having $V_{i,t}$ as its vertex. Put  $\Pi_{t,u}:=\theta_{1,t} \cap \theta_{2,u}$ for each pair $t,u$. By Lemma~\ref{1':int12}, $\Pi_{t,u}$ coincides with  $\<x'_{t,u},V_{1,t},V_{2,u}\>$, where $x'_{t,u}$ is any $X$-point in $\theta_{1,t} \cap \theta_{2,u}$, or, equivalently (cf.\ Lemma~\ref{1':V1V2}),  $x'_{t,u}$ is  any $X$-point collinear to $\<V_{1,t}, V_{2,u}\>$. In particular, $\dim\Pi_{t,u}=2r'$.

\textit{Claim 1: we can consecutively choose points $X$-points $x_{t,u} \in \Pi_{t,u}$, using the lexicographic order on the pairs $\{(t,u)\mid 0\leq t,u \leq r'\}$, in such a way that $x_{t,u} \perp x_{t,u'}$ for all $0 \leq u' < u$ and $x_{t,u} \perp x_{t',u}$ for all $0\leq t' < t$.}\\The first point $x_{0,0}$ can be chosen as any $X$-point in $\Pi_{0,0}$.  We proceed inductively.
Assume that we have to choose the point $x_{t,u}$ and that all preceding points (i.e., $x_{t',u'}$ with either $t' < t$ or $t'=t$, $u' < u$) are fine. Our requirements imply that $x_{t,u}$ is an $X$-point in the subspace 
\[\Pi'_{t,u}:= \Pi_{t,u} \cap \bigcap_{0\leq t' < t} x_{t',u}^\perp \cap \bigcap_{0 \leq u' <u} x_{t,u'}^\perp.\]

Note that the points $x_{t',u}$ with $t' < t$ belong to $\theta_{2,u} \supset \Pi_{t,u}$ and are hence collinear to a hyperplane $H_{t',u}$ of $\Pi_{t,u}$ with $H_{t',u} \cap Y=\<V_{2,u},V_{1,t} \cap V_{1,t'}\>\subsetneq H_{t',u}$; likewise,  the points $x_{t,u'}$ with $u' < u$  are collinear to a hyperplane $H_{t,u'}$ of $\Pi_{t,u}$ with  $H_{t,u'} \cap Y=\<V_{1,t},V_{2,u'} \cap V_{2,u}\>\subsetneq H_{t,u'}$. Since the set of hyperplanes $\{\<V_{1,t'} \cap V_{1,t}, V_{2,u}\> \mid 0 \leq t' < t\} \cup \{\<V_{1,t}, V_{2,u'} \cap V_{2,u}\> \mid 0 \leq u' < u\}$  of $\<V_{1,t},V_{2,u}\>$ is linearly independent by choice of the $V_{i,j}$, it follows that also the set $\{H_{t',u} \mid 0 \leq t' < t\} \cup \{H_{t,u'} \mid 0 \leq u' < u\}$ of hyperplanes of $\Pi_{t,u}$ is linearly independent. These facts imply that $\dim \Pi'_{t,u} = 2r'-(u+t) \geq 0$ and that $\dim(\Pi'_{t,u} \cap Y) = 2r-(u+t)-1$, i.e., $\Pi'_{t,u} \cap Y$ is a hyperplane of $\Pi_{t,u}'$ and hence $\Pi'_{t,u}$ always contains an $X$-point.  This shows the claim.

Let $t$ and $u$ be arbitrary in $\{0,...,r'\}$. We define $S^t_1:=\<x_{t,0},...,x_{t,r'}\>$ and $S^u_2:=\<x_{0,u},...,x_{r',u}\>$. By construction, $S^t_1$ is a singular subspace inside a unique member of $\Theta_1$ (namely, in $\theta_{1,t}$); moreover, since  $S^t_1 \cap Y \subseteq \bigcap_{u'=0}^{r'} (x_{t,u'}^\perp \cap Y)=V_{1,t}$, we get that $S^t_1$ is not collinear to any point of $R_1$, which is only the case if $S^t_1$ is contained in $X$ and has dimension $r'$.   Likewise, $S^u_2$ is an $r'$-dimensional $X$-space in $\theta_{2,u}\in \Theta_2$.

\textit{Claim 2: each point $x_0\in S^0_1$ is contained in a unique $r'$-dimensional $X$-space that intersects each of the $X$-spaces $S^t_1$ with $t \in \{0,...,r'\}$, and this subspace is contained in a unique member of $\Theta_2$.}\\
If $x_0 =x_{0,u}$ for some $0\leq u \leq r'$, then $x_0 \in S^u_2$ and the assertion follows. So, as a second step, suppose that $x_0$ is on a line joining two of the $r'+1$ chosen $X$-points in $S^0_1$, say, $x_0 \in \<x_{0,0},x_{0,1}\>\setminus\{x_{0,0},x_{0,1}\}$. Let $t \in \{1,...,r'\}$ be arbitrary. Then $x_{0,0}$ and $x_{t,1}$ determine a unique member $\xi\in \Xi$, since $V_{1,0} \neq V_{1,t}$ and $V_{2,0} \neq V_{2,1}$ (cf.\ Lemma~\ref{1':rel}). We get that $x_{t,0}, x_{0,1} \in x_{0,0}^\perp \cap x_{t,1}^\perp \subseteq \xi$ and as such, $x_0$ is collinear to a unique point $x_t$ on the line $\<x_{t,0},x_{t,1}\>$.  Inside $\xi$ it is clear that  $\<x_0,x_t\>$ is an $X$-line which is moreover contained in  $\<x_0,x_t\> \in \theta^{x_0}_2$ by Lemma~\ref{1':xlinesx} (otherwise $\xi$ coincides with $\theta_{1,0}=\theta^{x_0}_1$). Now, if there were a second $X$-line through $x_0$ meeting $S_1^t$ in a point $x'_t$, then $\theta^{x_0}_2 \cap \theta_{1,t}$ contains the $X$-line $\<x_t,x'_t\>$, contradicting Lemma~\ref{1':int12}. 

Secondly, we show that the points $\{x_0,x_1,...,,x_r\}$ thus obtained form an $r'$-dimensional $X$-space, which is moreover contained in a unique member of $\Theta_2$.
Let $t' \in \{1,...,r'\}\setminus\{t\}$  arbitrary. On the one hand, $x_t$ and $x_{t'}$ belong to $[x_{t,0},x_{t',1}]\in \Xi$ (as we choose  $x_t \in \<x_{t,0},x_{t,1}\>$ and $x'_t \in \<x_{t',0},x_{t',1}\>$); on the other hand, $x_t$ and $x'_t$ belong to $\theta^{x_0}_2$.  Therefore (S2) implies that $x_t \perp x_{t'}$, and hence  $\<x_0,...,x_{r'}\>$ is a singular and contained in $\theta_2^{x_0}$. As no point of $R_2$ is collinear to it (as $x_t^\perp \cap R_2 =V_{1,t}$ for $t\in\{0,...,r'\}$), we have, as before, that it has dimension $r'$ and belongs to $X$. Uniqueness follows from the fact that $x_0^\perp \cap S^t_1 = \{x_t\}$ for all $t \in \{1,...,r'\}$, as we obtained at the end of the previous paragraph. We can repeat the above argument for points on lines $\<x,x'\>$ with $x$ and $x'$ on lines joining two points of $\{x_{0,0},...,x_{0,r'}\}$, et cetera. This shows the claim.

Finally, we define $\Omega$ as the union of the $r'$-spaces intersecting  $S^t_1$ non-trivially for each $t\in \{0,...,r'\}$; so   $F^*=\<S^0_1, ..., S^{r'}_1\>$. It can be verified that $\Omega$ is also given by $S^0_1 \times S^0_2$ inside $F^*$ and that its line set coincides with the lines of $F^*$ which are contained in $\Omega$. By Fact~\ref{factuniemb},  $\Omega$ is an injective projection of $\mathcal{S}_{r',r'}(\K)$, which is moreover legal by (S2), as each grid of $\Omega$ is contained in a unique member of $\Xi$. This shows the main assertion.

$(i)$ Let $x\in X$ be arbitrary and recall that $Y_x=\<V^x_1,V^x_2\>$. As $S^0_1$ belongs to $\theta_{1,0}$, it contains a unique point $x_1$ with $V^{x_1}_2=V^x_2$. Let $\pi^{x_1}_2$ be the unique $r'$-space of $\Omega$ meeting $S^0_1$ in exactly $x_1$. Clearly, $\pi^{x_1}_2 \subseteq \theta^{x_1}_2$ (noting that $\theta^{x_1}_1=\theta_{0,1}$). As such, all points of $\pi_2^{x_1}$ are collinear to the vertex $V^{x_1}_2$ of $\theta_2^{x_1}$. Let $x'$ be the unique point of $\pi^{x_1}_2$ collinear to $V^x_1$. Then  $x'$ is the unique point of $\Omega$ collinear to $\<V^x_1,V^x_2\>$ and by Lemma~\ref{1':V1V2}, $x \in \<x',V^x_1,V^x_2\>=\<x,Y_x\>$. Note that $\<x,Y_x\> \cap \Omega = \{x'\}$ since $x'$ was unique. As $x\in X$ was arbitrary, we get $\bigsqcup_{x\in \Omega} \<x,Y_x\>\setminus Y_x =X$ indeed. Consequently,  $\<F^*,Y\>=\mathbb{P}^N(\K)$. 

$(ii)$ Take any $x\in F^* \cap X$.  By definition, $\rho^*(x)=\<x,F^* \cap Y\> \cap F$ and  $\rho(x)=\<x,Y\> \cap F$.  Since $x \in \<\rho^*(x),F^* \cap Y\> \subseteq \<\rho^*(x), Y\>$, we get $\rho^*(x) \in \<x,Y\> \cap F$, so $\rho(x)=\rho^*(x)$. 

$(iii)$ By $(ii)$, it follows from Proposition~\ref{1':projsegre} that the two families of singular $r'$-spaces of $\Omega$ correspond bijectively to the sets $\Theta_1$ and $\Theta_2$, and since $\Omega \subseteq X$, this correspondence is given by containment. 

$(iv)$ Finally, suppose that $r'=2$ (in which case $\Omega$ is isomorphic to $\mathcal{S}_{2,2}(\K)$ by Proposition~\ref{noproj} and hence $\dim(F^*)=8$) and suppose for a contradiction that $F^* \cap Y$ contains a point $y$. Then there is a line $L$ through $y$ in $Y$ meeting both planes $R_1$ and $R_2$ in a point, so $L$ occurs as the vertex of some member $\xi \in \Xi$. In $\Omega$, $\xi$ corresponds to a grid $G$. The $4$-space $\<G,y\>$ is hence contained in the $8$-dimensional subspace generated by the Segre variety $\Omega$ and intersects $\Omega$ in precisely $G$. However, using Lemma~\ref{s22p}, we then obtain a contradiction to (S2). We conclude that $F^* \cap Y=\emptyset$ in this case.\end{proof}

Henceforth we will assume that the subspace $F$ is chosen such that it is contained in $F^*$, with $F^*$ as in the previous proposition.

\subsection{Conclusion}
Finally, we show that $X$ is a mutant of the dual Segre variety $\mathcal{DS}_{r',r'}(\K)$ (see Subsection~\ref{HDS} and Definition~\ref{defmut}). 

%Henceforth we assume that $F^*$ is as described in the statement of the previous proposition. We warn the reader that we will use both $F$ and $F^*$ (taking $F$ as a subspace inside $F^*$ complementary to $Y$). We now focus on the connection between $\rho(X)$ and $Y$. Consider the partial connection map $\chi_i: \rho(X) \rightarrow R_i: \rho(x) \mapsto x^\perp \cap R_i$. Clearly, $\chi(\rho(x)) = \<\chi_1(\rho(x)),\chi_2(\rho(x))\>$ for each $x\in X$. 

\begin{theorem}\label{1':projuni}\label{1:tangent}
Let $(X,Z,\Xi,\Theta)$ be a duo-symplectic pre-DSV with with parameters $(1,v,r',v')$, for which there are at least two members of $\Theta$ through each $X$-point. Then:
\begin{compactenum}[$(i)$]
\item $X$  is the point set of a mutant of the dual Segre variety $\mathcal{DS}_{r',r'}(\K)$ with $\<Z\>$ as subspace at infinity and whose symps are given by $\Xi \cup \Theta$;
\item if additionally $(X,Z,\Xi,\Theta)$ satisfies (S3), then $r'=2$ and $X$ is projectively unique.
\end{compactenum}
\end{theorem}
\begin{proof}
$(i)$ By Proposition~\ref{1':FX}, $X$ contains a legal projection $\Omega$ of $\mathcal{S}_{r',r'}(\K)$, and $X=\bigcup_{x\in \Omega} \<x,Y_x\>\setminus Y_x$. Therefore it suffices to show that the map $\chi:\Omega \rightarrow Y: x \mapsto Y_x$ satisfies the properties mentioned in the definition of the dual Segre varieties (cf.\ Subsection~\ref{HDS}), where now $R_i$ plays the role of $Z_i$ for $i=1,2$. 

Let $S_1$ and $S_2$ be two singular $r'$-spaces of $\Omega$ intersecting each other in a point $x_0$ (i.e., belonging to different families). Possibly changing $S_1$ and $S_2$, Proposition~\ref{1':FX} says that there is a unique member $\theta_i \in \Theta_i$ containing $S_i$, $i=1,2$. Denote by $\chi_{i}$ the map taking a point of $S_i$ to $x^\perp \cap R_i$, $i=1,2$.  Inside the quadric $XY(\theta_i)$, it is clear that $\chi_{i}$ coincides with the collinearity relation between the opposite subspaces $S_i$ and $R_i$, so $\chi_i$ is a linear duality between $S_i$ and $R_i$. 

Take $x\in X$ arbitrary. If $x \in S_i$, we put $x_i:=x$ and if $x\notin S_i$, then we let $x_i$ be  the unique point of $S_i$ collinear to $x$, $i=1,2$. We claim that $\chi(x)=\<\chi_1(x_1),\chi_2(x_2)\>$, which is equivalent to showing that $x_i^\perp \cap R_i = x^\perp \cap R_i$ for $i=1,2$. We show the latter statement for $i=1$, the case where $i=2$ being analogous. If $x_1=x$, the statement is trivial, so suppose $x_1 \neq x$. Then there is a unique $\theta \in \Theta_2$ containing the $X$-line $xx'_1$, and it follows that both $x^\perp \cap R_1$ and $x_1^\perp \cap R_1$ coincide with the vertex of $\theta$. This shows the claim. For each pair of non-collinear points  $p_1,p_2$ of $X$, (S1) and Lemma~\ref{convexclosure} imply that the unique symp $\zeta$  through $p_1,p_2$ (defined as their convex closure inside $X$) has $\zeta \cap X=X([p_1,p_2])$. Assertion $(i)$ follows.

$(ii)$ Recall that we assume that $F \subseteq F^*$ is complementary to $Y$ in $\mathbb{P}^N(\K)$. By Proposition~\ref{1':projsegre}, $\rho(X)\subseteq F$ is an injective projection of a Segre variety $\mathcal{S}_{r',r'}(\K)$. Let $T^{F}_{\rho(x)}$ be the set of $\rho(X)$-lines in ${F}$ through $\rho(x)$ and denote by $T^{F}_{\rho(x)}(\xi)$ the tangent space to $\rho(X(\xi))$ at $\rho(x)$ for some $\xi\in \Xi$ with $\rho(x)\in\rho(\xi)$. 

Axiom (S3) yields members $\xi_1,\xi_2\in \Xi$ through $x$ such that $T_x=\<T_x(\xi_1),T_x(\xi_2)\>$. Since for $i=1,2$, $T_x(\xi_i)=\<Y(\xi_i),T^{F}_{\rho(x)}(\xi_i)\>$, we obtain that $T_x=\<T_x(\xi_1),T_x(\xi_2)\>$ is equivalent with $Y_x=\<Y(\xi_1),Y(\xi_2)\>$ and  $T^{F}_{\rho(x)}=\<T^{F}_{\rho(x)}(\xi_1),T^{F}_{\rho(x)}(\xi_2)\>$. On the other hand, $\dim T^{F}_{\rho(x)}=2r'$ as the tangent space at $\rho(x)$ is generated by the two $r'$-dimensional subspaces $\rho(\theta^x_1)$ and $\rho(\theta^x_2)$ of $\rho(X)$ through $\rho(x)$. Furthermore, since $r=1$,  $T^{F}_{\rho(x)}(\xi_1)$ and $T^{F}_{\rho(x)}(\xi_2)$ are just planes, which generate at most a $4$-space in $F$, and so $2r'\leq 2+2=4$. Recalling $r' > r \geq 1$, this means $r'=2$. Since $v=2r'-1=1$ and $\dim Y_x=2r'+1=3$,  the requirement $Y_x=\<Y(\xi_1),Y(\xi_2)\>$ only implies that $\xi_1$ and $\xi_2$  have disjoint vertices. 

Since $r'=2$,  the variety $\mathcal{S}_{r',r'}(\K)$ does not admit legal projections (cf.\ Proposition~\ref{noproj}) and $F^* \cap Y = \emptyset$ by Proposition~\ref{1':FX}$(iii)$, i.e., $F=F^*$. The following are projectively unique: $Y$ and $F$  in $\mathbb{P}^N(\K)$, $R_1$, $R_2$ in $Y$,   $\Omega$ in $F$. Moreover, for $i=1,2$, the projectivity $\chi_{S_i}$ between $S_i$ and the dual of $R_i$ is unique up to a projectivity of $R_i$. As such,  $X$ is projectively unique. %Let $T^F_{\rho(x)}$ be the set of $\rho(X)$-lines in $F$ through $\rho(x)$. Since each such $\rho(X)$-line through $\rho(x)$ has an $X$-line through $x$ in its inverse image by Lemma~\ref{1':rhoXlines}, it follows that $\rho(T_x)=T^F_{\rho(x)}$. Consequently, $\<Y_x, T^F_{\rho(x)}\> \subseteq T_x \subseteq \<Y,T^F_{\rho(x)}\>$. From Proposition~\ref{1':projsegre}, we obtain $\dim(T^F_{\rho(x)}) = 2r'$; furthermore we have that $\dim(Y_x) = 2r'-1$ and $\dim(Y)=2r'+1$. 
%On the other hand, (S3) implies that there are members $\xi_1,\xi_2\in \Xi$ through $x$ such that $T_x$ is  generated by $T_x(\xi_1)$ and $T_x(\xi_2)$. Since $v=2r'-3$ and $r=1$, we have that $\dim(T_x(\xi_i))=2r'$, $i=1,2$.  Let $V_1$ and $V_2$ denote the respective vertices of $\xi_1$ and $\xi_2$. Since $\dim(V_1 \cap V_2)$ is at least $2r'-5$ and since $\xi_1$ and $\xi_2$ contain $x$, the dimension of their intersection is at least $2r'-4$ (they intersect at least in $x$). Hence, $T_x$ can have dimension at most $2r'+4$. By the above, this gives $4r' \leq 2r'+4 \leq 4r' + 2$, and hence $r'=2$ (recall $r' > r=1$). 
\end{proof}

\section{The half dual Segre varieties}\label{1SS/point}

Throughout this section, we suppose that $(X,Z,\Xi,\Theta)$ is a duo-symplectic pre-DSV containing at least one $X$-point $x^*$ through which there is exactly one member $\theta^*$ of $\Theta$.

We first show that each point $x\in X$ is contained in a unique member of $\Theta$. 
\begin{lemma}\label{1:YinT}\label{1:uniqueSS}
Each $x\in X$ is contained in a unique member $\theta^x$ of $\Theta$ and $Y(\theta^x)=Y$. In particular, $\dim Y=r'+v'+1$, $\dim Y_x=r'+v'$ and $\Theta$ induces a partition of $X$.
\end{lemma}
\begin{proof}
By assumption, there is a point $x^*\in X$ through which there is a unique member $\theta^*\in \Theta$. Suppose for a contradiction that $Y(\theta^*) \subsetneq Y$. Then there is a point $z\in Z\setminus \theta^{*}$. By Lemma~\ref{lineSS}$(ii)$ (applied to any $X$-line $L$ in $\theta^{*}$ through $x^*$),  $\theta^{*}$ contains $Y_{x^*}$ and hence $x^*$ is not collinear to $z$. As such, (S1) implies that $[x^*,z]$  is a second member of $\Theta$ through $x^*$, a contradiction. We conclude that $Y=Y(\theta^{*})$ indeed. In particular, $\dim Y =r'+v'+1$ and $\dim Y_{x^*} = r'+v'$.

Let $x\in X$ be arbitrary. Lemma~\ref{xinSS} guarantees the existence of at least one $\theta \in \Theta$ through $x$. As $\dim Y(\theta) = r'+v'+1 = \dim Y$, we obtain $Y(\theta)=Y$ and hence $Y_{x}$ is a hyperplane of $Y$. Therefore, taking a point $z\in Z\setminus Y_{x}$, $\theta=[x,z]$. It follows that $\theta$ is the only member of $\Theta$ through $x$.\end{proof}

\textbf{Notation.} For any $x\in X$, we denote by $\theta^x$ the unique member of $\Theta$ through $x$. 

\subsection{The $X$-lines through a point of $X$}

We study the structure of the $X$-lines through a point $x$ of $X$. We will prove that the ones that are not contained in $\theta_x$ belong to a unique maximal singular subspace. Along the way we deduce that $v'=-1$, $v=r'-2$ and at the end we can show $r=1$.

\begin{lemma}\label{1:v'=-1}
 For any $X$-line $L$ sharing exactly a point $x \in X$ with $\theta^x$, $Y_L = Y_x=Y_{x'}$ for any $x' \in L$. Moreover, each point $x' \in X\setminus X(\theta^x)$ is collinear to the vertex $V^x$ of $\theta^x$.
\end{lemma}
\begin{proof}
Note that $L$ is a $0$-line, for otherwise there are at least two members of $\Theta$ through $x$, contradicting Lemma~\ref{1:YinT}. It then follows from Lemma~\ref{lineSS}$(i)$ that  $Y_L = Y_x$, and $Y_{x'}=Y_{x}$ for each $x' \in L$.

Take any point $x' \in X\setminus \theta^x$. Suppose first that $x'$ is collinear to some point $x\in X(\theta^x)$.  By Lemma~\ref{1:YinT} and $x' \notin \theta^x$, $xx'$ is an $X$-line. By the above paragraph, $Y_{x'}=Y_x$, so in particular, $x' \perp V^x$. Next, suppose that $x'$ is not collinear to any point of $X(\theta^x)$. Taking $x \in X(\theta^x)$ arbitrary, (S1) and Lemma~\ref{1:YinT} imply that $x$ and $x'$ are contained in a member $\xi$ of $\Xi$. Let  $L_1$ and $L_2$ be two non-collinear $X$-lines of $\xi$ through $x$. For $i=1,2$, we have that $L_i$ is collinear to $V$ because, if $L_i$ belongs to $\theta^x$ then $L_i \perp V^x$ by definition and if $L_i$ does not belong to $\theta^x$ then $L_i \perp V^x$ by the first paragraph. Consequently, $V$ is contained in the vertex of $\xi$ and, in particular, $x' \perp V^x$. 
\end{proof}
\par\bigskip

By Lemma~\ref{resY}, we can project $(X,Z,\Xi,\Theta)$ from $V$ and obtain a pre-DSV with parameters $(r,v-v'-1,r',-1)$; with which we will continue to work without changing  notation, i.e., we just assume that $(X,Z,\Xi,\Theta)$ has $v'=-1$.

\begin{cor}\label{1:Y=Z}
We have $Y=Z$ and $\dim Y=r'$.
\end{cor}
\begin{proof}
Take $\theta \in \Theta$ arbitrary. By Lemma~\ref{1:uniqueSS} and $v'=-1$, $Y=Y(\theta) \subseteq Z$, so $Y=Z$ as $\<Z\>=Y$. The same lemma says $\dim Y=r'+v'+1=r'$.
\end{proof}

\begin{lemma}\label{1:externallines}
Each pair of $X$-lines $L_1,L_2$ through $x$ not contained in $\theta^x$ is contained in a singular plane. Moreover, $v=r'-2$.
\end{lemma}
\begin{proof}
Take any $X$-line $L$ through $x$ inside $\theta$ which is not collinear to $L_1$ (cf.\ Lemma~\ref{convexclosure}). By Lemma~\ref{planes}, $L$ and $L_1$ determine a unique member $\xi$ of $\Xi$ with vertex $Y_L \cap Y_{L_1}$. Lemma~\ref{1:v'=-1} says $Y_{L_1}=Y_{L_2}=Y_x$, so $v=r'-2$. If $L_1$ and $L_2$ would not be collinear, the same argument implies that $[L_1,L_2]\in \Xi$ has vertex $Y_x$, but then $v=r'-1$, a contradiction.  Hence $L_1$ and $L_2$ are collinear indeed.  \end{proof}

We consider the set $\pi(H):=\{x\in X \mid Y_x=H\} \cup H$ for any hyperplane $H$ of $Y$. 

\begin{lemma}\label{1:singsub}
Let $H$ be a hyperplane of $Y$. Then $\pi(H)$ is a (maximal) singular subspace, which intersects each $\theta \in \Theta$ in a maximal singular subspace of $\theta$ of the form $\<x,H\>$ with $x\in X(\theta)$.%, implying that each point $x' \in X\setminus\theta$ is collinear to a maximal singular subspace of $\theta$ (namely, $\pi(Y_{x'}) \cap \theta$).
\end{lemma}
\begin{proof}
Take any $\theta \in \Theta$. Since $Y(\theta)=Y$ by Lemma~\ref{1:uniqueSS}, $X(\theta)$ contains a point $x$ with $Y_x=H$, and for such a point $x$, clearly $X(\theta) \cap \pi(H)=\<x,H\>$. Take two points $x_1,x_2$ with $Y_{x_1}=Y_{x_2}=H$. Let $\theta_i$ be the unique member of $\Theta$ containing $x_i$ (cf.\ Lemma~\ref{1:uniqueSS}), for $i=1,2$. If $\theta_1=\theta_2$, then $x_2 \in \<x_1,H\>$ and hence $x_1x_2$ is a singular line (with a unique point in $Y$). So suppose $\theta_1 \neq \theta_2$. If $x_1$ and $x_2$ are not collinear, then they determine a member of $\Xi$, which has $H$ as its vertex, contradicting $v=r'-2$ (cf. Lemma~\ref{1:externallines}). Hence $\pi(H)$ is a singular subspace indeed, maximal by definition. 
\end{proof}
\par\bigskip

\textbf{Notation.}   For $x\in X$, we denote by $\pi^x$ the subspace $\pi(Y_x)$. We define $\Pi$ as the set $\{\pi(H) \mid H \text{ a hyperplane of } Y\}$. 

The sets $\Theta$ and $\Pi$ play a role similar as $\Theta_1$ and $\Theta_2$ in the previous section. For instance, the following corollary is the analogue of Lemma~\ref{1':xlinesx}.

\begin{cor}\label{1:Xlines}
Let $x \in X$ be arbitrary. Then each $X$-line through $x$ is contained in exactly on of $\theta^x, \pi^x$. 
\end{cor}
\begin{proof}
Let $L$ be any $X$-line through $x$ and suppose $L$ is not contained in $\theta^x$. By Lemma~\ref{1:v'=-1}, $Y_{x'}=Y_x$ for all points $x'\in L$. So, by definition,  $L \subseteq \pi^x$. By Lemma~\ref{1:singsub}, $\theta^x \cap \pi^x=\<x,Y_x\>$, so no $X$-line is contained in both $\theta^x$ and $\pi^x$. 
\end{proof}
\par\bigskip

The above implies that $r=1$:

\begin{lemma}\label{1:r=1}
Suppose $\xi\in \Xi$ and $x\in X(\xi)$. Then $\xi \cap \theta^x$ is a maximal singular subspace of both $\theta^x$ and $\xi$ of the form $\<L,Y_L\>$, with $L$ an $X$-line. In particular, $r=1$. 
\end{lemma}
\begin{proof}
Let $V$ be the vertex of $\xi$ and observe that  $\<x,V\> \subseteq \<x,Y_x\>= \theta^x \cap \pi^x$. Moreover by  Corollary~\ref{1:Xlines},  $\pi^x \cup \theta^x$ contains all $X$-lines of $\xi$ through $x$, and by the foregoing, in fact all singular lines of $\xi$ through $x$. The singular subspaces $M_1:=\pi^x \cap \xi$ and $M_2:=\theta^x \cap \xi$  hence contains all singular lines of $X(\xi)$ through $x$. This implies that $r=1$, so $M_i=\<L_i,V\>$ for (non-collinear) $X$-lines $L_i$ of $X(\xi)$ through $x$. As $\dim(M_i)=2+v=r'$ (cf.\ Lemma~\ref{1:externallines}), $M_2$ is a maximal singular subspace in both $\xi$ and~$\theta^x$, and hence $V=Y_{L_2}$.
\end{proof}

\subsection{The projection $\rho(X)$ and its connection to $Y$}

What follows contains many similarities compared to the situation in Section~\ref{1':1line}. This is caused by the fact that we will obtain a half dual Segre variety $\mathsf{HSD}_{r',k}(\K)$ for some natural number $k \geq 1$, and in case $k=r'$ this is the projection of the dual Segre variety $\mathsf{HSD}_{r',r'}(\K)$ (as encountered in the previous section) from one of its two $r'$-spaces $R_1$ or $R_2$ in $Z$ (which causes $\Theta_2$ to `collapse' to $\Pi$).  Despite the similarities between both cases there is no upshot in treating them simultaneously as it would boil down to a similar amount of work and obscure some of the arguments. We give the proofs up to the part where there are differences, and point out the parts which are exactly the same.

\par\medskip
We again consider the maps $\rho$ and $\chi$ (cf.\ Definitions~\ref{RHO} and~\ref{CHI}).

 Next, we have the counterpart of Lemma~\ref{1':rhoXlines}. Except for a slight change in assertion $(iii)$, the statement is analogous, however, the proof has to take into account the additional possibility that $Y_{x_1}=Y_{x_2}$.
 
\begin{lemma}~\label{1:rhoXlines}
Suppose $\rho(x_1)$ and $\rho(x_2)$ determine a singular line of $\rho(X)$, for $x_1,x_2 \in X$. Let $x'_i  \in \rho^{-1}(\rho(x_i))$ be arbitrary, for $i=1,2$. Then:

\begin{compactenum}[$(i)$]
\item there is a unique $\zeta \in \Theta \cup \Pi$ containing $x'_1 \cup x'_2$, and $\rho^{-1}(\rho(x_1)) \cup  \rho^{-1}(\rho(x_2))\subseteq \zeta$, with $\zeta\in \Theta$ if and only if $Y_{x_1} \neq Y_{x_2}$;
\item for each $x'_1\in\rho^{-1}(\rho(x_1))$, we can choose $x''_2\in\rho^{-1}(\rho(x_2))$ such that $\<x'_1,x''_2\>$ is an $X$-line;
\item  If $\zeta\in \Theta$, then $\{Y_x \mid \rho(x) \in \<\rho(x_1),\rho(x_2)\>\}$ is the set of all $(r'-1)$-spaces through the $(r'-2)$-space $Y_{x_1}\cap Y_{x_2}$ inside $Y$. 
\end{compactenum}
\end{lemma}

\begin{proof}
Again, $x'_i\in\rho^{-1}(\rho(x_i)) = \<x_i, Y_{x_i}\> \cap X$ for $i=1,2$ by Lemma~\ref{1:inj}. We distinguish two cases.

Suppose first that $Y_{x_1} = Y_{x_2}$. Then $\pi^{x_1}$ is the unique member of $\Pi \cup \Theta$ containing both $x_1$ and $x_2$. Moreover, $\<x_1,Y_{x_1}\>$ and  $\<x_2,Y_{x_2}\>=\<x_2,Y_{x_1}\>$ are contained in  $\pi^{x_1}$ by definition of the latter. Assertion $(i)$ follows. Since $\pi^{x_1}$ is a singular subspace and $\<x_1,Y_{x_1}\> \cap \<x_2,Y_{x_1}\> =Y_{x_1}$, any  $x''_2\in\<x_2,Y_{x_2}\>\cap X$ is on an $X$-line with $x'_1$, so $(ii)$ follows.

Next, suppose $Y_{x_1}\neq Y_{x_2}$. We first show assertion $(ii)$. To that end, consider the $(r'+2)$-space $\<x_1,x_2,Y\>$ and note that $\<x_1,x_2,Y\>=\<x_1,x_2,Y_{x_1}, Y_{x_2}\>$. Let $x_3$ be an $X$-point with $\rho(x_3)$ on $\<\rho(x_1),\rho(x_2)\>\setminus\{\rho(x_1),\rho(x_2)\}$. Then $x_3\in \<x_1,x_2,Y\>\setminus (\<x_1,Y_{x_1}\> \cup \<x_2,Y_{x_2}\>)$. Inside $\<x_1,x_2,Y\>$, we see that the $(r'+1)$-space $\<x_3,x_1, Y_{x_1}\>$ intersects the $r'$-space $\<x_2, Y_{x_2}\>$ in an $(r'-1)$-space $M_2$ intersecting $Y$ in the $(r'-2)$-space $Y_{x_1} \cap Y_{x_2}$. Then the line $\<x'_1,x_3\>$ intersects $M_2$ in a point of $X$, say $x''_2$, and as such, it is a singular line. If it would not be an $X$-line, then it is contained in $\<x_1,Y_{x_1}\>$, contradicting the choice of $x_3$. Assertion $(ii)$ follows. 

By Corollary~\ref{1:Xlines}, the $X$-line $x'_1x''_2$ belongs to a member $\zeta$ of $\Theta \cup \Pi$, and since $Y_{x_1} \neq Y_{x_2}$, $\zeta \in \Theta$. Uniqueness follows from $\zeta=\theta^{x'_1}$. Moreover, $\zeta$ contains $\<x'_1,Y_{x_1}\> \cup \<x'_2,Y_{x_2}\>$.

$(iii)$ As in the proof of Lemma~\ref{1':rhoXlines}.  \end{proof}
\par\bigskip

Recall that, for a vertex $V$ of any $\xi \in \Xi$, we denote by $\Xi_V$ the subset of $\Xi$ whose members have vertex $V$ and by $X_V$ the $X$-points collinear to $V$. 

The following is a  weaker version of Lemma~\ref{1':welldef}. Indeed, this time, there is no isomorphism between the points of $\rho(X(\xi))$ and the $(r'-1)$-spaces in $Y$ through $Y(\xi)$; and the set of points of $X$ collinear to $Y(\xi)$ is now given by $\bigcup_{T\subseteq Y_x} \pi^x \cap X$, whose image under $\rho$ is the disjoint union of subspaces which are not necessarily lines (cf.\ Remark~\ref{XT}).

\begin{lemma}~\label{1:welldef}
Suppose $\xi=[x_1,x_2]\in \Xi$ for points $x_1,x_2\in X$ and put $T=Y(\xi)$. Then:
\begin{compactenum}[$(i)$]
\item If $x'_i \in \rho^{-1}(\rho(x_i))$ for $i=1,2$, then $\xi':=[x'_1,x'_2]$ belongs to $\Xi_T$;
\item $\rho(X(\xi'))=\rho(X(\xi))$;
\end{compactenum}
\end{lemma}
%\item The set of $X$-points collinear to $T$ is mapped by $\rho$ on a Segre variety $\mathcal{S}_{1,k}(\K)$ where $k=\dim(\pi)-r'$ for $\pi \in \Pi$.
\begin{proof} 
$(i)$ By Lemma~\ref{1:inj}, $x'_i \in \<x_i, Y_{x_i}\> \cap X$ and $Y_{x'_i}=Y_{x_i}$ for $i=1,2$.   Note that  $T=Y_{x_1} \cap Y_{x_2}$; in particular, $Y_{x_1} \neq Y_{x_2}$. Firstly, suppose $x'_1$ and $x'_2$ belong to a member $\theta \in \Theta$. Then $\theta$ also contains $\<x'_i,Y_{x'_i}\> \ni x_i$ for $i=1,2$, and hence $\theta=\xi$, a contradiction. Secondly, suppose the line  $\<x'_1,x'_2\>$ is singular. Then, since  $Y_{x'_1} \neq Y_{x'_2}$, it is an $X$-line. However, it is not contained in a member of $\Theta$ by the foregoing, and by Lemma~\ref{1':v'=-1}, this implies $Y_{x_1}=Y_{x_2}$, a contradiction. So $x'_1$ and $x'_2$ are also non-collinear points of some member $\xi'$ of $\Xi$, with vertex~$T=Y_{x'_1}\cap Y_{x'_2}$. 

$(ii)$ We show that $\rho(X(\xi))=\rho(X(\xi'))$.  Let $x \in X(\xi)$ be a point contained in $x_1^\perp \cap x_2^\perp$. According to Corollary~\ref{1:Xlines} and renumbering if necessary, $xx_1 \in \theta^x$ and $xx_2 \in \pi^x$. Then $\theta^x=\theta^{x_1}$ and $\pi^x=\pi^{x_2}$, so $\theta^{x_1} \cap \pi^{x_2}=\<x,Y_x\>$. Therefore, $\<x,T\>$ is the unique generator of $X(\xi)$ in $\theta^{x_1} \cap\pi^{x_2}$.  Likewise,  $X(\xi')$ has  a unique generator, say $\<x',T\>$, contained in $\theta^{x'_1} \cap\pi^{x'_2}$. Recalling that $x'_i \in \<x_i,Y_{x_i}\>$ for $i=1,2$, we obtain $x' \in \theta^{x'_1} \cap\pi^{x'_2} = \theta^{x_1} \cap \pi^{x_2}=\theta^x \cap \pi^x=\<x,Y_x\>$, and hence $\rho(x)=\rho(x')$. 

Next, for $i=1,2$, we claim that for each point $u_i$ on the $X$-line $xx_i$, there is a unique point $u'_i$ on the $X$-line $x'x'_i$ with $\rho(u_i) = \rho(u'_i)$, i.e., $u'_i \in \<u_i,Y_{u_i}\>$. By assumption, $x'_i\in \<x_i,Y_{x_i}\>$ for $i=1,2$ and by the previous paragraph, $x'\in\<x,Y_x\>$; in particular $x'x'_i\subseteq \<x,x_i,Y_x,Y_{x'_i}\>$. For $i=1$, we recall that $Y_{x_1}=Y_{x'_1}\neq Y_x = Y_{x'}$ and hence, looking inside $\theta^x$, one sees that collinearity gives a bijection between the hyperplanes of $Y$ through $T=Y_{x} \cap Y_{x'}$ and the points of $xx_1$, and also those of $x'x'_1$. So, for any $u \in xx_1$, there is a unique $u' \in x'x'_1$ with $Y_u=Y_{u'}$.  We claim that $u' \in Y_u$. Suppose the contrary. Then $uu'$ is an $X$-line. There are two options.  Firstly, assume that $xx'$ and $x_2x'_2$ generate a plane. Without loss of generality, $x\neq x'$ (if not then $x_2 \neq x'_2$ and the same argument holds), and hence the $X$-line $uu'$ has a point $u'' \in xx' \cap X$. But then $u''$ is collinear to both $Y_u$ and $Y_x$, a contradiction. So secondly, assume that $xx'$ and $x_2x'_2$ are disjoint lines. Then $y:=xx' \cap Y$ and $y_2:=x_2x'_2 \cap Y$ are disjoint points of $Y$. Recalling that we are in the hyperbolic quadric $\theta^x$, we obtain that the $X$-line $uu'$ contains a point of $yy_2$, a contradiction. The claim follows. 
For $i=2$, the lines $xx_2$ and $x'x'_2$ belong to the singular subspace $\pi^x$ and hence for each point $u\in xx_2 \cup x'x'_2$ we have $Y_u=Y_x$. Since $x'\in \<x,Y_x\>$ and $x'_2\in\<x_2,Y_x\>$, we more precisely have $x'x'_2 \subseteq\<xx_2,Y_x\>$. Hence, for each point $u\in xx_2$, we obtain that $\<u,Y_x\>$ contains a unique point $u'$ of $x'x'_2$ and clearly, $\rho(u)=\rho(u')$.

Finally, let $u$ be an arbitrary point on $X(\xi)\setminus(\<xx_2,T\> \cup \<xx_1,T\>)$. Then $u$ is collinear to unique points $u_2$ on $xx_2$ and $u_1$ on $xx_1$. As in the first paragraph of $(ii)$ (but switching the roles of $1$ and $2$),  $\<u,T\>$ is the unique generator of $X(\xi)$ contained in $\theta^{u_2} \cap \pi^{x_1}$ and $X(\xi')$ contains a unique generator, say $\<u',T\>$,  contained in $\theta^{u'_2} \cap \pi^{u'_1}$, where $\rho(u'_i)=\rho(u_i)$ (as in the previous paragraph). Since $\theta^{u'_2}=\theta^{u_2}=\theta^u$ and $\pi^{u'_1}=\pi^{u_1}=\pi_u$, we obtain that $u'\in \theta^{u'_2} \cap \pi^{u'_1}=\<u,Y_u\>$, i.e., $\rho(u')=\rho(u)$. We conclude that $\rho(X(\xi))=\rho(X(\xi'))$. This shows the assertion.

$(iii)$ Suppose $x \in X$ has $\rho(x)=\rho(x')$ for some $x'\in X(\xi)$. Then $T\subseteq Y_{x'}=Y_x$, so $x\in X_T$ indeed. 
 \end{proof}
 
Recall that we denote by $\mathsf{L}$ be the set of $X$-lines. We  encounter the counterpart of Proposition~\ref{1':projsegre}. The proof is entirely different, since $Y$ is a highly non-injective projection of $\mathcal{S}_{r',k}(\K)$.

\begin{prop}\label{1:projsegre}
The point-line geometry $\mathcal{S}:=(\rho(X),\rho(\mathsf{L}))$ is isomorphic to an injective projection of the Segre geometry $\mathcal{S}_{r',k}(\K)$ where $k=\dim(\pi)-r'$ for any $\pi \in \Pi$. Moreover, we have
\begin{compactenum}
\item[$(i)$] for each maximal singular subspace $S$ in $\mathcal{S}$, there is a unique $\zeta_S \in \Theta\cup \Pi$ with $\rho^{-1}(S)=X(\theta_S)$.
 \item[$(ii)$] the sets $\mathcal{S}\Theta:=\{\rho(X(\theta)) \mid \theta \in \Theta\}$ and $\mathcal{S}\Pi:=\{\rho(X(\pi)) \mid \pi \in \Pi\}$ are the two natural families of maximal singular subspaces of $\mathcal{S}$.
\item[$(iii)$] for each grid $G$ of $\mathcal{S}$, there is a unique  $v$-space $V$ in $Y$ with $\rho^{-1}(G)\subseteq X_V$, and vice versa.
\end{compactenum}
\end{prop}
\begin{proof}
We start by determining the maximal singular subspaces of  $\mathcal{S}$. Let $\zeta \in \Theta \cup \Pi$ be arbitrary. By Lemma~\ref{proprhochi} and the fact that members of $\Pi$ are singular subspaces, $\rho(X(\theta))$ is a singular subspace of $\rho(X)$, which we denote by $S_\zeta$. 

\textit{Claim 1: each point $p \in \rho(X)\setminus S_\zeta$ is collinear to at most one point of $S_\zeta$.}\\
Suppose for a contradiction that there is a  $p \in \rho(X)\setminus S_\zeta$ collinear to two points $s_1$ and $s_2$ of $S_\zeta$. Take $x \in X$ with $\rho(x)=p$.  Lemma~\ref{1':rhoXlines} implies that, for $i=1,2$, we can choose $x_i \in \rho^{-1}(s_i)$ such that $xx_i$ is an $X$-line, and by the same lemma, $x_1,x_2 \in \zeta$. For $i=1,2$, Corollary~\ref{1:Xlines} and  $x \notin \zeta$ imply that $xx_i$ belongs to either $\pi^{x_i}=\pi^x$ (in case $\zeta\in \Theta$) or to $\theta^{x_i}=\theta^x$ (in case $\zeta\in\Pi$). So, if $\zeta \in \Theta$ then $x_1x_2 \subseteq \pi^x \cap \zeta$ and if $\zeta\in \Pi$ then $x_1x_2 \subseteq \theta^x \cap \zeta$, and hence by Lemma~\ref{1:singsub}, $s_1=\rho(x_1)=\rho(x_2)=s_2$, a contradiction. This shows the claim.

Let $S$ be an arbitrary singular subspace of $\rho(X)$ containing a line $L$. By  Lemma~\ref{1:rhoXlines}, $L$ is contained in $S_\zeta$ for a \emph{unique} $\zeta\in \Theta \cup \Pi$. In case $L \subsetneq S$,  the above claim implies $S \subseteq S_\zeta$. We conclude that each maximal singular subspace of $\mathcal{S}$ is given by $S_\zeta$ for a unique $\zeta \in \Theta\cup \Pi$, and that $S_\zeta \cap S_{\zeta'}$ is at most a point if $\zeta,\zeta'$ are distinct members of $\Theta\cup\Pi$. Assertion $(i)$ is proven.

Recall that $\mathcal{S}\Theta:=\{S_\theta \mid\theta \in \Theta\}$ and $\mathcal{S}\Pi:=\{S_\pi \mid\pi \in \Pi\}$. Let $p\in \rho(X)$ be arbitrary and take $x\in \rho^{-1}(p)$. Let $p\in \rho(X)$ be arbitrary and take $x\in \rho^{-1}(p)$. 
Note that $p \in S_\zeta$ for some $\zeta \in \Theta$ means that $x\in \zeta$ (since $\zeta$ contains $\<x,Y_x\>=\rho^{-1}(p)$) and recall that $\theta^x$ and $\pi^x$ are the unique members of $\Theta$ and $\Pi$ through $x$, respectively (cf.\ Lemma~\ref{1:Xlines}). Then the previous paragraph implies that all singular subspaces of $\mathcal{S}$ through $p$ of dimension at least 1 are contained in either $S^p_\Theta:=S_{\theta^x}$ or $S^p_\Pi:=S_{\pi^x}$. This in particular implies that two maximal singular subspaces of $\mathcal{S}$ can only have a point in common if one of them belongs to $\mathcal{S}\Theta$ and the other to $\mathcal{S}\Pi$. 
Moreover, for any pair $(S_\theta,S_\pi) \in \mathcal{S}\Theta \times \mathcal{S}\Pi$, Lemma~\ref{1:singsub} implies that $S_\theta \cap S_\pi$ contains at least one point, which is unique by the previous paragraph. Assuming that two maximal singular subspaces of $\mathcal{S}$ are in the same ``natural family'' of $\mathcal{S}$ (from which we have not proved what it is yet) if and only if they are disjoint, this shows~$(ii)$.

We can now determine the structure of $\rho(X)$. To that end, let $(S_\Theta,S_\Pi)\in\mathcal{S}\Theta \times \mathcal{S}\Pi$ be arbitrary and denote their unique intersection point by $p$.

\textit{Claim 2: $\rho(X)$ is the direct product of $S_{\Theta}$ and $S_{\Pi}$.}\\
Let $q\in\rho(X)$ be arbitrary. If $q\in S_{\Theta} \cup S_{\Pi}$, then $q=(p,s) \in \{p\} \times S_{\Pi}$ for some $s \in S_{\Pi}$ or $q=(s,p) \in S_{\Theta} \times \{p\}$ for some $s \in S_{\Theta}$. So suppose $q \notin S_\Theta \cup S_\Pi$. As mentioned above, and using the same notation, $S^q_\Theta\in \mathcal{S}\Theta$ and $S^q_\Pi \in \mathcal{S}\Pi$ are the two maximal singular subspaces of $\mathsf{S}$ through $q$. Moreover, as above, $S^q_\Theta \cap S_\Pi$ is a unique point, say $s^q_\Pi$; likewise, $S^q_\Pi \cap S_\Theta$ is a unique point, say $s^q_\Theta$. The points $s^q_\Theta$ and $s^q_\Pi$ determine $q$ uniquely: $S^q_\Pi$ is the unique member of $\mathcal{S}\Pi$ through $s^q_\Theta$ and  $S^q_\Theta$ is the unique member of  $\mathcal{S}\Theta$ through $s^q_\Pi$, and  $S^q_\Theta \cap S^q_\Pi=\{q\}$. This shows the claim.

Observe that this also implies that $p$ and $q$ are the unique points of $\rho(X)$ collinear to both $s^q_{\Theta}$ and $s^q_{\Pi}$. Another important consequence is that $\dim(S_\pi)=\dim(S_{\pi'})$ for all $\pi,\pi' \in \Pi$, as such the value $k:=\dim(S_\pi)=\dim(\pi)-(\dim(Y)-1)-1$ is well-defined.

Take any grid $G$ in $\mathcal{S}$. Since $p$  was a generic point in the above, we may assume that $L_\Theta:=ps_\Theta \subseteq S_\Theta$ and $L_\Pi:=ps_\Pi \subseteq S_\Pi$ are two lines of $G$. Note that $G$ is determined as the convex closure of $s_\Theta$ and $s_\Pi$ in $\mathcal{S}$.  Put $V=\chi(s_\Theta) \cap \chi(s_\Pi)$.

\textit{Claim 3:  $\rho^{-1}(G)\subseteq X_V$}. \\
Let $x_\Theta$ and $x_\Pi$ be $X$-points such that $\rho(x_\Theta)=s_\Theta$ and $\rho(x_\Pi)=s_\Pi$. Since $s_\Theta$ and $s_\Pi$ are distinct points not on a line of $\rho(X)$, $[x_\Theta,x_\Pi]$ is a member $\xi$ of $\Xi$ with vertex $Y_{x_\Theta} \cap Y_{x_\Pi}=V$. By Lemma~\ref{1:welldef}, $\rho(X(\xi))$ does not depend on the choice of $X$-points in the inverse images of $s_\Theta$ and $s_\Pi$. Since $\rho(X(\xi))$ is a full grid of $\rho(X)$-lines containing the points $s_\Theta$ and $s_\Pi$, the observation at the end of Claim 2 implies that $\rho(X(\xi))=G$.  It then follows that $\rho^{-1}(G) \subseteq X_V$.  This shows the claim.

Since, each subspace $v$-space $V$ in $Y$ occurs as $V_1\cap V_2$ of two $(r'-1)$-spaces $V_1$ and $V_2$ of $Y$, Lemma~\ref{1:singsub} implies that there are points $x_1,x_2\in X$ with $Y_{x_1}\cap Y_{x_2}=V$. Clearly, $[x_1,x_2]\in \Xi$. Then $V$ corresponds with the grid of $\mathcal{S}$ determined by $\rho(x_1)$ and $\rho(x_2)$. This completes  the proof of assertion $(iii)$.

Lastly, we claim that $\rho(\mathsf{L})$ is the union of $\{\{s_\Theta\} \times L_\Pi \mid  s_\Theta \in S_\Theta, L_\Pi \text{ line of } S_\Pi\}$ and  $\{L_\Theta \times \{s_\Pi\} \mid L_\Theta \text{ line of } S_\Theta, s_\Pi \in S_\Pi\}$. Indeed, each line $L$ of $\rho(\mathsf{L})$ is contained in a unique member of $\mathcal{S}\Theta \cup \mathcal{S}\Pi$, suppose $L \subseteq S'_\Theta$ for some $S'_\Theta \in \mathcal{S}\Theta\setminus\{S_\Theta\}$ (if $L \subseteq S_\Theta$ then $L_\Theta=\{p\} \times L_\Theta$). Let $s_\Pi$ be the unique intersection point $S_\Pi \cap S'_\Theta$ and put $L_\Theta:=\{s^q_\Theta \mid q \in L\}$. The latter set is a \emph{line} of $S_\Theta$ indeed: take any $q \in L$ and consider the grid $L \times qs_\Theta^q$, which is a well-defined subset of $\rho(X)$ by Claim 2 (reversing the roles of $p$ and $q$ for a moment), and contains $L_\Theta$ as one of its lines. So $L=L_\Theta \times \{s_\Pi\}$.

Since $\dim(S_\Theta)=r'$ and $\dim(S_\Pi)=k$, this concludes the proof of the proposition.
\end{proof}

\begin{rem}\label{k>0}\em
Note that $k \geq 1$: considering any $\theta\in \Theta$ and any point $x\in X\setminus \Theta$ (which exists, for otherwise $|\Xi|=0$), Lemma~\ref{1:singsub} implies that $\pi^x$ contains a point $x' \in X(\theta)$, and hence $xx'$ is an $X$-line.
\end{rem}

\begin{rem}\label{XT}\em
In view of the previous proposition, it does not require too much additional effort to prove the following. For any vertex $T$ of a $\xi \in \Xi$,  $\rho(X_T)$ is, as a subgeometry of $(\rho(X),\rho(\mathcal{L}))$, isomorphic to the Segre geometry $\mathcal{S}_{1,r'}(\K)$, and is given by the direct product of $\rho(\pi^x)$ and $\rho(L)$, where $x$ is any point in $X_T$ (i.e., $T\subseteq Y_x)$ and where $L$ is any $X$-line contained in a member of $\Theta$ with $Y_L=T$. We however do not include a proof, as we do not need this in the sequel.
\end{rem}

\begin{cor}\label{NHDS}
We have $N\leq (r'+1)(k+2)-1$. 
\end{cor}
\begin{proof}
By Lemma~\ref{1:uniqueSS}, we know $\dim(Y)=r'$ and by Proposition~\ref{1:projsegre}, $\dim(F)\leq (r'+1)(k+1)-1$. Since $F$ and $Y$ generate $\mathbb{P}^N(\K)$, we obtain $N\leq(r'+1)(k+2)-1$. 
\end{proof}

Recall that $(\rho(X),\rho(\mathsf{L}))$ is a \emph{legal} projection of $\mathcal{S}_{r',k}(\K)$ if (S2) also holds here (cf.\ Definition~\ref{legal}). The next proposition is the counterpart of Proposition~\ref{1:FX}.

\begin{prop}\label{1:FX}
The set $X$ contains a legal projection $\Omega$ of $\mathcal{S}_{r',k}(\K)$ which is such that:\begin{compactenum}[$(i)$]
\item $\bigsqcup_{x\in \Omega} \<{x},Y_{{x}}\>\setminus Y_x=X$ and hence, putting $F^*=\<\Omega\>$, $\<F^*,Y\>=\mathbb{P}^N(\K)$;
\item Re-choosing the subspace $F$ so that it is inside $F^*$, the projection $\rho^*$ of $F^* \cap X$ from $F^*\cap Y$ onto $F$ is the restriction of $\rho$ to $F^* \cap X$;
\item  Containment gives a bijection between the two natural families of $r'$-spaces and $k$-spaces of $\Omega$ and the sets $\Theta$ and $\Pi$;
\item If $r' =2$ and $k \in \{1,2\}$, then $F^* \cap Y = \emptyset$.
\end{compactenum}
\end{prop}
\begin{proof}
By Proposition~\ref{1:projsegre}, $\rho(X)$ is the point set of an injective projection of the Segre geometry $\mathcal{S}_{r',k}(\K)$, and the elements $\zeta \in \Theta \cup \Pi$ are in $1-1$-correspondence to the set of maximal singular subspaces $S_\zeta$. We  construct a legal projection of $\mathcal{S}_{r',k}(\K)$ \emph{inside} $X$ (where we know that (S2) holds), by using well-chosen $r'$-dimensional $X$-spaces in certain members of $\Theta$. 

Take a basis of hyperplanes $V_0,...,V_{r'}$ in $Y$. For $t \in \{0,...,r'\}$, and let $\pi_t$ be short for the subspace $\pi({V_t})$ consisting of $V_t$ and the  $X$-points collinear to $V_t$ (cf.\ Lemma~\ref{1:singsub}). Take any $k$-dimensional $X$-space $S^0_\Pi$ in $\pi_0$ complementary to $V_0$ and let $(x_{0,0},...,x_{0,k})$ be a basis of $S^0_\Pi$. For each $u\in \{0,...,k\}$, let $\theta_u$  be short for $\theta^{x_{0,u}}$, the unique member of $\Theta$ through $x_{0,u}$; and for each $t\in \{0,...,r'\}$, put $\Pi_{t,u}:=\pi_t \cap \theta_u$. Recall that Lemma~\ref{1:singsub} says that $\Pi_{t,u}:=\<x'_{t,u},V_t\>$ where $x'_{t,u}$ is any $X$-point in $\theta_u$ collinear to $V_t$. In particular, $\dim \Pi_{t,u}=r'$.
Just like  in Claim 1 of the proof of Proposition~\ref{1':FX}, we can consecutively select $X$-points $x_{t,u} \in \Pi_{t,u}$, using the lexicographic order on the pairs $\{(t,u)\mid 0\leq t \leq r', 0 \leq u \leq k\}$ in such a way that $x_{t,u} \perp x_{t',u}$ with $0 \leq t' <t$ (the condition $x_{t,u} \perp x_{t,u'}$ with $0\leq u' <u$ being trivially fulfilled as $\pi_t$ is a singular subspace). 

Let $t \in \{0,...,.r'\}$ and $u\in \{0,...,k\}$ be arbitrary. We define $S^t_\Pi:=\<x_{t,0},...,x_{t,k}\>\subseteq \pi_t$ and $S^u_\Theta:=\<x_{0,u}, ...,x_{r',u} \>\subseteq \theta_u$. As in the proof of Proposition~\ref{1':FX}, we obtain that $S^t_\Pi$ and $S^u_\Theta$ are $X$-spaces of dimensions $k$ and $r'$, respectively. Using an argument similar to the one in Claim 2 of that same proof, we obtain that  for each point of $S^0_\Pi$, there is a unique $X$-space of dimension  $r'$  through it which intersects $S^{t'}_\Pi$ non-trivially for each $t' \in \{0,...,r'\}$, and this subspace is contained in a unique member of $\Theta$. 

Finally, we define $\Omega$ as the union of the $r'$-spaces intersecting  $S^t_\Pi$ non-trivially for each $t\in \{0,...,r'\}$; so   $F^*=\<S^0_\Theta, ..., S^{k}_\Theta\>$. It can be verified that $\Omega$ is also given by $S^0_\Theta \times S^0_\Pi$ inside $F^*$ and that its line set coincides with the lines of $F^*$ which are contained in $\Omega$. By Fact~\ref{factuniemb},  $\Omega$ is an injective projection of $\mathcal{S}_{r',k}(\K)$, which is moreover legal by (S2), as each grid of $\Omega$ is contained in a unique member of $\Xi$. This shows the main assertion.

$(i)$ Let $x\in X$ be arbitrary. Since $S^0_\Theta$ is an $r'$-dimensional $X$-space of $\Theta_0$ and $Y(\Theta_0)=Y$, there is a unique point $x'\in S^0_\Theta\subseteq \Omega$ with $Y_x=Y_{x'}$, so $x \in \pi^{x'}$. Let $S^{x'}_\Pi$ be the unique member of $\Omega$ meeting $S^0_\Theta$ in $x'$, and note that $S^{x'}_\Pi$ is a $k$-space contained in $\pi^{x'}$. As such, $S^{x'}_\Pi$ contains a unique point $x''$ inside $\<x,Y_x\>$, so $x\in \<x'',Y_{x''}\>$. Note that $\<x,Y_x\> \cap \Pi = \{x''\}$ since $x''$ is unique. As $x\in X$ was arbitrary, we get $\bigsqcup_{x\in X} \<x,Y_x\>\setminus Y_x =X$ indeed. Consequently,  $\<F^*,Y\>=\mathbb{P}^N(\K)$.

$(ii)$, $(iii)$  Same as in the proof of Proposition~\ref{1':FX}.

%$(iii)$ By $(ii)$, it follows from Proposition~\ref{1:projsegre} that the two families of singular $r$-spaces and $k$-spaces of $\Omega$ correspond bijectively to the sets $\Theta$ and $\Pi$, and since $\Omega \subseteq X$, this correspondence is given by containment.

$(iv)$ Let $r'=2$ and $k\in \{1,2\}$. If $k=2$, then we can use  the same argument as in the proof of Proposition~\ref{1':FX}; so suppose $k=1$. If $y$ is a point of $Y$ contained in $F^*$, then $y$ is on a unique line intersecting two planes of $\Omega$. This line contains three points of $X \cup Y$ and is hence singular, and by the above, this line is an $X$-line, a contradiction.
\end{proof}

Henceforth we will assume that the subspace $F$ is chosen such that it is contained in $F^*$, with $F^*$ as in the previous proposition.

\subsection{Conclusion}
Finally, we show that $X$ is, up to projection from a subspace in $Y$, a mutant of the half dual Segre variety $\mathcal{HDS}_{r',k}(\K)$ (see Subsection~\ref{HDS} and Definition~\ref{defmut}).

\begin{theorem}\label{1:projuni}\label{1:tangent}
Let $(X,Z,\Xi,\Theta)$ be a duo-symplectic pre-DSV with parameters $(r,v,r',v')$, containing an $X$-point through which there is precisely one member of $\Theta$. Then $r=1$ and: \begin{compactenum}[$(i)$]
\item Up to projection from a $v'$-space $V \subseteq Y$  collinear to all points of  $X$, we obtain that $X$ is the point set of a mutant of a half dual Segre variety $\mathcal{HDS}_{r',k}(\K)$ with $Z$ as subspace at infinity and $\Xi \cup \Theta$ as its symps; 
\item if additionally, $(X,Z,\Xi,\Theta)$ satisfies (S3), then $(r',k) \in \{(2,1),(2,2)\}$ and $X$ is projectively unique.
\end{compactenum}
\end{theorem}
\begin{proof}
$(i)$  Lemma~\ref{1:v'=-1} yields the $v'$-space $V$ collinear to all points of $X$ and Lemma~\ref{resY} allows us to project from $V$, so that we only need to deal with the case where $v'=-1$. 

By Proposition~\ref{1:FX}, $X$ contains a legal projection $\Omega$ of $\mathcal{S}_{r',k}(\K)$, and $X=\bigcup_{x\in \Omega} \<x,Y_x\>\setminus Y_x$. Therefore it suffices to show that the map $\chi:\Omega \rightarrow Y: x \mapsto Y_x$ satisfies the properties mentioned in the definition of the half dual Segre varieties (cf.\ Subsection~\ref{HDS}).

Let $S$ be a maximal singular subspace of $\Omega$. Proposition~\ref{1:FX} says that there is a unique member $\zeta_S$  in $\Theta \cup \Pi$ containing $S$. Choose $S$ such that $\theta:=\zeta_S\in \Theta$.  Inside the quadric $XY(\theta)$, it is clear that the restriction $\chi_S$ of $\chi$ to $S$ coincides with the collinearity relation between the opposite subspaces $S$ and $Y$, so $\chi_S$ is a linear duality between $S$ and $Y$. 

Take $x\in X$ arbitrary. If $x\notin S$, then there is a unique point $s^x \in S$ collinear to $x$. Since the line $\<x,s^x\>$ belongs to $\pi^x$, we have $Y_{s^x_\theta} =Y_x$, and hence $\chi(s^x)=\chi(x)$. We conclude that $\chi$ is indeed as described in Section~\ref{HDS}. For each pair of non-collinear points  $p_1,p_2$ of $X$, (S1) and Lemma~\ref{convexclosure} imply that the unique symp $\zeta$  through $p_1,p_2$ (defined as their convex closure inside $X$)  has $\zeta \cap X=X([p_1,p_2])$. Assertion $(i)$ follows.

$(ii)$ Recall that we assume that $F \subseteq F^*$ is complementary to $Y$ in $\mathbb{P}^N(\K)$.  By Proposition~\ref{1:projsegre}, $\rho(X)\subseteq F$ is an injective projection of a Segre variety $\mathcal{S}_{r',k}(\K)$. Let $x\in X$ be arbitrary. We denote by $T^{F}_{\rho(x)}$  the set of $\rho(X)$-lines in ${F}$ through $\rho(x)$ and  by $T^{F}_{\rho(x)}(\xi)$ the tangent space to $\rho(X(\xi))$ at $\rho(x)$ for some $\xi\in \Xi$ with $\rho(x)\in\rho(X(\xi))$. 

Axiom (S3) yields members $\xi_1,\xi_2\in \Xi$ through $x$ such that $T_x=\<T_x(\xi_1),T_x(\xi_2)\>$. Since for $i=1,2$, $T_x(\xi_i)=\<Y(\xi_i),T^{F}_{\rho(x)}(\xi_i)\>$, we obtain that $T_x=\<T_x(\xi_1),T_x(\xi_2)\>$ is equivalent with $Y_x=\<Y(\xi_1),Y(\xi_2)\>$ and  $T^{F}_{\rho(x)}=\<T^{F}_{\rho(x)}(\xi_1),T^{F}_{\rho(x)}(\xi_2)\>$. On the other hand, $\dim T^{F}_{\rho(x)}=r'+k$ as the tangent space at $\rho(x)$ is generated by the two maximal singular subspaces $\rho(\theta^x)$ and $\rho(\pi^x)$ of $\rho(X)$ through $\rho(x)$. Furthermore, since $r=1$,  $T^{F}_{\rho(x)}(\xi_1)$ and $T^{F}_{\rho(x)}(\xi_2)$ are just planes, which generate at most a $4$-space in $F$, and so $r+k'\leq 2+2=4$.  Recalling that  $r' > r \geq 1$ by assumption and $k \geq 1$, as noted in Remark~\ref{k>0}, we deduce that $(r',k)\in \{(3,1),(2,1),(2,2)\}$. However, if $k=1$, then the planes $T^{F}_{\rho(x)}(\xi_1)$ and $T^{F}_{\rho(x)}(\xi_2)$  share a line, and hence generate at most a $3$-space, so $r'+k=r'+1\leq 3$, excluding the possibility $(r',k)=(3,1)$, so $r'=2$. Since $v=r'-2=0$ and $\dim Y_x=r'-1=1$,  the requirement $Y_x=\<Y(\xi_1),Y(\xi_2)\>$ only implies that $\xi_1$ and $\xi_2$ have disjoint vertices. 

Since $r'=2$ and $k \in \{1,2\}$,  the variety $\mathcal{S}_{r',k}(\K)$ does not admit legal projections (cf.\ Proposition~\ref{noproj}) and $F^* \cap Y = \emptyset$ by Proposition~\ref{1:FX}$(iv)$, i.e., $F=F^*$.
The following are projectively unique: $Y$ and $F$  in $\mathbb{P}^N(\K)$,   $\Omega$ in $F$. Moreover, the projectivity $\chi_{S}$ between $S$ and the dual of $Y$ is unique up to a projectivity of $Y$. We conclude that $X$ is projectively unique. 
\end{proof}

\section{The dual line Grassmannians}\label{r>1}

Let $(X,Z,\Xi,\Theta)$ be a duo-symplectic pre-DSV with parameters $(r,v,r',v')$ with $r \geq 2$. By Lemma~\ref{resY}, we may assume that no point of $Y$ is collinear to all points of $X$.

Recall that, by Lemma~\ref{res}, for each $x\in X$, the point residue $\Res_X(x)=(X_x,Z_x,\Xi_x,\Theta_x)$ (cf.\ Definition~\ref{ptres}) is a pre-DSV with parameters $(r-1,v,r'-1,v')$, with $H_x=\<X_x,Z_x\>$ a hyperplane of $T_x$ (in particular, $\dim H_x \leq 2d-1$). Moreover, in view of the first paragraph, Lemma~\ref{projV} says no point of $Y_x$ is collinear to all points of $X_x$. Recall as well that the members $\zeta$ of $\Xi$ and $\Theta$ through $x$ are in $1-1$-correspondence with the members $\zeta_x$ of $\Xi_x$ and $\Theta_x$, respectively; and observe that (as is reflected by the parameters), the vertices of $\zeta$ and $\zeta_x$ coincide. 

In  Sections~\ref{r=1,2SS} and~\ref{1SS/point} we dealt with duo-symplectic pre-DSVs such that, respectively:
\begin{compactenum}
\item[(A)] $r=1$ and $|\Theta_x|>1$ for each point of $X$, in which case we showed that $(X,Z,\Xi,\Theta)$ is isomorphic to a mutant of the dual Segre variety $\mathcal{DS}_{r',r'}(\K)$ (cf.\ Theorem~\ref{1':projuni}$(i)$) and in particular $|\Theta_x|=2$ for all $x \in X$;
\item[(B)] for some $x\in X$, $|\Theta_x|=1$, in which case we showed that $(X,Z,\Xi,\Theta)$ is isomorphic to the half dual Segre variety $\mathcal{HDS}_{r',k}(\K)$ (cf.\ Theorem~\ref{1:projuni}$(i)$) and in particular $r=1$ and $|\Theta_x|=1$ for all $x\in X$. 
\end{compactenum}
\subsection{Case distinction and reduction}\label{dist}
In view of the above, the current case distinction will depend on whether there is an $X$-line contained in a unique member of $\Theta$ (recall that we refer to such a line as a $1$-line)  or not. We first show that there are no $0$-lines.

 \begin{lemma}\label{2:uniSS}
No $X$-line of $(X,Z,\Xi,\Theta)$ is a $0$-line. \end{lemma}
\begin{proof}
Suppose that $L$ is an $X$-line in $(X,Z,\Xi,\Theta)$, and let $x\in L$. Then, in $\Res_X(x)$, the point corresponding to the line $L$ is contained in at least one member of $\Theta_x$ by Lemma~\ref{xinSS} (the fact that $|\Theta_x|\geq 1$ follows from the same lemma, since $|\Theta|\geq 1$ by assumption). Hence, $L$ is not a $0$-line. \end{proof}

\textbf{Case distinction.} There are two possibilities:

\begin{enumerate}
\item \textit{There is an $X$-line $L$ contained in a unique member of $\Theta$.} This means that, for $x\in L$, there is a point of $\Res_X(x)$ contained in a unique member of $\Theta_x$. As mentioned in (B) above, this implies that $\Res_X(x)$  is a mutant of the half dual Segre variety  $\mathcal{HDS}_{r'-1,k_x}(\K)$ for some $k_x \geq 1$, and in particular, $r=2$. Note that $k_x < \infty$ since $k_x \leq \dim(T_x)\leq 2d$. 
\item \textit{Each $X$-line is contained in at least two members of $\Theta$.} Here, there are two subcases:
\begin{enumerate}
\item If $r=2$, then as mentioned in (A) above,  $\Res_X(x)$ is a mutant of the dual Segre variety  $\mathcal{DS}_{r'-1,r'-1}(\K)$, for every $x\in X$. 
\item If $r > 2$, we consider the residue $\Res_X(L)$ of an $X$-line $L$, which can be obtained by taking two subsequent point residues of points on $L$. According to Corollary~\ref{noXplane}, each point of $X_L$ (corresponding to an $X$-plane of $X$) is contained in a unique member of $\Theta_L$ (with self-explanatory notation). So, as in (B), $\Res_X(L)$ is a mutant of a half dual Segre variety $\mathcal{HDS}_{r'-2,k_L}(\K)$ for some $k_L\geq 1$ and in particular, $r=3$. Note that, for $x\in L$, the residue $\Res_X(x)$ belongs to Case 1.
\end{enumerate}
\end{enumerate}

When (S3) also holds, only Case 1 will lead to an existing case. We first show that we are always in Case 1 if (S3) holds.

\begin{lemma}\label{reshalf}
Let $(X,Z,\Xi,\Theta)$ be a duo-symplectic DSV with parameters $(2,v,r'v')$. Then there is a $1$-line. \end{lemma}
\begin{proof}
Suppose for a contradiction that  there are no $1$-lines, so by Lemma~\ref{2:uniSS}, each $X$-line is contained in at least two members of $\Theta$. According to the above case distinction and recalling that $r=2$, this means that, for each $x\in X$, $\Res_X(x)$ is isomorphic to a mutant of the dual Segre variety $\mathcal{DS}_{r'-1,r'-1}(\K)$. By Lemma~\ref{1':YZ}, $Y_x$ is generated by $(r'-1)$-spaces $R^x_1$ and $R^x_2$, in particular $\dim(Y_x)=2r'-1$. 
Now (S3) yields members $\xi_1,\xi_2$ of $\Xi$ through $x$ such that $T_x$ is generated by $T_x(\xi_1)$ and $T_x(\xi_2)$. Let $V_1$ and $V_2$ be the respective vertices of $\xi_1$ and $\xi_2$. Then $\<V_1,V_2\>\subseteq Y_x$, so by the above
$\dim(\<V_1,V_2\>)\leq  2r'-1$. Since $r=2$, it follows that $\dim\<T_x(\xi_1),T_x(\xi_2)\>\leq (2r'-1)+8+1=2r'+8$. 

As such, $\Res_X(x)$ is contained in a projective space of dimension $2r'+7$, implying that a legal projection $\Omega$ of $\mathcal{S}_{r'-1,r'-1}(\K)$ (cf.\ Lemma~\ref{1':FX}) is contained in a projective space $F^*_x$ in $\Res_X(x)$. Since $r' > r = 2$,  $\Omega$ contains an isomorphic copy $\Omega'$ of $\mathcal{S}_{2,2}(\K)$ (cf.\ Lemma~\ref{noproj}). Since $\dim\<\Omega'\>=8$, there is a point $y\in  \<\Omega'\> \cap Y$. By Proposition~\ref{1':projsegre} and the fact that $\Omega \subseteq X$, each grid $G$ of $\Omega$ corresponds to a unique $(2r'-5)$-space  of $Y$ (namely, the vertex of the unique member of $\Xi_x$ containing $G$), generated by two $(r'-3)$-spaces $V^G_1$ and $V^G_2$ in $R^x_1$ and $R^x_2$, respectively. Restricting to the grids contained in $\Omega'$, we obtain that $\{V^G_i \mid G \text{ grid of } \Omega'\}$ is the set of all $(r'-3)$-spaces in $R^x_i$ containing a certain $(r'-4)$-space $H^G_i$, for $i=1,2$. Thus, the point $y$, being contained in $\<R^x_1,R^x_2\>$, is contained in the vertex of a member of $\Xi_x$ which has a grid $G$ in $\Omega'$. 
The $4$-space $\<G,y\>$ is hence contained in the $8$-dimensional subspace generated by the Segre variety $\Omega'$ and intersects $\Omega'$ in precisely $G$. Just like in the proof of Lemma~\ref{1':FX},  Lemma~\ref{s22p} then leads to a contradiction to (S2). This concludes the proof.
\end{proof}
\par\bigskip

So if (S3) holds, only Cases 1 and 2(b) remain. As noted above, Case 1 is in fact a subcase of Case 2(b). So we start with a study of Case~1 (without assuming that  (S3) holds).%In both cases, it is of interest to study  the duo-symplectic pre-DSVs with parameters $(2,v,R',v')$, for some $R' >2$, where, for each $x\in X$, the residue $\Res_X(x)$ is a mutant of a half dual Segre variety $\mathcal{DS}_{R'-1,k_x}(\K)$ for some $k_x \geq 1$. Indeed, in Case 1, $(X,Z,\Xi,\Theta)$ is such a Veronese set for $R'=r'$ (note that we additionally know that (S3) holds); in Case 2(b), $\Res_X(x)$ is such a Veronese set for $R'=r'-1$ since for each $X$-line $L$ through $x$, $\Res_X(L) \cong \Res_{X_x}(x')$ is a mutant of a half dual Segre variety $\mathcal{DS}_{r'-2,k_x}(\K)$.   The next section is devoted to this study. For ease of notation, we use $r'$ instead of $R'$, as usual.

\subsection{Case 1: Duo-symplectic pre-DSVs  containing a $1$-line}\label{cas1}

Throughout this section, let $(X,Z,\Xi,\Theta)$ be a duo-symplectic pre-DSV with parameters $(r,v,r',v')$, such that no point of $Y$ is collinear to all points of $X$, containing a $1$-line.%for each $x\in X$, $\Res_X(x)$ is a mutant of a half dual Segre variety $\mathcal{HDS}_{r'-1,k_x}(\K)$ for some $k_x \geq 1$. 
As deduced in the case distinction of the previous section, this immediately implies that $r=2$.

\begin{lemma}\label{2:Xlines}
Each $X$-line is a $1$-line. \end{lemma}
\begin{proof}
Suppose that $L$ is a $1$-line and let $x$ be a point of $L$. Then $|\Theta_x|=1$. By Lemma~\ref{1:uniqueSS}, this holds for each point of $\Res_X(x)$, i.e., each $X$-line through $x$ is a $1$-line. Since $x\in L$ was arbitrary and since $X$ is connected via $X$-lines by (S1), all $X$-lines are $1$-lines. \end{proof}

As noted in Case 1 above, we have that $\Res_X(x)$ is isomorphic to a mutant of the half dual Segre variety $\mathcal{HDS}_{r'-1,k_x}(\K)$, for some $k_x\geq 1$, $x\in X$. 

We start by relating $\dim Y$ to $k_x$ (and show that the latter does not depend on $x\in X$). 

\begin{lemma}\label{2:Y}
The sets $Y$ and $Z$ coincide, and $\dim Y=r'+k_x$. %and each subspace of $Y$ containing $Y_x$ as a hyperplane occurs as $Y(\theta)$ for some $\theta \in \Theta$ containing $x$. 
\end{lemma}
\begin{proof}
According to Proposition~\ref{1:FX}, $X_x$ (recall that $\Res_X(x)$ is isomorphic to a mutant of $\mathcal{HDS}_{r'-1,k_x}(\K)$) contains a legal projection $\Omega$ of the Segre variety $\mathcal{S}_{r'-1,k}(\K)$, and $F^*_x:=\<\Omega\>$ and $Y_x$ generate $H_x$. Assertion $(ii)$ of the same proposition, together with Proposition~\ref{1:projsegre}, implies that containment gives a bijection between the 
$(r'-1)$-spaces of $\Omega$ and  the members of $\Theta_x$ (and containment also gives a bijection between them and the members of $\Theta$ through $x$). Furthermore, by Lemmas~\ref{1:v'=-1},~\ref{projV} and the assumption on $Y$, we have $v'=-1$; and each member of $\Theta_x$ contains $Y_x$ by Lemma~\ref{1:uniqueSS}. Finally,  $Y_x=Z_x$ and $\dim Y_x=r'-1$ by Corollary~\ref{1:Y=Z}. 

Let $I$ be an index set such that $\{\theta_i \in \Theta \mid i \in I\}$ ranges over all members of $\Theta$ containing $x$. Let $i\in I$ be arbitrary.  By the above, there is a unique $(r'-1)$-space $S_i$ in $\Omega$ contained in $\theta_i$. Recalling that $v'=-1$, we get that $\dim  Y(\theta_i)=r'$ and hence $Y(\theta_i)\subseteq Z$ contains a unique point  $z_i$ collinear to $S_i$. Obviously, $z_i \notin Y_x$ since $S_i$ and $Y_x \cap \theta_i$ are opposite singular $(r'-1)$-spaces in $\theta_i$.  So  $\theta_i = [x,z_i]$ and $Y(\theta_i)=\<Y_x,z_i\>$. Since, for $i,i' \in I$ with $i \neq i'$, we have $\theta_i \cap \theta_{i'} = Y_x$ as the subspaces $S_i$ and $S_{i'}$ are disjoint, $z_{i'}\in \theta_i$ if and only if $i=i'$.
We claim that  set $M:=\{z_i \mid i \in I\}$ contains the points of a $k_x$-dimensional subspace $K_x$ of $Y$ complementary to $Y_x$. 
To that end, take two arbitrary members $\theta_1,\theta_2 \in \{\theta_i \mid i\in I\}$  and let $J \subseteq I$ be such that $\{S_j \mid j \in J\}$ is the unique regulus of $\Omega$ determined by $S_1$ and $S_2$. We show that the points $\{z_j \mid j\in J\}$ are the points of a line of $Y$.

Take any point $x_1 \in S_1$. Then there is a unique line $L$ through $x_1$ meeting each $S_j$ for $j \in J$ in a point, say $x_j$. Note that, for each $j \in J$, $\theta_j$ is the unique member of $\Theta$ containing the $X$-line $xx_j$; so Lemma~\ref{lineSS}$(ii)$ then implies that $Y_{x_j} \subseteq \theta_j$. %Since $z_2 \notin \theta_1$, this means that $z_2 \notin Y_{x_1}$. Consequently, $[x_1,z_2] \in \Theta$. As $x_2\in X$ is collinear to both $x_1$ and $z_2$,  Lemma~\ref{convexclosure} gives us $x_2 \in [x_1,z_2]$, moreover $z_2 \in Y_{x_2} \subseteq [x_1,z_2]$.  Likewise,  $x_1,z_2 \in [x_2,z_1] \in \Theta$.  By Lemma~\ref{2:uniSS}, there is only one member of $\Theta$ containing $x_1x_2$, and therefore $[x_1,z_2]=[x_2,z_1]$. 
By Lemma~\ref{2:Xlines}, there is a unique member $\theta_L \in \Theta$ containing $L$, and by the foregoing, $L$ contains $Y_{x_j}$ (and hence $z_j$) for all $j\in J$. Then, inside $\Theta_L$, we see that $\<z_1,z_2\>$ and $\<x_1,x_2\>$ are opposite singular lines, so $\<z_1,z_2\>$ contains a unique point $z'_j$ collinear to $x_j$ for each $j \in J$ (clearly, $z'_1=z_1$ and $z'_2=z_2$). Since $z'_j \in Y_{x_j} \subseteq \theta_j$, we obtain $\<z_1,z_2\> \cap \theta_j = \{z'_j\}$ for each $j \in J$. Varying the point $x_1 \in S_1$, we obtain that $z'_j$ is collinear to all points of $S_j$, so $z'_j=z_j$. This shows the claim: $M$ carries the same structure as a $k_x$-space of $\Omega$ intersecting all subspaces $\{S_i \mid i \in I\}$.

Clearly, $M$ is disjoint from $Y_x$ and $\<Y_x,M\>=\bigcup_{i\in I} \<Y_x,z_i\>=\bigcup_{i\in I} Y(\theta_i) \subseteq Z$. Suppose for a contradiction that $\<Y_x,M\>$ is a strict subspace of $Y$. Since $\<Z\>=Y$, there is a point $z\in Z$ in  $Y\setminus\<Y_x,M\>$. But then $[z,x]=\theta_i$ for some $i\in I$, and hence $z \in Y(\theta_i)=\<Y_x,z_i\>\subseteq\<Y_x,M\>$, a contradiction. We conclude that $Y=\<Y_x,M\>\subseteq Z$, so $\dim Y=r'+k_x$ indeed.
\end{proof}

Henceforth, we write $k$ instead of $k_x$ since the latter does not depend on $x\in X$. 

Next, we have an analogue of Lemma~\ref{1':rel}, which will in particular allow us to prove that $k=1$.

\begin{lemma}\label{2:rel}%independent of k
Take two distinct points $x_1,x_2 \in X$. Then
\begin{compactenum}[$(i)$]
\item $x_1x_2$ is a singular line with a unique point in $Y$ $\Leftrightarrow Y_{x_1}= Y_{x_2}$;
\item $x_1$ and $x_2$ belong to a member of $\Theta$ $\Leftrightarrow  \dim(Y_{x_1} \cap Y_{x_2})=r'-2$;
\item $x_1$ and $x_2$ are non-collinear points of a member of $\Xi$  $\Leftrightarrow \dim(Y_{x_1} \cap Y_{x_2})=r'-3$.
\end{compactenum}
\end{lemma}

\begin{proof}
If $x_1x_2$ is an $X$-line, then it is a $1$-line by Lemma~\ref{2:Xlines}. As such, the possibilities for $x_1,x_2$ described in $(i)$, $(ii)$ and $(iii)$ exhaust the mutual positions between $x_1$ and $x_2$, and therefore we only need to verify the ``$\Rightarrow$''s. 

$(i), \Rightarrow$: This is clear.

$(ii), \Rightarrow$: Suppose $x_1,x_2\in\theta$ for some $\theta\in \Theta$. Then $Y_{x_1} \cup Y_{x_2} \subseteq Y(\theta)$ and hence one deduces, similarly as in the proof of Lemma~\ref{1':rel}, that $\dim(Y_{x_1} \cap Y_{x_2})=\dim Y(\theta)-2=r'-2$. 

$(iii), \Rightarrow$: If $[x_1,x_2]=\xi\in \Xi$, then $Y(\xi)=Y_{x_1} \cap Y_{x_2}$ and hence the dimension of the latter is $v=r'-3$. 
%Firstly, if  $x_1x_2$ is a singular line with a unique point in $Y$, then $x_2 \in \<x_1,Y_{x_1}\>$ and as such it is clear that $Y_{x_1}= Y_{x_2}$.Secondly, suppose that $x_1x_2$ is an $X$-line. Since $x_1x_2$ is contained in a unique member of $\Theta$ (cf. Lemma~\ref{2:uniSS}), Lemma~\ref{lineSS}$(i)$ implies that $Y_{x_1x_2}$ is a hyperplane of $Y_{x_1}$, from which we obtain that $\dim(Y_{x_1} \cap Y_{x_2})=r'-2$.Thirdly, suppose that $x_1$ and $x_2$ are not collinear. If $[x_1,x_2] \in \Xi$, then $ \dim(Y_{x_1} \cap Y_{x_2})=v=r'-3$. If $[x_1,x_2] \in \Theta$, then since $\dim(Y([x_1,x_2]))=r'$, we obtain that $\dim(Y_{x_1} \cap Y_{x_2})=r'-2$ as  $x_1$ and $x_2$ are non-collinear $X$-points. Since each $X$-line is contained in a member of $\Theta$, the lemma follows. 
\end{proof}

We record an obvious but important consequence.

\begin{cor}\label{2:rho}
For all  $x_1,x_2 \in X$, we have $Y_{x_1}=Y_{x_2} \Leftrightarrow \rho(x_1)=\rho(x_2)$.
\end{cor}
\begin{proof}
By the previous lemma, $Y_{x_1}=Y_{x_2}$ is equivalent with $x_1x_2$ being a singular line with a unique point in $Y$, i.e., with $x_2 \in \<x_1, Y_{x_1}\>$. By Lemma~\ref{1:inj}, this is at its turn equivalent with $\rho(x_1)=\rho(x_2)$. 
\end{proof}

Lemma~\ref{2:rel} becomes a powerful tool if we can show that each $(r'-1)$-space of $Y$ occurs as $Y_x$ for some $x\in X$:

\begin{lemma}\label{2:r'sp}
For each $(r'-1)$-space $H$ in $Y$, there is a point $x\in X$ such that $Y_x=H$, and all $X$-points collinear to $H$ are precisely the points in $\<x,H\>\setminus H$.
\end{lemma}
\begin{proof}
Take $x\in X$ arbitrary. Recall that  $\dim Y= r'+k$ by Lemma~\ref{2:Y}, so in particular $\dim Y < \infty$. This implies that it suffices to show that, for each $(r'-1)$-space $H'$ of $Y$ with $\dim(Y_x \cap H')=r'-2$, we have $H'=Y_{x'}$ for some $x'\in X$, as then connectivity argument finishes the argument. So without loss of generality, $\dim(Y_x \cap H)=r'-2$.  Let $z \in H \setminus Y_x$ be arbitrary. Recall that $z \in Z$ by Lemma~\ref{2:Y}. So by (S1), $[z,x]$ is a member of $\Theta$, which contains $Y_x$ and $x$, and hence $H$. As such, $[x,z]$ contains an $X$-point $x'$ collinear to $H$, i.e., $Y_{x'}=H$. 
The second part of the assertion follows from Lemma~\ref{2:rel}.
\end{proof}

As promised, the above leads us to $k=1$:

\begin{cor}\label{2:kis1}
We have $k=1$.
\end{cor}

\begin{proof}
  Lemma~\ref{2:rel} implies that $\dim(Y_{x_1} \cap Y_{x_2}) \geq r'-3$ for all $x_1,x_2 \in X$ and that there is a pair for which equality is reached (since $|\Xi|\geq 1$). In view of Lemma~\ref{2:r'sp}, this means that all pairs of  $(r'-1)$-spaces of $Y$ should share at least an $r'-3$-space and there is a pair whose intersection is precisely an $(r'-3)$-space. Thus, $\dim Y = r'+1$. Since Lemma~\ref{2:Y} says that  $\dim Y= r'+k$, we conclude that $k=1$.
\end{proof}

Another corollary is the following.

\begin{cor}\label{2:r'th}
For each $r'$-space $Y'$ in $Y$, there is a unique $\theta \in \Theta$ with $Y(\theta)=Y'$. Moreover, if $x\in X$ has $Y_x \subseteq Y'$, then $x \in \theta$. 
\end{cor}
\begin{proof}
Take any $(r'-1)$-space $H$ in $Y'$. By Lemma~\ref{2:r'sp}, we know that $H=Y_x$ for some $x\in X$. Take any point $z\in Y'\setminus H$. Then $\theta:=[x,z]$ is a member of $\Theta$ with $Y(\theta)=Y'$. Let $x'\in X$ be such that $Y_{x'} \subseteq Y'$. Then $X(\theta)$ contains a point $x''$ with $Y_{x''}=Y_{x'}$, so by Corollary~\ref{2:rho}, $x' \in \<x'',Y_{x''}\>\subseteq \theta$. This also shows that $\theta$ is the unique member of $\Theta$ containing $Y'$.
\end{proof}

We proceed similarly as in Section~\ref{r=1,2SS} and nail down the structure of $(X,Z,\Xi,\theta)$, using the projection $\rho$ and the connection map $\chi$ (cf.\ Definitions~\ref{RHO} and~\ref{CHI}).

\begin{lemma}~\label{2:rhoXlines}
Let $\rho(x_1)$ and $\rho(x_2)$ be distinct points on a line of $\rho(X)$, for $x_1,x_2 \in X$. Let $x'_i \in \rho^{-1}(\rho(x_i))$ be arbitrary, for $i=1,2$. Then:
\begin{compactenum}[$(i)$]
\item there is a unique $\theta \in \Theta$ containing  $x'_1 \cup x'_2$, and $\rho^{-1}(\rho(x_1)) \cup  \rho^{-1}(\rho(x_2))\subseteq \theta$;
\item there  is an $x''_2\in\rho^{-1}(\rho(x_2))$ such that $\<x'_1,x''_2\>$ is an $X$-line. 
\item  $\{Y_x \mid \rho(x) \in \<\rho(x_1),\rho(x_2)\>\}$ is the set of all $(r'-1)$-spaces through the $(r'-2)$-space $Y_{x_1}\cap Y_{x_2}$ inside the $r'$-space $Y(\theta)$. 
\end{compactenum}
\end{lemma}

\begin{proof} This can be proven similarly as Lemma~\ref{1':rhoXlines};  the only difference being that, in $(i)$, we now have to show that $\dim(Y_{x_1}\cap Y_{x_2})=r'-2$, but assuming the contrary leads as well to the situation where $\<Y_{x_1},Y_{x_2}\>=Y$. \end{proof}

\begin{lemma}\label{2:welldef}
Suppose $\xi=[x_1,x_2]\in \Xi$ for points $x_1,x_2\in X$ and put $T=Y(\xi)$. Then:
\begin{compactenum}[$(i)$]
\item $\sigma_\xi: \rho(X(\xi)) \rightarrow \{H \mid T \subseteq H \subseteq Y, \dim H = r'-1\}: \rho(x) \mapsto Y_x$ is an isomorphism;
\item  If $x'_i \in \rho^{-1}(\rho(x_i))$ for $i=1,2$, then $\xi':=[x'_1,x'_2]$ belongs to $\Xi_T$; 
\item for each $\xi'\in \Xi_T$, $\rho(X(\xi'))=\rho(X(\xi))$;
\end{compactenum}
\end{lemma}

\begin{proof} This can be proven completely similarly as Lemma~\ref{1':welldef}; the only difference being that in $(i)$, the line set $\{L(x) \mid \rho(x)\in \rho(X(\xi))\}$ is precisely the set of lines of the $3$-space $\Pi_T(\K)$, which is, equipped with full planar point pencils, isomorphic to the Klein quadric $Q$ (using the Klein correspondence). 
%$(i)$  First of all, note that $\sigma_\xi$ is well-defined by Lemma~\ref{proprhochi}. We consider the residue $\Res_Y(T)$. Since $\dim(T)=r'-3$ and $\dim(Y)=r'+1$,  the residue $\Res_Y(T)$ is isomorphic to  a projective $3$-space over $\K$, say $\Pi_T(\K)$, in which $Y_x$ corresponds to a line $L(x)$.  Let $x,x'$ be two points of $X(\xi)$. By Lemma~\ref{2:rel}, $L(x)=L(x')$ if and only if $x$ and $x'$ belong to the same generator of $X(\xi)$, i.e., if and only if $\rho(x)=\rho(x')$; $L(x)$ and $L(x')$ intersect in precisely a point if and only if $xx'$ is an $X$-line in $X(\xi)$ and $L(x)$ and $L(x')$ are disjoint if and only if $x$ and $x'$ are non-collinear. Moreover, Lemma~\ref{2:rhoXlines} implies that each $X$-line of $X(\xi)$ corresponds to a full planar point pencil in $\Pi_T$. 
%On the other hand, the Klein correspondence yields that the point-line geometry whose point set is the set of all lines of  $\Pi_T(\K)$ and whose lines are the planar line pencils of $\Pi_T(\K)$, is isomorphic to the point-line geometry associated to a Klein quadric $Q$ in $\mathbb{P}^{5}(\K)$. Moreover, also  $\rho(X(\xi))$ is a Klein quadric in $\mathbb{P}^{5}(\K)$ as $r=2$.  By the previous paragraph, $\sigma(\rho(X(\xi)))$  is embedded isometrically into $Q$ Since the fields of definition are the same, $\sigma(\rho(X(\xi)))=Q$, i.e., $\sigma$ is an isomorphism. Assertion $(i)$ follows.
%$(ii), (iii)$: Exactly the same as $(ii)$ and $(iii)$ of the proof of Lemma~\ref{1':welldef}.
\end{proof}

We  use $Y$ to define the following point-line geometry $(\mathsf{P},\mathsf{B})_Y$.

\begin{defi}\label{2:PB}
Let  $\mathsf{P}$ denote the set of $(r'-1)$-dimensional subspaces of $Y$. For subspaces $S_-\subseteq S_+ \subseteq Y$ with $\dim S_\pm=r'-2\pm 1$, we define the pencil $P(S_-,S_+)$ as the set $\{P \in \mathsf{P} \mid S_- \subseteq P \subseteq S_+\}$. Then we denote by $\mathsf{B}$ the set $\{P(S_1,S_2) \mid S_1 \subseteq S_2 \subseteq Y,  \dim S_\pm=r'-2\pm 1\}$. 
\end{defi}

\begin{lemma}\label{2:Ystructuur}
The point-line geometry $(\mathsf{P},\mathsf{B})_Y$ is isomorphic to the (point-line truncation of the) line Grassmannian $\mathcal{G}_{r'+2,2}(\K)$. 
\end{lemma}

\begin{proof}
Since a projective space is self-dual and $\dim Y=r'+1$, the point-line geometry $(\mathsf{P},\mathsf{B})_Y$ (with natural incidence relation) is by definition isomorphic to the (point-line truncation of the) line Grassmannian $\mathcal{G}_{r'+2,2}(\K)$. 
\end{proof}

\begin{prop}\label{2:projgras}
The point-line geometry $\mathcal{G}:=(\rho(X),\rho(\mathsf{L}))$ is isomorphic to an injective projection of the line Grassmannian $\mathcal{G}_{r'+2,2}(\K)$. Moreover, we have:
\begin{compactenum}
\item[$(i)$] for each singular $r'$-space $S$ in $\mathcal{G}$, there is a unique $\theta_S \in \Theta$ with
 $\rho^{-1}(S)=X(\theta_S)$.
\item[$(ii)$] for each symp $Q$ of $\mathcal{G}$ (viewing the latter as a parapolar space), there is a unique  $v$-space $V$ in $Y$ with $\rho^{-1}(Q)=X_V$ and vice versa.
\end{compactenum}
\end{prop}
\begin{proof}
 We  claim that $\chi$ induces an isomorphism between the abstract point-line geometries $(\rho(X),\rho(\mathsf{L}))$ and $(\mathsf{P},\mathsf{B})_Y$. Indeed, the fact that $\chi: \rho(X) \rightarrow \mathsf{P}: x \mapsto \chi(x)=Y_x$ is a bijection between $\rho(X)$ and $\mathsf{P}$ follows immediately from Corollary~\ref{2:rho} (injectivity) and Lemma~\ref{2:r'sp} (surjectivity). The fact that a member of $\rho(\mathsf{L})$ is mapped by $\chi$ to a member of $\mathsf{B}$ follows from Lemma~\ref{2:rhoXlines}$(iii)$. This shows the claim. By Fact~\ref{factuniemb} and Lemma~\ref{2:Ystructuur},  $(\rho(X),\rho(\mathsf{L}))\subseteq F$ arises as an injective projection of the line Grassmannian $\mathcal{G}_{r'+2,2}(\K)$.

$(i)$ Let $S$ be a maximal singular subspace of $\rho(X)$ of dimension $r'$ and take a line $L$ in $S$. By Lemma~\ref{2:rhoXlines}$(i)$, there is a unique $\theta\in \Theta$ containing $L$. The properties of the line Grassmannian $\mathcal{G}_{r'+2,2}(\K)$ imply that there is a unique $r'$-space through $L$; as such, the $r'$-space $\rho(X(\theta))$ coincides with $S$.

$(ii)$ Lastly, let $Q$ be any symp of $\mathcal{G}$ (so $Q$ is a Klein quadric since $r=2$).  By Lemma~\ref{2:welldef}$(iii)$, it suffices to show that $Q$  coincides with $\rho(X(\xi))$ for some $\xi \in \Xi$.  Let $p_1$ and $p_2$ be non-collinear points of $Q$ and take points $x_1,x_2 \in X$ with $\rho(x_i)=p_i$. Then, since $p_1$ and $p_2$ are distinct and non-collinear,  $\xi:=[x_1,x_2]\in \Xi$. Now, $\rho(X(\xi))$ is a Klein quadric in $\rho(X)$ containing the points $p_1$ and $p_2$, and since two non-collinear points determine a unique symp in $\rho(X)$, we obtain $\rho(X(\xi))=Q$.  \end{proof}

\begin{lemma}\label{2:FX}
The set $X$ contains a legal projection $\Omega$ of $\mathcal{G}_{r'+2,2}(\K)$ which is such that: 
\begin{compactenum}[$(i)$]
\item $\bigsqcup_{x\in \Omega} \<{x},Y_{{x}}\>\setminus Y_x=X$ and hence, putting $F^*=\<\Omega\>$, $\<F^*,Y\>=\mathbb{P}^N(\K)$;
\item Re-choosing the subspace $F$ so that it is inside $F^*$, the projection $\rho^*$ of $F^* \cap X$ from $F^*\cap Y$ onto $F$ is the restriction of $\rho$ to $F^* \cap X$;
\item  Containment gives a bijection between the singular $r'$-spaces of $\Omega$ and the set $\Theta$;
\item If $r' =4$, then $F^* \cap Y = \emptyset$.
\end{compactenum}

 \end{lemma}
\begin{proof}
By Proposition~\ref{2:projgras},  $\rho(X)$ is the point set of an injective projection of the line Grassmannian $\mathcal{G}_{r'+2,2}(\K)$, and its set of singular $r'$-spaces is in $1-1$-correspondence to the members of $\Theta$. We  construct a legal projection of $\mathcal{G}_{r'+2,2}(\K)$ \emph{inside} $X$ (where we know that (S2) holds), by choosing a set of points $X'\subseteq X$ such that $\{Y_x \mid x \in X'\}$ is a basis of $(\mathsf{P},\mathsf{B})_Y$ (cf.\ Definition~\ref{2:PB}). 

To that end, let $\mathcal{B}:=\{p_0,...,p_{r'+1}\}$ be the set of points of a basis of $Y$. Put $A=\{(i,j) \mid 0 \leq i, j \leq r'+1 \text{ and } i\neq j\}$. For each pair $(i,j) \in A$, let $H_{ij}$ be the $(r'-1)$-space  generated by the points of $\mathcal{B}\setminus\{p_i,p_j\}$. By Lemma~\ref{2:r'sp} and Corollary~\ref{2:rho}, each $H_{i,j} \in \mathcal{H}$ corresponds to a unique $r'$-space $\overline{H}_{i,j}:=\<x'_{i,j},H_{i'j}\>$ whose $X$-points are precisely the set of $X$-points collinear to $H_{i,j}$. %Note that, for two pairs $(i,j),(i',j')\in A$, if $|\{i,j,i',j'\}|=3$ then  $\dim(H_{i,j} \cap H_{i',j'})=r'-2$ and hence, by Lemma~\ref{2:rel}, the points $x'_{i,j}$ and $x'_{i',j'}$ determine a member of $\Theta$, which then also contains the subspaces $\overline{H}_{i,j}$ and $\overline{H}_{i',j'}$. 
For each $i\in\{0,...,r'+1\}$, Lemma~\ref{2:r'th} yields a unique $\theta_i \in \Theta$ with  $Y(\theta_i)=\<\mathcal{B}\setminus\{p_i\}\>$, and by the same lemma, $\theta_i$ contains $\overline{H}_{i,j}$ for all $j\in \{0,...,r'+1\}$.
Just like in the proof of Lemma~\ref{1':FX} (more precisely, Claim 1), we can consecutively choose points $X$-points $x_{i,j} \in \overline{H}_{i,j}$, using the lexicographic order on the pairs in $A$, in such a way that  $x_{i,j}$ is collinear to $x_{i,j'}$ for each $j' \in \{0,...,j-1\}$ and to $x_{i',j}$ with $i' \in \{0,...,i-1\}$. 

For each $i\in\{0,...,r'+1\}$, we define $R'_i:=\<x_{i,j}\mid 0 \leq j \leq r' \>$. By construction, $R'_i$ is a singular subspace of $\theta_i$. Arguments analogous to those in the proof of Lemma~\ref{1':FX}, show that $R'_i$ is an $X$-space of dimension $r'$.
Put $F^*=\<R'_0,...,R'_{r'}\>$. 

\textit{Claim: Each $\theta\in \Theta$ contains an $r'$-space $R'_\theta$ in $F^* \cap X$.}\\ Recall that, for each $r'$-space $Y'$ in $Y$, there is a unique member $\theta_{Y'}$ of $\Theta$ with $Y(\theta_{Y'})=Y'$ (cf.\ Corollary~\ref{2:r'th}). Firstly, let $H$ be any $r'$-space through $H_{0,1}$. Then $H$ meets the line $\<p_0,p_1\>$ in a point $p$. If $p=p_i$, for $i\in\{0,1\}$, then $\theta_H=\theta_i$ and $R'_i \subseteq F^*\cap X$ by definition of $F^*$; so suppose $p \notin \{p_0,p_1\}$.  Now, for each $j \in \{2,...,r'+1\}$,  Lemma~\ref{2:rhoXlines} implies that the $X$-line $\<x_{0,j},x_{1,j}\>$ contains a unique point $q_j$ which is collinear to the $(r'-1)$-space generated by $p$ and $\<\mathcal{B}\setminus\{p_0,p_1,p_j\}\>$; and $q_j \in \theta_H$ by Corollary~\ref{2:r'th}. We show that $\<x_{0,1},q_2,...,q_{r'+1}\>$ is a singular $r'$-space in $X$ (clearly, it is contained in $F^* \cap \theta_H$). 

Take $j,j' \in \{2,...,r'+1\}$ with $j\neq j'$. Then $\dim (H_{0,1} \cap H_{j,j'})=r'-3$, so Lemma~\ref{2:rel} implies that $\xi_{j,j'}:=[x_{0,1},x_{j,j'}] \in \Xi$. Clearly, the points $x_{e,f}$ with $e\in\{0,1\}$ and $f\in\{j,j'\}$  belong to $x_{0,1}^\perp \cap x_{b,c}^\perp\subseteq X(\xi_{j,j'})$.  It follows that $x_{0,1},q_j,q_{j'} \subseteq \xi_{j,j'} \cap \theta_H$, so by (S2) these points belong to a singular subspace. The fact that $\<x_{0,1},q_2,...,q_{r'+1}\>$  is an $r'$-space in $X$ then follows as it is a set of pairwise collinear points in $\theta_H$ which is, by construction, collinear to no point of $Y$. We conclude that the claim holds for each subspace through one of the $(r'-1)$-spaces $H_{i,j}$ with $(i,j)\in A$, and we can repeat this for each $(r'-1)$-space that arises as the intersection of any two such $r'$-spaces. Continuing like this, the claim follows. 

We define $\Omega$ as $\bigcup_{\theta\in \Theta} R'_\theta$. By the previous claim , $\Omega \subseteq F^* \cap X$, and as $R'_i=R'_{\theta_i}$, $R'_i  \subseteq \Omega$ for $i\in\{0,...,r'+1\}$,  so $\<\Omega\>= F^*$. We observe that points $x_1,x_2 \in \Omega$ are on an $X$-line $L$ if $\dim(Y_{x_1}\cap Y_{x_2})=r'-2$ and that $L \subseteq \Omega$: Indeed, if $\dim(Y_{x_1}\cap Y_{x_2})=r'-2$, then there is a member $\theta\in \Theta$ containing $x_1,x_2$ by Lemma~\ref{2:rel}, and hence $x_1,x_2 \in L \subseteq R'_\theta\subseteq \Omega$. In particular,  it follows that $\Omega$ is a subspace with respect to $X$-lines. 

$(i)$ Let $x\in X$ be arbitrary and take any $\theta \in \Theta$ through $x$. Then $\theta$ contains an $X$-space  $R'_\theta$ of dimension $r'$ in $\Omega$ by definition of the latter. Clearly, $R'_\theta$ contains a unique point $x'$ with $Y_x=Y_{x'}$. As such, $x \in \<x',Y_x\>$ indeed and $\<x',Y_x\>\cap \Omega=\{x'\}$. As $x\in X$ was arbitrary, we get $\bigsqcup_{x\in \Omega} \<x,Y_x\>\setminus Y_x =X$ indeed. Consequently,  $\<F^*,Y\>=\mathbb{P}^N(\K)$. 

$(ii)$, $(iii)$ Same as in the proof of Proposition~\ref{1':FX}.

$(iv)$ Finally, suppose  $r'=4$ (in which case $\Omega$ is isomorphic to  $\mathcal{G}_{6,2}(\K)$ by Proposition~\ref{noproj} and hence $\dim F^*=14$)  and suppose for a contradiction that $F^* \cap Y$ contains a point $y$. Set $\Omega':=\{x\in \Omega \mid y \in Y_x\}$. Then $\Omega'$ is isomorphic to $\mathcal{G}_{5,2}(\K)$ (cf.\ Lemma~\ref{2:Ystructuur}); in particular $\dim \<\Omega'\>=9$. We first claim that $\<\Omega'\>\cap X=\Omega'$. Indeed, if $x\in \<\Omega'\>\setminus \Omega'$, then it is a property of $\mathcal{G}_{5,2}(\K)$ that $x$ lies on a line $\<x_1,x_2\>$ with $x_1,x_2$ points of $\Omega'$ not on an $X$-line. By the above, $\dim (Y_{x_1} \cap Y_{x_2})=r'-3$ and hence, Lemma~\ref{2:rel} implies that $[x_1,x_2]\in \Xi$. But then $x\in \<x_1,x_2\>\setminus\{x_1,x_2\}$ contradicts the structure of $X(\xi)$. The claim follows. In particular, $\dim(\<y,\Omega'\>=10$, for $y\in \<\Omega'\>$ yields a point $x\in X$ (on a line $\<x',y\>$ with $x'\in\Omega'$) in $\<\Omega'\>\setminus\Omega'$, contradicting the previous claim. By Lemma~\ref{a52p}, the subspace $\<y,\Omega'\>$ contains a point $p$ of $\Omega \setminus\Omega'$. Let $p'$ be the unique point of $\<y,p\> \cap \<\Omega'\>$. Note that $p' \notin \Omega'$, for $\Omega$ contains no two points that lie on a singular line with a point in $Y$ by $(i)$. As in the previous claim, $p' \in [x_1,x_2]\in \Xi$ for two non-collinear points $x_1,x_2\in\Omega'$. But then $p\in[x_1,x_2]$ is a contradiction to the prescribed structure of $X([x_1,x_2])$.
\end{proof}

Putting everything together, we obtain the following rather general classification result. 

\begin{theorem}\label{2:projuni}
Let $(X,Z,\Xi,\Theta)$ be a duo-symplectic pre-DSV with parameters $(2,v,r',v')$ containing an $X$-line through which there is precisely one member of $\Theta$. Then:
\begin{compactenum}[$(i)$]
\item Up to projection from a $v'$-space $V \subseteq Y$ collinear to all points of $X$, we obtain that $X$ is the point set of a mutant of the dual line Grassmannian $\mathcal{DG}_{r'+2,2}(\K)$ with $Z$ as subspace at infinity and $\Xi\cup \Theta$ as its symps;
\item if additionally $(X,Z,\Xi,\Theta)$ satisfies (S3), then $r'=4$ and $X$ is projectively unique.
\end{compactenum}
\end{theorem}
\begin{proof}
$(i)$ Take any $x\in X$. By Lemma~\ref{1:v'=-1}, $Y_x$ contains a $v'$-space $V$ collinear to all points of $X_x$, and by Lemma~\ref{projV}, $V$ is collinear to all points of $X$ (note that $V$ is the common vertex of all members of $\Theta$). Lemma~\ref{resY} allows us to project from $V$, so henceforth we assume $v'=-1$.

 By Proposition~\ref{2:FX}, $X$ contains a legal projection $\Omega$ of $\mathcal{G}_{r'+2,2}(\K)$, and $X=\bigsqcup_{x\in \Omega} \<x,Y_x\>\setminus Y_x$. Therefore it suffices to show that the map $\chi:\Omega \rightarrow Y: x \mapsto Y_x$ satisfies the properties mentioned in the definition of the dual Line Grassmannians varieties (cf.\ Subsection~\ref{HDS}). 

Let $\mathbb{P}$ be a projective space of dimension $r'+1$, such that $\Omega$ is the image of $\mathbb{P}$ under the Pl\"ucker map $\mathsf{pl}$ (cf.\ Subsection~\ref{DLG}). Consider the map $\chi':\mathbb{P}\rightarrow Y$, taking a line $L\in \mathbb{P}$ to $\chi(\mathsf{pl}^{-1}(L))=Y_{\mathsf{pl}^{-1}(x)}$. 

\textit{Claim: $\chi'$ induces a linear duality between $\mathbb{P}$ and $Y$.}\\ Let $p$ be any point of $\mathbb{P}$. Let $x_L$ denote the point $\mathsf{pl}^{-1}(L)$ of $\Omega$. Then $\{ x_L \mid L \text{ line of } \mathbb{P} \text{ with } p\in L\}$ gives the points of a singular $r'$-space $R'_p$ of $\Omega$, and by Lemma~\ref{2:FX}$(iii)$, there is a unique member $\theta_p \in \Theta$ containing $R'_p$. Clearly, the set of $(r'-1)$-spaces in $Y(\theta_p)$ is precisely the set $\{Y_x \mid x\in R'_p\}$, so $\chi'(p)=Y(\theta_p)$. Since $\mathsf{pl}$ gives a bijection between the points of $\mathbb{P}$ and the $r'$-spaces of $\Omega$, and by Lemma~\ref{2:FX}$(iii)$, $\chi'$ is a bijection between the point set of $\mathbb{P}$ and the hyperplanes  of $Y$  (i.e., the point set of the dual of $Y$). 
Next, let $L$ be a line of $\mathbb{P}$ and recall that $\chi'(L)=Y_{x_L}$. Take an auxiliary point $r \in \mathbb{P}\setminus L$. Let $I$ be an index set such that $\{p_i\mid i\in I\}$ gives the set of points on $L$. For each $i\in I$, let $L_i$ denote the line $\<p_i,r\>$. For ease of notation, put $x_i:=x_{L_i}$.  Note that the image $\chi'(p_i)=Y(\theta_{p_i})$ is also given by $\<Y_{x_L},Y_{x_i}\>$ and that $Y_{x_i}$ contains the $(r'-2)$-space $Y_x \cap Y(\theta_r)$. In particular, $Y(\theta_{p_i})$ contains $Y_{x_L}$ for each $i\in I$. This already shows that $\chi'$ is a collineation between $\mathbb{P}$ and the dual of $Y$. To see that $\chi'$ is linear, we consider $\theta_r$. Indeed, in $\theta_r$, the map $x_i \mapsto Y_{x_i}$ is given by the collinearity relation in the quadric $XY(\theta_r)$, and restricted to $R'_r$, this is a linear collineation between $R'_r$ and the dual of $Y(\theta_r)$. Since $\chi'(p_i)=\<Y_{x_L},Y_{x_i}\>$, the claim follows.

 For each pair of non-collinear points  $p_1,p_2$ of $X$, (S1) and Lemma~\ref{convexclosure} imply that the unique symp $\zeta$  through $p_1,p_2$ (defined as their convex closure inside $X$) has $\zeta \cap X=X([p_1,p_2])$. The assertion follows.

$(ii)$ Recall that we assume that $F \subseteq F^*$ is complementary to $Y$ in $\mathbb{P}^N(\K)$.  By Proposition~\ref{2:projgras}, $\rho(X)\subseteq F$ is an injective projection of a Line Grassmannian $\mathcal{G}_{r'+2,2}(\K)$. Let $x\in X$ be arbitrary. We denote by  $T^{F}_{\rho(x)}$  the set of $\rho(X)$-lines in ${F}$ through $\rho(x)$ and  by $T^{F}_{\rho(x)}(\xi)$ the tangent space to $\rho(X(\xi))$ at $\rho(x)$ for some $\xi\in \Xi$ with $\rho(x)\in\rho(X(\xi))$. 

Axiom (S3) yields members $\xi_1,\xi_2 \in \Xi$ through $x$ such that $T_x$ is generated by $T_x(\xi_1)$ and $T_x(\xi_2)$.  Since for $i=1,2$, $T_x(\xi_i)=\<Y(\xi_i),T^{F}_{\rho(x)}(\xi_i)\>$, we obtain that $T_x=\<T_x(\xi_1),T_x(\xi_2)\>$ is equivalent with $Y_x=\<Y(\xi_1),Y(\xi_2)\>$ and  $T^{F}_{\rho(x)}=\<T^{F}_{\rho(x)}(\xi_1),T^{F}_{\rho(x)}(\xi_2)\>$. On the other hand, $\dim T^{F}_{\rho(x)}=2r'-1$ as $\Res_X(x)$ is isomorphic to $\mathcal{S}_{r'-1,1}(\K)$. Furthermore, since $r=2$,  $T^{F}_{\rho(x)}(\xi_1)$ and $T^{F}_{\rho(x)}(\xi_2)$ are just $4$-spaces, which generate at most an $8$-space in $F$, and so $2r'-1\leq 8$.  Recalling that  $r' > r \geq 2$ by assumption, we deduce that $r'\in\{3,4\}$. Since $v=r'-3$ and $\dim Y_x=r'-1$,  the requirement $Y_x=\<Y(\xi_1),Y(\xi_2)\>$ implies that $r'=4$ (and that $\xi_1$ and $\xi_2$ have disjoint vertices).

Since $r'=4$,  the variety $\mathcal{G}_{r'+2,2}(\K)$ does not admit legal projections (cf.\ Proposition~\ref{noproj}) and $F^* \cap Y = \emptyset$ by Proposition~\ref{2:FX}$(iv)$, i.e., $F=F^*$.
The following are projectively unique: $Y$ and $F$  in $\mathbb{P}^N(\K)$,   $\Omega$ in $F$. Moreover, the projectivity $\chi'$ between $\mathbb{P}$ and the dual of $Y$ is unique up to a projectivity of $Y$. We conclude that $X$ is projectively unique. 
\end{proof}

\subsection{Case 2: Duo-symplectic DSVs containing no $1$-lines}\label{cas2}

Let $(X,Z,\Xi,\Theta)$ be a duo-symplectic DSV with parameters $(r,v,r',v')$,  containing no $Y$-points collinear to all points of $X$, and containing no $1$-lines.

Note that we assume that (S3) holds.

\begin{prop}\label{r>2kanni}
There are no duo-symplectic dual split Veronese sets $(X,Z,\Xi,\Theta)$ with parameters $(r,v,r',v')$, where $r>1$, without $1$-lines.
\end{prop}

\begin{proof}
Suppose for a contradiction that such a DSV $(X,Z,\Xi,\Theta)$ exists. Recall that, according to  the case distinction in the beginning of this section, there were to cases: Case 2(a), which is already excluded by Lemma~\ref{reshalf}; and Case 2(b), in which $\Res_X(x)$ (for any $x\in X$) is a duo-symplectic pre-DSV $(X_x,Z_x,\Xi_x,\Theta_x)$ with parameters $(2,v,r'-1,v')$, containing no $Y_x$-points collinear to all points of $X_x$ and such that each $X_x$-line is contained in a unique member of $\Theta_x$. It  follows from Proposition~\ref{2:projuni}$(i)$ that $X_x$ is the point-set of  a mutant of the dual line Grassmannian $\mathcal{DG}_{r'+1,2}(\K)$. 

Let $x\in X$ be arbitrary. Axiom (S3) yields two members $\xi_1,\xi_2$ of $\Xi_x$ such that $X_x$ is generated by $\xi_1$ and $\xi_2$. Noting that $v=r'-4$, we obtain $\dim \xi_i=r'+2$, $i=1,2$, and hence $\dim\<X_x\> \leq 2r'+5$. Moreover, $\dim(Y_x)=r'$ and projected from $Y_x$, we get an injective projection $\Omega$ of $\mathcal{G}_{r'+1,2}(\K)$ (cf.\ Proposition~\ref{2:projgras}), and we put $B:=\dim \<\Omega\>$. So $\dim\<X_x\>=B+r'+1$. Combined, this yields $B \leq r' +4$. Since $B \geq 2r'-2$ (an injective projection of $\mathcal{G}_{r'+1,2}(\K)$ contains two $(r'-1)$-spaces intersecting each other in precisely a point), we obtain $r' \leq 6$. Recall that $r' > r=3$, so $r' \in \{4,5,6\}$. 
As the variety $\mathcal{G}_{6,2}(\K)$ is contained in  $\mathcal{G}_{7,2}(\K)$ and does not admit legal projections (cf.\ Proposition~\ref{noproj}), we would have $B \geq 14$, contradicting $B \leq r'+4 \leq 10$. We conclude that $r'=4$. By the same token, $\mathcal{G}_{5,2}(\K)$ does not admit legal projections, hence $B=9$,  contradicting $B \leq r'+4=8$.
This shows the proposition.
%there is a pre-DSV $(X',Z',\Xi',\Theta')$ with parameters $(2,v,r'-1,v')$ isomorphic to a mutant of the dual line Grassmannian $\mathcal{DG}_{r'+1,2}(\K)$ is generated by two members of $\Xi$, i.e., there exist $\xi_1,\xi_2\in\Xi'$ such that $\<\xi_1,\xi_2\>=\<X'\>$. Observing that $v=r'-4$, $r=2$ and $\dim Y=r'$, we see that $\dim\xi_i=r'+2$, $i=1,2$, and so $\dim\<\xi_1,\xi_2\>\leq 2r'+5$. However, the vertices of $\xi_1$ and $\xi_2$ generate at most $Y$, hence we also have $\dim\<\xi_1,\xi_2\>\leq r'+12$. Let $B$ be the dimension of $\<\rho(X)\>$, i.e., $B$ is the dimension of a legal projection of a line Grassmannian variety. Then $\<X\>=B+r'+1$. Consequently the combination of these inequalities and equality yields $B\leq \min\{11,r'+4\}$. Since $\mathcal{G}_{6,2}(\K)$ lives in $14$-space, does not admit proper legal projections (see Proposition~\ref{noproj}), and is contained in $\mathsf{G}_{r',1}(\K)$, for $r'\geq 5$, we deduce $r'\leq 4$; hence $r'=4$ in view of $r'>r=3$. But if $r'=4$, then $B\leq 8$, contradicting the fact that $\mathcal{G}_{5,2}(\K)$ lives in $9$-space and  does not admit proper legal projections. 
\end{proof}

From Proposition~\ref{r>2kanni} and Lemma~\ref{2:uniSS} we conclude that, if $(X,Z,\Xi,\Theta)$ is a duo-symplectic DSV with parameters $(r,v,r',v')$, then it contains a $1$-line and hence $r=2$. Main Result~\ref{main}$(iii)$  follows from Proposition~\ref{2:projuni}.

\par\bigskip
As can be double-checked in the structure of the proof (cf.\ Section~\ref{sotp}), this finishes the proof of Main Result~\ref{main}.

\end{document}